\documentclass[10pt,reqno,oneside,a4paper]{amsart}

\usepackage{subfiles}
\usepackage[english]{babel}
\usepackage[utf8]{inputenc}
\usepackage{amsmath}
\usepackage{amsthm}
\usepackage{amssymb}
\usepackage{amsfonts}
\usepackage{graphicx}
\usepackage[shortlabels]{enumitem}

\usepackage[paper=a4paper]{geometry}

\usepackage{adjustbox}
\usepackage{listings}
\usepackage{color}
\usepackage{upgreek}
\usepackage{mathtools}
\usepackage{color}
\usepackage{empheq}
\usepackage{bm}
\usepackage{csquotes}
\usepackage{tikz}
\usetikzlibrary{matrix}
\usepackage{tikz-cd}
\usepackage{setspace}
\usepackage{commath}
\usepackage[norelsize]{algorithm2e}
\usepackage{adjustbox}
\usepackage{mathrsfs}
\usepackage{stmaryrd}
\usetikzlibrary{decorations.pathmorphing,shapes}
\usepackage{rotating}
\usepackage{pdflscape}
\usepackage{subcaption}
\usepackage{afterpage}
\usepackage[all,cmtip]{xy}
\usepackage{tabu}
\usepackage{stackengine}
\usepackage{caption}
\usetikzlibrary{arrows,calc}
\usepackage{hyperref}
\usepackage{scalerel}
\usepackage{centernot}
\usepackage{float}
\usepackage{tabularx}

\allowdisplaybreaks

\newcommand{\tikzAngleOfLine}{\tikz@AngleOfLine}
\def\tikz@AngleOfLine(#1)(#2)#3{%
\pgfmathanglebetweenpoints{%
\pgfpointanchor{#1}{center}}{%
\pgfpointanchor{#2}{center}}
\pgfmathsetmacro{#3}{\pgfmathresult}%
}

\newcommand{\gr}{\text{\normalfont gr-}}

\newcommand{\gap}{\hspace{1pt}}
\newcommand{\ZZ}{\mathbb{Z}}

\newcommand{\NN}{\mathbb{N}}

\renewcommand{\AA}{\mathbb{A}}
\newcommand{\DD}{\mathbb{D}}

\renewcommand{\mod}{\text{\normalfont mod-}}

\DeclareMathOperator{\GKdim}{GKdim}

\DeclareMathOperator{\im}{im}

\DeclareMathOperator{\gldim}{gl.\hspace{-1pt}dim}

\DeclareMathOperator{\idim}{i.\hspace{-1pt}dim}

\DeclareMathOperator{\End}{End}

\DeclareMathOperator{\Spec}{Spec}

\DeclareMathOperator{\Ext}{Ext}

\DeclareMathOperator{\hilb}{hilb}

\DeclareMathOperator{\GL}{GL}
\DeclareMathOperator{\SL}{SL}
\DeclareMathOperator{\Tr}{Tr}
\DeclareMathOperator{\tr}{tr}
\DeclareMathOperator{\hdet}{hdet}

\DeclareMathOperator{\Autgr}{Aut_{gr}}

\newcommand{\hash}{\hspace{1pt} \# \hspace{1pt}}

\numberwithin{equation}{section}

\theoremstyle{definition}

\newtheorem{defn}[equation]{Definition}

\newtheorem{example}[equation]{Example}
\newtheorem*{graphing*}{Graphing a Function by Plotting Points}
\newtheorem*{completing*}{Completing the Square}
\newtheorem*{propofpoly*}{Properties of Polynomial Functions}
\newtheorem*{expgrowth*}{The Exponential Growth Function}
\newtheorem*{firstprops*}{First Properties of Logarithms}
\newtheorem*{secondprops*}{Product, Quotient, and Power Properties of Logarithms}
\newtheorem*{changebase*}{Change of Base Theorem}
\newtheorem*{limitconstant*}{Limit of a Constant Function}
\newtheorem*{limitidentity*}{Limit of the Identity Function}
\newtheorem*{limitprops*}{Properties of Limits}
\newtheorem*{limitpoly*}{Limits of Polynomials}
\newtheorem*{limitrat*}{Limits of Rational Functions}

\theoremstyle{plain}
\newtheorem{thm}[equation]{Theorem}
\newtheorem{prop}[equation]{Proposition}
\newtheorem{lem}[equation]{Lemma}
\newtheorem{cor}[equation]{Corollary}

\newtheorem*{thm*}{Theorem}

\newtheorem*{ver*}{Vertical Line Test}

\theoremstyle{remark}
\newtheorem{rem}[equation]{Remark}

\newcommand\restr[2]{{
  \left.\kern-\nulldelimiterspace 
  #1 
  \right|_{#2} 
  }}

\newcounter{sarrow}

\newcounter{darrow}

\makeatletter
\renewcommand*\env@matrix[1][\arraystretch]{%
  \edef\arraystretch{#1}%
  \hskip -\arraycolsep
  \let\@ifnextchar\new@ifnextchar
  \array{*\c@MaxMatrixCols c}}
\makeatother

\newboolean{printBibInSubfiles}
\setboolean{printBibInSubfiles}{true}
\def\bib{\ifthenelse{\boolean{printBibInSubfiles}}
           { \bibliographystyle{amsalpha} \bibliography{thesisbib} }
       {}
}

\makeatletter
\tikzset{
  column sep/.code=\def\pgfmatrixcolumnsep{\pgf@matrix@xscale*(#1)},
  row sep/.code   =\def\pgfmatrixrowsep{\pgf@matrix@yscale*(#1)},
  matrix xscale/.code=%
    \pgfmathsetmacro\pgf@matrix@xscale{\pgf@matrix@xscale*(#1)},
  matrix yscale/.code=%
    \pgfmathsetmacro\pgf@matrix@yscale{\pgf@matrix@yscale*(#1)},
  matrix scale/.style={/tikz/matrix xscale={#1},/tikz/matrix yscale={#1}}}
\def\pgf@matrix@xscale{1}
\def\pgf@matrix@yscale{1}
\makeatother

\definecolor{mygray}{gray}{0.8}

\def\dotfill#1{\cleaders\hbox to #1{.}\hfill}
\makeatletter
\def\myrulefill{\leavevmode\leaders\hrule height .7ex width 1ex depth -0.6ex\hfill\kern\z@}
\makeatother

\title[Actions of Small Groups on Two-Dimensional AS Regular Algebras]{Actions of Small Groups on Two-Dimensional Artin-Schelter Regular Algebras}
\author{Simon Crawford}
\address{Department of Pure Mathematics, University of Waterloo, 200 University Ave W, Waterloo, ON N2L 3G1, Canada}
\email{simon.crawford@uwaterloo.ca}

\date{\today}
\subjclass[2010]{14J17, 16S35, 16W22.}

\begin{document}
\begin{abstract}
In commutative invariant theory, a classical result due to Auslander says that if $R = \Bbbk[x_1, \dots, x_n]$ and $G$ is a finite subgroup of $\Autgr(R) \cong \GL(n,\Bbbk)$ which contains no reflections, then there is a natural graded isomorphism $R \hash G \cong \End_{R^G}(R)$. In this paper, we show that a version of Auslander's Theorem holds if we replace $R$ by an Artin-Schelter regular algebra $A$ of global dimension 2, and $G$ by a finite subgroup of $\Autgr(A)$ which contains no quasi-reflections. This extends work of Chan--Kirkman--Walton--Zhang. As part of the proof, we classify all such pairs $(A,G)$, up to conjugation of $G$ by an element of $\Autgr(A)$. In all but one case, we also write down explicit presentations for the invariant rings $A^G$, and show that they are isomorphic to factors of AS regular algebras.
\end{abstract}
\maketitle

\section{Introduction}
\noindent Throughout let $\Bbbk$ be an algebraically closed field of characteristic 0. A classical theorem of Auslander is the following:

\begin{thm} \label{auslandersthm}
Let $R = \Bbbk[x_1, \dots, x_n]$ and let $G$ be a finite subgroup of $\Autgr(R) = \GL(n,\Bbbk)$. Consider the graded ring homomorphism
\begin{align*}
\phi : R \hash G \to \End_{R^G}(R), \quad \phi(rg)(s) = r (g \cdot s).
\end{align*}
Then $\phi$ is an isomorphism if and only if $G$ is small.
\end{thm}

\noindent By \emph{small}, we mean that $G$ contains no reflections, in the sense described in Section \ref{commutativecase}. For example, if $G$ is a finite subgroup of $\SL(2,\Bbbk)$ acting on $\Bbbk[u,v]$ (meaning that $\Bbbk[u,v]^G$ is a Kleinian singularity), then $G$ is small and therefore the map $\phi$ is an isomorphism. \\
\indent Recently, a number of authors have studied when noncommutative generalisations of this result hold. More specifically, one can replace the polynomial ring $R$ by an Artin-Schelter (AS) regular algebra $A$ and consider the action of a finite group $G$ of graded automorphisms on $A$, and ask whether the natural map
\begin{align*}
\phi : A \hash G \to \End_{A^G}(A), \quad \phi(ag)(b) = a (g \cdot b),
\end{align*}
is an isomorphism. If this is the case, then we say that \emph{the Auslander map is an isomorphism for the pair} $(A,G)$. More generally, one can replace $G$ by a finite dimensional semisimple Hopf algebra $H$, and ask the same question. In \cite{bhz}, the authors provided a useful criterion for determining when the Auslander map is an isomorphism, which allows one to check specific examples by hand. We state only the group-theoretic version of their result below; technical terminology is defined in Section \ref{prelimsec}.

\begin{thm}[{\cite[Theorem 0.3]{bhz}}] \label{introbhz}
Let $G$ be a finite group acting on an AS regular, GK-Cohen-Macaulay algebra $A$ with $\GKdim A \geqslant 2$. Let $\overline{g} = \sum_{g \in G} g$, viewed as an element of $A \hash G$. Then the Auslander map is an isomorphism for the pair $(A,G)$ if and only if 
\begin{align*}
\GKdim \big((A \hash G)/\langle \overline{g} \rangle \big) \leqslant \GKdim A - 2.
\end{align*}
\end{thm}

Using this, the Auslander map has been shown to be an isomorphism in the following cases:
\begin{itemize}[topsep=0pt,leftmargin=25pt]
\item Actions of small subgroups of $\text{Aut}_{\text{Lie}}(\mathfrak{g})$ on universal enveloping algebras of finite dimensional Lie algebras $U(\mathfrak{g})$, \cite[Theorem 0.4]{bhz2};
\item Actions by finite groups on noetherian graded down-up algebras, \cite[Theorem 0.6]{bhz2} and \cite[Theorem 4.3]{won};
\item Permutation actions on $\Bbbk_{-1}[x_1, \dots, x_n]$, \cite[Theorem 2.4]{won};
\item Actions of semisimple Hopf algebras on AS regular algebras of dimension 2, such that the action has trivial homological determinant, \cite[Theorem 4.1]{ckwz}.
\end{itemize}
The last of these examples can be viewed as a noncommutative generalisation of the fact that the Auslander map is an isomorphism for Kleinian singularities. Here, the ``trivial homological determinant'' condition serves as a noncommutative analogue of requiring that the finite subgroup of $\Autgr(\Bbbk[u,v]) = \GL(2,\Bbbk)$ lie inside $\SL(2,\Bbbk)$. In fact, it is conjectured that the Auslander map is an isomorphism whenever the action has trivial homological determinant; the result of Chan--Kirkman--Walton--Zhang establishes this in the dimension 2 case. \\
\indent The notion of smallness has a noncommutative generalisation, and so it is natural to ask whether a version of Theorem \ref{auslandersthm} holds in the noncommutative setting. 
One of the main results of this paper is to show that one direction of this result generalises to the noncommutative setting in dimension 2:

\begin{thm}[Theorem \ref{diagausthm}, Theorem \ref{nondiagausthm}] \label{introauslanderthm}
Suppose that $A$ is a two-dimensional AS regular algebra and that $G$ is a small subgroup of $\Autgr(A)$. Then the Auslander map is an isomorphism for the pair $(A,G)$.
\end{thm}

If a finite group $G$ acts on a two dimensional AS regular algebra with trivial homological determinant, then one can check that $G$ is small (in fact, this follows from Theorem \ref{introthmclassify} below), and so the above theorem can be viewed as an extension of \cite[Theorem 4.1]{ckwz}. \\
\indent Our method for proving Theorem \ref{introauslanderthm} is to first classify all possible actions of small groups on two-dimensional AS regular algebras and to then verify that the criterion of Theorem \ref{introbhz} is met in each case. The following is the main result of Section \ref{smallactionssec}, where $\Bbbk_q[u,v]$ and $\Bbbk_J[u,v]$ are the quantum plane and Jordan plane, respectively, whose definitions can be found in Section \ref{definitions}. We also write $\omega_n$ for a primitive $n$th root of unity.

\begin{thm}[Theorem \ref{smallclassification}] \label{introthmclassify}
Suppose that $A$ is a two-dimensional AS regular algebra which is not commutative and that $G$ is a small subgroup of $\Autgr(A)$. Then, up to conjugation of $G$ by an element of $\Autgr(A)$, the possible pairs $(A,G)$ are as follows:
\begin{figure}[h]
{\tabulinesep=1.5pt
\begin{tabu}{c|c|c|c|l}
\normalfont{Case} & $A$ & $G$ & \normalfont{Generators} & \normalfont{Conditions} \\ \hline
\normalfont{(i)} &  $\Bbbk_q[u,v]$ & $\frac{1}{n}(1,a)$ & \hspace{-3pt} $\begin{pmatrix} \omega_n & 0 \\ 0 & \omega_n^a \end{pmatrix} \hspace{-3pt} $ & $q \neq 1$, $1 \leqslant a < n$, $\gcd(a,n) = 1$. \\ \hline
\normalfont{(ii)} &  $\Bbbk_{-1}[u,v]$ & $G_{n,k}$ & $ \hspace{-3pt} \begin{pmatrix} \omega_n & 0 \\ 0 & \omega_n^{-1} \end{pmatrix}$, $\begin{pmatrix} 0 & \omega_{2k} \\ \omega_{2k} & 0 \end{pmatrix} \hspace{-3pt}$ & $k \not\equiv 2 \text{\normalfont{ mod }} 4$, $\gcd(n,k) = 1$. \\ \hline
\normalfont{(iii)} &  $\Bbbk_J[u,v]$ & $\frac{1}{n}(1,1)$ & $ \hspace{-3pt} \begin{pmatrix} \omega_n & 0 \\ 0 & \omega_n \end{pmatrix} \hspace{-3pt} $ & $n \geqslant 2$.  \\
\end{tabu}}
\end{figure}
\end{thm}

We remark that if $A$ is a commutative AS regular algebra of dimension $n$ then it is isomorphic to $\Bbbk[x_1, \dots, x_n]$. In this case, the classification of small subgroups when $n=2$ (i.e. of small subgroups of $\GL(2,\Bbbk)$) can be found in \cite{brie}. \\
\indent Establishing the reverse implication in Theorem \ref{introauslanderthm} appears to be more difficult, since if one wishes to utilise Theorem \ref{introbhz}, it is necessary to show that $\GKdim \big((A \hash G)/\langle \overline{g}\rangle \big) \geqslant 1$ whenever $G \leqslant \Autgr(A)$ is not small. In general, this is more difficult to show than showing that an algebra has GK dimension 0 (i.e. that it is finite dimensional). However, the reverse implication appears to hold based off the examples we have computed. \\
\indent The final part of this paper is devoted to writing down presentations for the invariant rings $A^G$, where $A$ and $G$ appear in the classification of Theorem \ref{introthmclassify}. There is a long-standing belief that invariant rings arising from group actions (or, more generally, Hopf actions) on AS regular algebras can be written as factors of AS regular algebras. Assuming this is true, this gives a strong motivation for writing down a presentation for $A^G$, as from this one might be able to construct new examples of AS regular algebras. For cases (i) and (ii), we are able to write the invariant rings as factors of AS regular algebras:

\begin{thm}[Corollary \ref{jordancor}, Corollary \ref{kqcor}]
Let $(A,G)$ be a pair from case \emph{(i)} or case \emph{(iii)} of Theorem \ref{introthmclassify}. Then $A^G$ is isomorphic to a factor of an AS regular algebra.
\end{thm}
For case (i), this is not a particularly surprising result, and the AS regular algebra of which $A^G$ is a factor is a quantum polynomial ring. Case (iii) is more interesting: the AS algebra of which $A^G$ is a factor is a quantisation of a Poisson structure on a commutative polynomial ring. These rings were first studied in \cite{lecoutre}. \\
\indent For the remaining case, case (ii) of Theorem \ref{introthmclassify}, even writing down a set of generators for the invariant ring $A^G$ is a nontrivial task. When $n$ or $k$ is even, $A^G$ is commutative, and turns out to be a ring which is already well understood. If instead $n$ and $k$ are both odd, then $A^G$ is not commutative, and we are only able to write down a set of generators.

\begin{thm}
Suppose that $A = \Bbbk_{-1}[u,v]$ and $G = G_{n,k}$, where $n$ and $k$ are coprime and $k \not\equiv 2 \text{\normalfont{ mod }} 4$.
\begin{enumerate}[{\normalfont (1)},leftmargin=*,topsep=0pt,itemsep=2pt]
\item \emph{(Proposition \ref{commvsnoncomm})} The invariant ring $A^G$ is commutative if and only if $n$ or $k$ is even.
\item \emph{(Propositions \ref{nequals1}, \ref{nequals2}, \ref{noddiso}, and \ref{neveniso})} If $A^G$ is commutative, then it is a surface quotient singularity of type $\mathbb{A}$ or type $\mathbb{D}$.
\item \emph{(Theorems \ref{ngtkgens} and \ref{nltkgens})} If $A^G$ is not commutative, then there is a formula to write down its generators which depends on the Hirzebruch-Jung continued fraction expansion of $\frac{n}{\frac{1}{2}(n+k)}$.
\end{enumerate}
\end{thm}

In the setting of the above theorem, it is difficult to write down an explicit presentation for the invariant rings $A^G$ when they are not commutative. We anticipate that they can always be written as a factor of an AS regular algebra, and show that this is the case for a specific example. Writing down presentations for arbitrary $n$ and $k$, and studying further properties of these rings in general, is the topic of work in progress. \\
\indent We now explain how our results fit into a broader picture. Let $R = \Bbbk[u,v]$ and let $G$ be a small subgroup of $\GL(2,\Bbbk)$. In this case, $\Spec R^G$ is a (generically non-Gorenstein) surface quotient singularity. There are strong connections between properties of the rings $R^G$, representation theory of the group $G$, and geometric properties of $\Spec R^G$. The results in this paper can be viewed as a first step towards better understanding a noncommutative version of this setup: by Corollary \ref{gorensteincor}, the invariant rings $A^G$ under consideration are generically non-Gorenstein (they are Gorenstein precisely when $G$ acts with trivial homological determinant). Moreover, they are graded isolated singularities in the sense of \cite{ample}. It is natural to ask how far the analogy can be pushed: for example, how much of the Auslander--Reiten theory for commutative surface quotient singularities holds in the noncommutative setting? Can one identify analogues of ``exceptional curves'' in a noncommutative resolution of $A^G$? Answering versions of these questions is the subject of work in progress. \\

\noindent \textbf{Organisation of the paper.} This paper is organised as follows. In Section \ref{prelimsec}, we fix the notation used throughout this paper, and recall some basic definitions and results. In Section \ref{smallactionssec} we prove Theorem \ref{introthmclassify}, and then in Section \ref{auslandersec} we use this to prove Theorem \ref{introauslanderthm}. The remainder of the paper is dedicated to determining the invariant rings $A^G$: we treat case (iii) from Theorem \ref{introthmclassify} in Section \ref{jordanplaneinvsec}, case (i) in Section \ref{kqinvsec}, and case (ii) in Sections \ref{comminvsec} and \ref{noncomminvsec}. 


\section{Preliminaries} \label{prelimsec}

\subsection{Conventions}
Throughout $\Bbbk$ will denote an algebraically closed field of characteristic 0. Let $R$ be an $\NN$-graded ring. We write $\gr R$ (respectively, $R\text{-gr}$) for the category of finitely generated $\ZZ$-graded right (respectively, left) $R$-modules with degree-preserving morphisms. Given $M \in \gr R$, we define $M[i]$ to be the graded module which is isomorphic to $M$ as an ungraded module, but which satisfies $M[i]_n = M_{i+n}$. If $M, N \in \gr R$ then $\End_R(M,N)$ has a natural grading given by $\bigoplus_{i \in \ZZ} \End_{\gr R}(M,N[i])$ as graded vector spaces. In this paper, we shall use right modules unless otherwise stated. We write $\idim M$ for the injective dimension of $M \in \mod R$, and $\gldim R$ for the global dimension of $R$. We write $\omega_n$ for a primitive $n$th root of unity.

\subsection{Actions of Small Groups on Polynomial Rings} \label{commutativecase}
We begin by recalling some of the definitions and basic results from commutative algebraic geometry which we seek to generalise to a noncommutative setting. Suppose that $R \coloneqq \Bbbk[x_1, \dots, x_n]$ and that $g \in \Autgr(R) = \GL(n,\Bbbk)$. Then $g$ is said to be a \emph{reflection} if the fixed subspace $\{ v \in \Bbbk^n \mid gv = v \}$ has dimension $n-1$. If $G$ is a finite subgroup of $\Autgr(R)$, then it is said to be \emph{small} if it contains no reflections. As stated in the introduction, it is a classical result due to Auslander that the natural map
\begin{align*}
\phi : R \hash G \to \End_{R^G}(R), \quad \phi(rg)(s) = r(g \cdot s)
\end{align*}
is an isomorphism if and only if $G$ is small. \\
\indent The classification of small groups is easy to write down in the two dimensional case. These were classified by Brieskorn in \cite{brie}, but we follow the notation of \cite{riemen}, which lends itself more naturally to the study of the invariant rings. An abridged version of the classification is as follows:


\begin{figure}[h]
{\tabulinesep=1.5pt
\begin{tabu}{c|c|c|l}
Type & Group & Generators & Conditions \\ \hline
$\mathbb{A}$ & $\frac{1}{n}(1,a)$ & $\begin{pmatrix} \omega_n & 0 \\ 0 & \omega_n^a \end{pmatrix}$ & $1 \leqslant a < n$, $\gcd(a,n) = 1$. \\ \hline
$\mathbb{D}$ & $\mathbb{D}_{m,q}$ & $\begin{pmatrix} \omega_{2q} & 0 \\ 0 & \omega_{2q}^{-1} \end{pmatrix}$, $\begin{pmatrix} 0 & \omega_{4(m-q)} \\ \omega_{4(m-q)} & 0 \end{pmatrix}$ & $1 < q < m$, $\gcd(m,q) = 1$.
\end{tabu}}
\end{figure}

\noindent We remark that the generators we have chosen differ from those that appear elsewhere in the literature, but it is straightforward to show that these yield the same groups. In both cases, the ``type'' refers to the dual graph of the exceptional divisor in the minimal resolution of $\Spec R^G$. There are three other infinite families that we have omitted since they will play no role in this paper, called types $\mathbb{T}$, $\mathbb{O}$, and $\mathbb{I}$. \\
\indent Given a non-negative rational number $x$, its \emph{Hirzebruch-Jung continued fraction expansion} is the unique representation of $x$ in the form
\begin{align*}
x = a_1 - \frac{1}{\displaystyle{a_2 - \frac{1}{\displaystyle{\ddots - \frac{1}{a_n}}}}} \eqqcolon [a_1, a_2, \dots, a_n]
\end{align*}
where the $a_i$ are all integers and $a_1 \geqslant 1$ and $a_2, \dots, a_n \geqslant 2$. Somewhat surprisingly, ring-theoretic, representation-theoretic and geometric properties of the invariant rings $R^G$ are controlled by the Hirzebruch-Jung continued fraction expansion of certain rational numbers which depend on $n$ and $a$ for type $\AA$, or $m$ and $q$ for type $\DD$. We briefly summarise some of the results that will be needed later in this paper regarding the invariant rings $R^G$ in these two cases. 
\subsubsection{Type $\AA$} \label{typeApres}
Fix integers $a$ and $n$ satisfying $1 \leqslant a < n$ and $\gcd(a,n) = 1$, and set $G = \frac{1}{n}(1,a)$. Write $[\alpha_1, \dots, \alpha_N]$ for the Hirzebruch-Jung continued fraction expansion of $\frac{n}{a}$. By \cite[Satz 8]{riemen}, the dual graph of the minimal resolution of of $\Spec R^G$ is as follows:

\begin{figure}[h]
\centering
\begin{tikzpicture}[-,thick,scale=1]

\node (1) at (0,0) {$\bullet$};
\node (2) at (1.5,0) {$\bullet$};
\node (3) at (3,0) {$\phantom{3}$};
\node (4) at (3.5,0) {$\cdots$};
\node (5) at (4,0) {$\phantom{2}$};
\node (6) at (5.5,0) {$\bullet$};
\node (7) at (7,0) {$\bullet$};

\node at (0,0.4) {$-\alpha_1$};
\node at (1.5,0.4) {$-\alpha_2$};
\node at (5.5,0.4) {$-\alpha_{N-1}$};
\node at (7,0.4) {$-\alpha_N$};

\draw (1) to (2);
\draw (2) to (3);
\draw (5) to (6);
\draw (6) to (7);

\end{tikzpicture}
\end{figure}
\noindent where each vertex corresponds to a curve isomorphic to $\mathbb{P}^1$, and the label $-\alpha_i$ give the self-intersection numbers of the corresponding curve. 
In the case when $a=n-1$ (so that $G \leqslant \SL(2,\Bbbk)$), $\Spec R^G$ is a Kleinian singularity, and the Hirzebruch-Jung continued fraction expansion of $\frac{n}{n-1}$ is $[2, 2, \dots, 2]$, where there are $n-1$ repetitions of $2$. Therefore this recovers the well-known result that the dual graph of the minimal resolution of a type $\mathbb{A}$ Kleinian singularity consists of a chain of $\mathbb{P}^1$, each having self-intersection $-2$. \\
\indent One can also determine generators for the invariant ring $R^G$ using Hirzebruch-Jung continued fractions. Write 
\begin{align*}
\frac{n}{n-a} = [\beta_1, \dots, \beta_{d-2}].
\end{align*}
Now define two series of integers $i_1, \dots, i_d$ and $j_1, \dots, j_d$ as follows:
\begin{gather*}
\begin{array}{llllll}
i_1 = n, & i_2 = n-a & & \text{and} & & i_k = \beta_{k-2} i_{k-1} - i_{k-2} \text{ for } 3 \leqslant k \leqslant d, \\[2pt]
j_1 = 0, & j_2 = 1 & & \text{and} & & j_k = \beta_{k-2} j_{k-1} - j_{k-2} \text{ for } 3 \leqslant k \leqslant d.
\end{array}
\end{gather*}
\noindent Then, by \cite[Satz 1]{riemen}, the invariant ring $R^G$ is minimally generated by the $d$ elements
\begin{align*}
x_k \coloneqq u^{i_k} v^{j_k}, \quad 1 \leqslant k \leqslant d.
\end{align*}
One can also use the $\beta_k$ to write down a minimal set of relations between the $x_i$ \cite[Satz 8]{riemen}:
\begin{gather*}
\begin{array}{lcl}
x_{k-1} x_{k+1} = x_k^{\beta_{k-1}} & & \text{for } 2 \leqslant k \leqslant d-1, \\[6pt]
x_k x_\ell = x_{k+1}^{\beta_{k} - 1} x_{k+2}^{\beta_{k+1} - 2} \dots x_{\ell-2}^{\beta_{\ell-3} - 2}  x_{\ell-1}^{\beta_{\ell-2} - 1} & & \text{for }  2 \leqslant k+1 < \ell-1 \leqslant d-1.
\end{array}
\end{gather*}
We draw attention to the fact that, on the right hand side of the relation on the second line, the first and last exponents have the form $\beta_m-1$, while the remaining exponents have the form $\beta_m-2$.

\subsubsection{Type $\DD$} Now fix integers $m$ and $q$ satisfying $1 < q < m$ and $\gcd(m,q) = 1$, and set $G = \mathbb{D}_{m,q}$. Write $[\alpha_1, \dots, \alpha_N]$ for the Hirzebruch-Jung continued fraction expansion of $\frac{m}{q}$. By \cite[Satz 8]{riemen}, the dual graph of the minimal resolution of of $\Spec R^G$ is as follows:

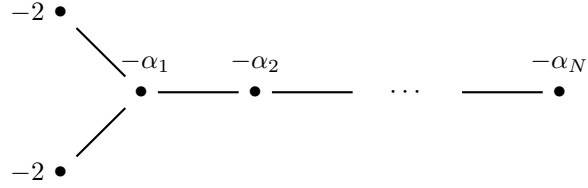
\begin{figure}[H]
\centering
\begin{tikzpicture}[-,thick,scale=1]

\node (-1) at (-1.061,1.061) {$\bullet$};
\node (0) at (-1.061,-1.061) {$\bullet$};
\node (1) at (0,0) {$\bullet$};
\node (2) at (1.5,0) {$\bullet$};
\node (3) at (3,0) {$\phantom{3}$};
\node (4) at (3.5,0) {$\cdots$};
\node (5) at (4,0) {$\phantom{2}$};
\node (6) at (5.5,0) {$\bullet$};

\node at (0.05,0.4) {$-\alpha_1$};
\node at (1.5,0.4) {$-\alpha_2$};
\node at (5.5,0.4) {$-\alpha_{N}$};
\node at (-1.5,1.061) {$-2$};
\node at (-1.5,-1.061) {$-2$};

\draw (-1) to (1);
\draw (0) to (1);
\draw (1) to (2);
\draw (2) to (3);
\draw (5) to (6);

\end{tikzpicture}
\end{figure}
\noindent where we use the same notation as in the type $\mathbb{A}$ case. If $q = m-1$ then $R^G$ is a type $\mathbb{D}$ Kleinian singularity, and $\alpha_i = 2$ for all $i$. \\
\indent As in the type $\AA$ case, one is able to write down a minimal set of generators for $R^G$ using Hirzebruch-Jung continued fractions. Write
\begin{align*}
\frac{m}{m-q} = [\beta_1, \dots, \beta_{d-2}],
\end{align*}
and define three series of integers $r_1, \dots, r_{d-1}$, $s_1, \dots, s_{d-1}$, and $t_1, \dots, t_{d-1}$  as follows:
\begin{gather*}
\begin{array}{lll | llll}
s_1 = 1, & s_2 = 1 & & & s_3 = \beta_2, & & s_k = \beta_{k-1} s_{k-1} - s_{k-2} \text{ for } 4 \leqslant k \leqslant d-1, \\[2pt]
t_1 = \beta_1, & t_2 = \beta_1 - 1 & & & t_3 = \beta_2(\beta_1 - 1) -1, & & t_k = \beta_{k-1} t_{k-1} - t_{k-2} \text{ for } 4 \leqslant k \leqslant d-1,
\end{array} \\
r_k = (m-q) t_k - q s_k \text{ for } 1 \leqslant k \leqslant d-1,
\end{gather*}
\noindent where the entries to the right of the vertical line only exist when $d > 3$, which happens if and only if $q < m-1$, if and only if $\DD_{m,q}$ is not a subgroup of $\SL(2,\Bbbk)$. By \cite[Satz 2]{riemen}, the invariant ring $R^G$ is minimally generated by the $d$ elements
\begin{align*}
x_k \coloneqq (u^{2 q s_k} + (-1)^{t_k} v^{2 q s_k}) (uv)^{r_k}, \quad 1 \leqslant k \leqslant d-1, \quad \text{and} \quad x_d \coloneqq (uv)^{2(n-q)}.
\end{align*}
As in the type $\AA$ case, a minimal set of relations can be written down using the $\beta_k$; see \cite[Satz 8]{riemen}.

\subsection{Definitions and Basic Results} \label{definitions}
Suppose that $A$ is a $\Bbbk$-algebra. We say that $A$ is \emph{connected graded} if there is a direct sum decomposition $A = \bigoplus_{i \geqslant 0} A_i$ such that $A_0 = \Bbbk$ and $A_i \cdot A_j \subseteq A_{i+j}$. If, moreover, $A$ is finitely generated as a $\Bbbk$-algebra, then it is called \emph{finitely graded}. If this is the case, then $\dim_\Bbbk A_i < \infty$ for all $i$, and if $M \in \gr A$ then $\dim_\Bbbk M_i < \infty$ for all $i \in \mathbb{Z}$. We then define the \emph{Hilbert series} of $M$ (which of course allows $M = A_A$) to be the formal Laurent series
\begin{align*}
\hilb M \coloneqq \sum_{i \in \ZZ} (\dim_\Bbbk M_i) \gap t^i.
\end{align*}
Given an $\NN$-graded algebra $A$, its \emph{$n$th Veronese} is the subring $A^{(n)} \coloneqq \bigoplus_{i \geqslant 0} A_{ni}$. \\
\indent The algebras of interest in this paper are particular examples of finitely graded $\Bbbk$-algebras which have additional properties. These algebras are defined as follows.

\begin{defn} \label{ASregdef}
Let $A$ be a finitely graded $\Bbbk$-algebra, and also write $\Bbbk = A/A_{\geqslant 1}$ for the trivial module. We say that $A$ is \emph{Artin-Schelter Gorenstein} (or AS Gorenstein) \emph{of dimension $d$} if:
\begin{enumerate}[{\normalfont (1)},topsep=1pt,itemsep=1pt,leftmargin=35pt]
\item $\idim A_A = \idim {}_A A = d < \infty$, and
\item $\Ext^i_{\gr A}(\Bbbk_A,A_A) \cong \left \{
\begin{array}{cl}
0 & \text{if } i \neq d \\
{}_A \Bbbk[\ell] & \text{if } i = d
\end{array}
\right . $ \hspace{4pt} as left $A$-modules, for some integer $\ell$.
\end{enumerate}
We call $\ell$ the \emph{AS index} of $A$. If moreover
\begin{enumerate}[{\normalfont (1)},topsep=1pt,itemsep=1pt,leftmargin=35pt]
\item[(3)] $\gldim A = d$, and
\item[(4)] $A$ has finite GK dimension,
\end{enumerate}
then we say that $A$ is \emph{Artin-Schelter regular} (or AS regular) \emph{of dimension $d$}.
\end{defn}

If $A$ is a commutative AS regular algebra, then it is a polynomial ring. All known AS regular algebras are noetherian domains, and this is conjectured to always be the case. AS regular algebras are viewed as noncommutative analogues of commutative polynomial rings since they share many ring-theoretic and homological properties.  \\
\indent In dimension 2, which is the case of interest to us, it is easy to classify all AS regular algebras that are generated in degree one; up to isomorphism, they are the \emph{quantum plane} and \emph{Jordan plane}, which have respective presentations:
\begin{align*}
\Bbbk_q[u,v] = \frac{\Bbbk \langle u,v \rangle}{\langle vu - quv \rangle}, \quad q \in \Bbbk^\times,  \qquad \text{and} \qquad \Bbbk_J[u,v] = \frac{\Bbbk \langle u,v \rangle}{\langle vu - uv - u^2 \rangle}.
\end{align*}
\indent Given an AS regular algebra $A$, we are interested in actions of finite subgroups $G$ of $\Autgr(A)$ on $A$. In dimension $2$, all such groups can be viewed as subgroups of $\GL(2,\Bbbk)$, and in this case the action of $g = \begin{psmallmatrix} a & b \\ c & d \end{psmallmatrix}$ on $u$ and $v$ in the algebras above is given by
\begin{align*}
g \cdot u = au + cv, \quad g \cdot v = bu + dv.
\end{align*}
\indent In particular, we are interested in actions by groups which contain no \emph{quasi-reflections}, so we now recall the relevant definitions:

\begin{defn}
Suppose that $A$ is finitely graded and let $M \in \gr A$ (so that $M$ is left-bounded, in the sense that $M_i = 0$ for $i \ll 0$). Let $g$ be a graded endomorphism of $M$. Then the \emph{trace} of $g$ on $M$ is
\begin{align*}
\Tr_M(g) \coloneqq \sum_{i \in \ZZ} \tr(\restr{g}{M_i}) \gap t^i \in \Bbbk \llbracket t, t^{-1} \rrbracket,
\end{align*}
where $\tr(\restr{g}{M_i})$ is the usual trace of the linear map $\restr{g}{M_i} : M_i \to M_i$. \\
\indent Now assume that $A$ is AS regular, and that its Hilbert series has the form
\begin{align*}
\hilb A = \frac{1}{(1-t)^n f(t)},
\end{align*}
where $f(1) \neq 0$ (and hence $\GKdim A = n$). We say that $g \in \Autgr(A)$ is a \emph{quasi-reflection} if
\begin{align*}
\Tr_A(g) = \frac{1}{(1-t)^{n-1} p(t)},
\end{align*}
where $p(1) \neq 0$. We say that a finite subgroup $G \leqslant \Autgr(A)$ is \emph{small} if it contains no quasi-reflections.
\end{defn}

We remark that, if $A$ is a commutative polynomial ring, then $g$ is a quasi-reflection if and only if $g$ is a reflection in the classical sense. In the dimension $2$ case at hand, since both $\Bbbk_q[u,v]$ and $\Bbbk_J[u,v]$ have Hilbert series $(1-t)^{-2}$, a graded automorphism $g$ is a quasi-reflection if its trace has the form $\frac{1}{(1-t)(1-\lambda t)}$ for some $\lambda \neq 1$. \\
\indent We give a brief example which demonstrates the dependency of the trace on the algebra $A$.

\begin{example}
Let $h = \begin{psmallmatrix} 0 & 1 \\ 1 & 0 \end{psmallmatrix}$. Then $h$ can be viewed as a graded automorphism of both $R = \Bbbk[u,v]$ and $A = \Bbbk_{-1}[u,v]$. Noting that both $R$ and $A$ have $\Bbbk$-bases
\begin{align*}
\bigcup_{n \geqslant 0} \{ u^{n-i} v^i \mid 0 \leqslant i \leqslant n \},
\end{align*}
it is straightforward to calculate that
\begin{gather*}
\Tr_R(h) = 1 + t^2 + t^4 + t^6 + \dots = \frac{1}{1-t^2} = \frac{1}{(1-t)(1+t)}, \\
\Tr_A(h) = 1 - t^2 + t^4 - t^6 + \dots = \frac{1}{1+t^2}.
\end{gather*}
Hence $h$ is a quasi-reflection when it acts on $R$, but it is not a quasi-reflection when it acts on $A$.
\end{example}

We also note that the trace can be used to determine the Hilbert series of an invariant ring, giving a result which can be viewed as a noncommutative version of Molien's Theorem:

\begin{thm}[{\cite[Lemma 5.2]{jing}}]
Suppose that $A$ is finitely graded and $G$ is a finite subgroup of $\Autgr(A)$. Then
\begin{align*}
\hilb A^G = \frac{1}{|G|} \sum_{g \in G} \Tr_A(g).
\end{align*}
\end{thm}

\indent Initially, one might think that focusing attention on only small groups is quite restrictive. However, from the perspective of invariant theory, one can always assume that a finite subgroup $G$ of $\Autgr(A)$ is small. The proof of the following lemma is contained in the proof of \cite[Proposition 1.5 (a)]{stc}, where we note that the authors assume that $A$ is noetherian and AS regular, but these properties are not needed to establish what follows.

\begin{lem} \label{conjugateinvlem}
Suppose that $A$ is finitely graded. If $G$ is a finite subgroup of $\Autgr(A)$, then there exists a small group $G' \leqslant \Autgr(A)$ such that $A^G \cong A^{G'}$.
\end{lem}



\indent If we restrict to the case where $G$ is small, one can easily detect whether the invariant ring $A^G$ is AS Gorenstein. To be able to state this result, we first recall a definition:

\begin{defn}
Suppose that $A$ is AS regular and $g \in \Autgr(A)$. Then $\Tr_A(g)$ has a series expansion in $\Bbbk((t^{-1}))$ of the form
\begin{align*}
\Tr_A(g) = (-1)^d c^{-1} t^{-\ell} + \text{lower order terms},
\end{align*}
for some $c \in \Bbbk$, where $\ell$ is as in Definition \ref{ASregdef}. We call this constant $c$ the \emph{homological determinant of $g$}, which we denote by $\hdet(g)$. If $\hdet(g) = 1$ for all $g \in G$, then we say that the action of $G$ on $A$ has \emph{trivial homological determinant}.
\end{defn}

It is shown in \cite[Proposition 2.5]{jorgensen} that this assignment gives rise to a group homomorphism 
\begin{align*}
\hdet : G \to \Bbbk^\times.
\end{align*}
It is also shown in \cite[Section 2]{jorgensen} that if $A$ is the commutative polynomial ring $\Bbbk[x_1, \dots, x_n]$, then $\hdet(g) = \det(g)$ for all $g \in \GL(n,\Bbbk)$, justifying the terminology. \\
\indent We remark that the homological determinant is usually defined using local cohomology (see \cite[Definition 2.3]{jorgensen}), but the above definition is equivalent to the usual one by \cite[Lemma 2.6]{jorgensen}, and is sufficient for our needs. \\
\indent In dimension 2, by \cite[Theorem 2.1]{ckwz2} it is straightforward to calculate the homological determinant of a graded automorphism of an AS regular algebra.

\begin{lem} \label{homologicaldetlem}
Suppose that $A$ is a two-dimensional AS regular algebra, so either $A = \Bbbk_q[u,v]$ or $A = \Bbbk_J[u,v]$, and let $g \in \Autgr(A)$. Lift the action of $G$ on $A$ to an action on the free algebra $\Bbbk \langle u,v \rangle$. Let $r$ be the defining relation for $A$ so that $\Bbbk r$ is a one-dimensional submodule of $\Bbbk \langle u,v \rangle$. Then $g$ acts as scalar multiplication by $\hdet(g)$ on $\Bbbk r$.
\end{lem}

\indent This allows us to state the result to which we previously alluded:

\begin{thm}[{\cite[Theorem 4.10]{kkzGor}}] \label{ausgorthm}
Suppose that $A$ is AS regular and that $G$ is a finite subgroup of $\Autgr(A)$ which is small. Then $A^G$ is AS Gorenstein if and only if $\hdet(g) = 1$ for all $g \in G$.
\end{thm}

In the next section, we will determine the small subgroups of $\Autgr(A)$ when $A$ is AS regular of dimension 2. By the above theorem, it is then easy to check when the resulting invariant rings are AS Gorenstein.

\section{Small Subgroups of $\Autgr(A)$ when $A$ is AS Regular of Dimension 2} \label{smallactionssec} 
\noindent In this section, we will determine the small subgroups of $\Autgr(A)$ when $A$ is AS regular of dimension 2, i.e. $A$ is either a quantum plane or the Jordan plane.
In general, the graded automorphism group of $A$ is strictly contained in $\GL(2,\Bbbk)$, and so it will turn out that there are relatively few possibilities for $G$. We remark that, if $G$ and $H$ are finite subgroups of $\Autgr(A)$ which are conjugate, then $A^G$ and $A^H$ are isomorphic, so we only need to classify the groups up to conjugation. \\
\indent Throughout this section we will assume that $A$ is not commutative, i.e. if $A = \Bbbk_q[u,v]$, then $q \neq 1$. We will also assume that $G$ is nontrivial. The cases when $A = \Bbbk_q[u,v]$ ($q \neq -1$) or $A = \Bbbk_J[u,v]$ are easiest to analyse, and so we begin with these. The final case, when $A = \Bbbk_{-1}[u,v]$, is the most involved, and it is to this case that the majority of this section is devoted.

\subsection{Actions of Small Groups on the Jordan Plane}
It is straightforward to check that the graded automorphism group of the Jordan plane $A = \Bbbk_J[u,v]$ consists of automorphisms of the form
\begin{align*}
u \mapsto au, \quad v \mapsto bu + av,
\end{align*}
where $a \in \Bbbk^\times$ and $b \in \Bbbk$. Viewed as a subgroup of $\GL(2,\Bbbk)$, we therefore have 
\begin{align}
\Autgr(A) = \left \{ \begin{pmatrix} a & b \\ 0 & a \end{pmatrix} \bigg| \hspace{3pt} a \in \Bbbk^\times, b \in \Bbbk \right\}. \label{jordanaut}
\end{align}
\indent We now wish to identify all small subgroups of $\Autgr(A)$. By \cite[p. 7]{downup}, it is known that $\Autgr(A)$ contains no quasi-reflections, and so every finite subgroup of $\Autgr(A)$ is a small group, so we only need to classify the finite subgroups of $\Autgr(A)$. \\
\indent From (\ref{jordanaut}), it is clear that an element of $\Autgr(A)$ has finite order if and only if $b = 0$ and $a$ is a root of unity. 
We summarise the above discussion in the following lemma:

\begin{lem}
Let $A = \Bbbk_J[u,v]$. Then the small subgroups of $\Autgr(A)$ are 
\begin{align*}
\frac{1}{n}(1,1) \coloneqq \bigg\langle \hspace{-3pt} \begin{pmatrix} \omega_n & 0 \\ 0 & \omega_n \end{pmatrix} \hspace{-3pt} \bigg\rangle,
\end{align*}
where $n \geqslant 2$.
\end{lem}

\subsection{Actions of Small Groups on the Quantum Plane, $q \neq \pm 1$} \label{not-1classification} Now let $A = \Bbbk_q[u,v]$, where $q \neq \pm 1$ (the case $q=1$ is excluded to ensure that $A$ is not commutative). In this case, an easy calculation shows that every graded automorphism of $A$ is of the form 
\begin{align*}
u \mapsto au, \quad v \mapsto dv,
\end{align*}
where $a,b \in \Bbbk^\times$; that is,
\begin{align}
\Autgr(A) = \left \{ \begin{pmatrix} a & 0 \\ 0 & d \end{pmatrix} \bigg| \hspace{3pt} a, d \in \Bbbk^\times \right\}. \label{kqaut}
\end{align}
\indent We now wish to determine all small subgroups of $\Autgr(A)$. Clearly an element of $\Autgr(A)$ has finite order if and only if $a$ and $d$ are both (possibly distinct) roots of unity. Moreover, noting that the trace of an element $g$ of the form given in (\ref{kqaut}) is independent of $q$, we find that
\begin{align*}
\Tr_A(g) = \Tr_{\Bbbk[u,v]}(g) = \frac{1}{(1-at)(1-dt)}, 
\end{align*}
and so $g$ is a quasi-reflection if an only if precisely one of $a$ or $d$ is equal to 1. \\
\indent Now let $G$ be a small subgroup of $\Autgr(A)$, so that the only element of $G$ with at least one $1$ on the diagonal is the identity. Consider the map 
\begin{align*}
\phi : G \to \Bbbk^\times, \quad  \begin{pmatrix} a & 0 \\ 0 & d \end{pmatrix} \mapsto a.
\end{align*}
This map has trivial kernel (since if $a=1$ then necessarily $d=1$), and the image is a finite subgroup of $\Bbbk^\times$, which is necessarily cyclic of order $n$, say. Therefore $G \cong \im \phi$ is cyclic of order $n$. It is straightforward to show that this forces $G$ to be generated by an element of the form
\begin{align*}
g = \begin{pmatrix}
\omega_n & 0 \\ 0 & \omega_n^a \end{pmatrix},
\end{align*} 
for some integers $a$ and $n$ satisfying $1 \leqslant a < n$. To ensure that $G$ contains no quasi-reflections, we also require $\gcd(a,n) = 1$. 
Again, we summarise our findings:

\begin{lem}
Let $A = \Bbbk_q[u,v]$, where $q \neq \pm 1$. Then the small subgroups of $\Autgr(A)$ are 
\begin{align*}
\frac{1}{n}(1,a) \coloneqq \bigg\langle \hspace{-3pt} \begin{pmatrix} \omega_n & 0 \\ 0 & \omega_n^a \end{pmatrix} \hspace{-3pt} \bigg\rangle,
\end{align*}
where $1 \leqslant a < n$ and $\gcd(a,n) = 1$.
\end{lem}

\subsection{Actions of Small Groups on the $(-1)$-Quantum Plane}
Throughout this subsection, let $A = \Bbbk_{-1}[u,v]$. We now seek to classify all small subgroups of $\Autgr(A)$. We first note that
\begin{align*}
\Autgr(A) = \left \{ \begin{pmatrix} a & 0 \\ 0 & d \end{pmatrix} \bigg| \hspace{3pt} a,d \in \Bbbk^\times \right\} \cup \left \{ \begin{pmatrix} 0 & b \\ c & 0 \end{pmatrix} \bigg| \hspace{3pt} b,c \in \Bbbk^\times \right \} \cong (\Bbbk^*)^2 \rtimes C_2,
\end{align*}
where we identify automorphisms with elements of $\GL(2,\Bbbk)$ in the usual way. We call elements in the second set \emph{antidiagonal}. \\
\indent For $\Bbbk_J[u,v]$ and $\Bbbk_q[u,v]$ ($q \neq \pm 1$), we have seen that it is possible to write down all of the small subgroups of $\Autgr(A)$, not just up to conjugation. For $\Bbbk_{-1}[u,v]$ it is more convenient to classify the small subgroups of $\Autgr(A)$ up to conjugation which, by Lemma \ref{conjugateinvlem}, has no ill-effects from the perspective of invariant theory.
Note that conjugation by an element of $\Autgr(A)$ has no effect on the diagonal elements in $G$, so the presence (or absence) of antidiagonal elements in $G$ is not affected by conjugation by diagonal elements. \\
\indent Our first step is to identify precisely when an element of $\Autgr(A)$ is a quasi-reflection.

\begin{lem} \label{qreflem}
Suppose that $g = \begin{psmallmatrix} a & b \\ c & d \end{psmallmatrix} \in \Autgr(A)$. If $g$ is diagonal, then
\begin{align*}
\Tr_A(g) = \frac{1}{(1-at)(1-dt)},
\end{align*}
while if $g$ is antidiagonal, then
\begin{align*}
\Tr_A(g) = \frac{1}{1+bc \gap t^2}.
\end{align*}
In particular, if $g$ is diagonal (respectively, antidiagonal), then it is a quasi-reflection if and only if $a=1$ or $d=1$, but not both (respectively, $bc=-1$).
\end{lem}
\begin{proof}
The trace of a diagonal automorphism was noted in Section \ref{not-1classification}, so suppose that $g = \begin{psmallmatrix} 0 & b \\ c & 0 \end{psmallmatrix}$ is antidiagonal. By \cite[Corollary 4.4]{jing}, since $A$ is Koszul, the trace of $g$ on $A$ can be computed using the Koszul dual $A^!$ of $A$. It is easy to check that $A^! \cong \Bbbk[x,y]/\langle x^2, y^2 \rangle$, and that the induced map $g^!$ acts via $g^!(x) = cy$, $g^!(y) = bx$. From this one calculates
\begin{align*}
\Tr_{A^!}(g^!) = 1 + bc \gap t^2,
\end{align*}
and then \cite[Corollary 4.4]{jing} says that
\begin{align*}
\Tr_{A}(g) = \frac{1}{1 + bc \gap (-t)^2} = \frac{1}{1 + bc \gap t^2}.
\end{align*}
\indent Finally, since $g$ is a quasi-reflection if and only if its trace is given by 
\begin{align*}
\frac{1}{(1-t)(1-\lambda t)}
\end{align*}
for some $\lambda \neq 1$, the final claim follows.
\end{proof}

\indent To begin the classification, suppose that $G$ is a small subgroup of $\Autgr(A)$. If $G$ contains no antidiagonal elements then every non-identity element is a diagonal matrix. By the same argument as in Section \ref{not-1classification}, $G = \frac{1}{n}(1,a)$, where $1 \leqslant a < n$ and $\gcd(a,n) = 1$. \\
\indent So now suppose that $G$ contains at least one antidiagonal element. There exists a short exact sequence of groups
\begin{align*}
1 \to G \cap \SL(2,\Bbbk) \to G \xrightarrow{\det} \langle \omega_m \rangle \to 1,
\end{align*}
for some $m \geqslant 2$, where $G \cap \SL(2,\Bbbk)$ is a finite subgroup of $\SL(2,\Bbbk)$. However, since $G$ contains no quasi-reflections, by Lemma \ref{qreflem} all of its antidiagonal elements have determinant not equal to 1, so $G \cap \SL(2,\Bbbk)$ consists of diagonal matrices of determinant 1. Therefore the only possibility is that 
\begin{align*}
G \cap \SL(2,\Bbbk) = \bigg\langle \hspace{-3pt} \begin{pmatrix}
\omega_n & 0 \\ 0 & \omega_n^{-1}
\end{pmatrix} \hspace{-3pt} \bigg\rangle
\end{align*}
for some integer $n \geqslant 1$ (if $n=1$, this group is trivial); call the generator of this group $g$. Since $G/(G \cap \SL(2,\Bbbk))$ is cyclic, it follows that there exists an antidiagonal element $h \in G$ such that $G = \langle g,h \rangle$. If $h = \begin{psmallmatrix} 0 & b \\ c & 0 \end{psmallmatrix}$ has (necessarily even) order $2k$ then $bc$ is a primitive $k$th root of unity, and conjugating by $\begin{psmallmatrix} \sqrt{b} & 0 \\ 0 & \sqrt{c} \end{psmallmatrix}$ yields $\begin{psmallmatrix} 0 & \sqrt{bc} \\ \sqrt{bc} & 0 \end{psmallmatrix} = \begin{psmallmatrix} 0 & \omega_{2k} \\ \omega_{2k} & 0 \end{psmallmatrix}$. So we can assume $h = \begin{psmallmatrix} 0 & \omega_{2k} \\ \omega_{2k} & 0 \end{psmallmatrix}$ for some $k \geqslant 1$. \\
\indent However, not all values of $n$ and $k$ give rise to small subgroups of $\Autgr(A)$: for example, if $k=2$ then $h = \begin{psmallmatrix} 0 & \omega_4 \\ \omega_4 & 0 \end{psmallmatrix}$, which is a quasi-reflection. We now seek to eliminate those $n$ and $k$ which result in groups $G$ containing quasi-reflections. Henceforth, it will be convenient to let $\omega$ be a primitive $(2nk)$th root of unity, so that $w^{2k}$ is a primitive $n$th root of unity and $w^n$ is a primitive $(2k)$th root of unity, and to write
\begin{align}
G_{n,k} = \langle g,h \rangle, \qquad \text{where} \qquad g = \begin{pmatrix}\omega^{2k} & 0 \\ 0 & \omega^{-2k} \end{pmatrix} \quad \text{and} \quad h = \begin{pmatrix}0 & \omega^{n} \\ \omega^{n} & 0 \end{pmatrix}. \label{Ghyp}
\end{align}

\begin{lem} \label{2mod4lem}
Suppose that $G = G_{n,k}$ is as in \emph{(\ref{Ghyp})} and $k \equiv 2 \normalfont{\text{ mod }} 4$. Then $G$ contains a quasi-reflection.
\end{lem}
\begin{proof}
If $k \equiv 2 \normalfont{\text{ mod }} 4$ then $\frac{k}{2}$ is odd and so
\begin{align*}
h^{\frac{k}{2}} = 
\begin{pmatrix}
0 & \omega^{\frac{nk}{2}} \\ \omega^{\frac{nk}{2}} & 0
\end{pmatrix}.
\end{align*}
The product of the antidiagonal elements of $h^{\frac{k}{2}}$ is $\omega^{nk} = -1$, and hence $h^{\frac{k}{2}}$ is a quasi-reflection by Lemma \ref{qreflem}.
\end{proof}

\begin{lem} \label{nkleq2}
Suppose that $G=G_{n,k}$ is as in \emph{(\ref{Ghyp})}. Then $G$ contains no quasi-reflections if and only if $k \not\equiv 2 \normalfont{\text{ mod }} 4$ and $\gcd(n,k) \leqslant 2$. 
\end{lem}
\begin{proof}
\indent ($\Rightarrow$) Lemma \ref{2mod4lem} already shows that if $k \equiv 2 \normalfont{\text{ mod }} 4$ then $G$ contains a quasi-reflection, so write $d = \gcd(n,k)$ and suppose $d > 2$. Setting $i = \frac{n}{d}$ and $j = \frac{k}{d}$, direct calculation gives
\begin{align*}
g^i h^{2j} = \begin{pmatrix} \omega^{2(nj+ki)} & 0 \\ 0 & \omega^{2(nj-ki)} \end{pmatrix}.
\end{align*}
Now, $2(nj+ki) = 2(\frac{nk}{d} + \frac{kn}{d}) = \frac{4nk}{d} < 2nk$, where the inequality follows since $d > 2$. Therefore $\omega^{2(nj+ki)} \neq 1$. On the other hand, $2(nj-ki) = 2(\frac{nk}{d} - \frac{kn}{d}) = 0$, so $\omega^{2(nj-ki)} = 1$. In particular, by Lemma \ref{qreflem}, $g^i h^{2j}$ is a quasi-reflection. \\
\indent ($\Leftarrow$) Now assume that $k \not\equiv 2 \normalfont{\text{ mod }} 4$ and $\gcd(n,k) \leqslant 2$, and consider a diagonal element of $G$, which therefore has the form $g^i h^{2j}$ for some $i$ and $j$. As above,
\begin{align*}
g^i h^{2j} = \begin{pmatrix} \omega^{2(nj+ki)} & 0 \\ 0 & \omega^{2(nj-ki)} \end{pmatrix}.
\end{align*}
Now suppose $\omega^{2(nj+ki)} = 1$, so $nj+ki = nkr$ for some integer $r$. Rearranging and dividing through by $d = \gcd(n,k)$, we find $\frac{k}{d} i = \frac{n}{d}(kr-j)$, so that $\frac{n}{d}$ divides $\frac{k}{d}i$. Since $\frac{n}{d}$ and $\frac{k}{d}$ are coprime this implies that $\frac{n}{d} \mid i$. Therefore $n \mid di$, where $d$ is either $1$ or $2$, so $n \mid 2i$. Since $nj+ki = nkr$, we have $nj-ki = nkr - 2ki = k(nr-2i)$. By the above argument, we see that $n$ divides the term $nr-2i$, and hence $2(nj-ki) = 2k(nr-2i)$ is an integer multiple of $2nk$. Therefore $\omega^{2(nj-ki)} = 1$, and so $g^i h^{2j}$ is the identity. A similar analysis shows that if $\omega^{2(nj-ki)} = 1$ then also $\omega^{2(nj+ki)} = 1$, and so $G$ contains no diagonal elements which are quasi-reflections. \\
\indent Now consider an antidiagonal element of $G$, which must have the form $g^i h^{2j+1}$ for some $i$ and $j$. Direct calculation gives
\begin{align*}
g^i h^{2j+1} = \begin{pmatrix} 0 & \omega^{(2j+1)n+2ki} \\ \omega^{(2j+1)n-2ki} & 0 \end{pmatrix}
\end{align*}
and the product of the diagonal elements is $\omega^{2n(2j+1)}$. By Lemma \ref{qreflem}, it follows that $g^i h^{2j+1}$ is a quasi-reflection if and only if $\omega^{2n(2j+1)} = -1$, which happens if and only if $2(2j+1) = k(2r+1)$ for some integer $r$. Since $k \not \equiv 2 \normalfont{\text{ mod }} 4$, this has no solutions, and hence $g^i h^{2j+1}$ is not a quasi-reflection.
\end{proof}

We have therefore determined precisely when $G_{n,k}$ contains no quasi-reflections. However, some values of $n$ and $k$ give rise to the same groups:

\begin{lem} \label{equalgroups}
In the notation of \emph{(\ref{Ghyp})}, if $n \equiv 2 \normalfont{\text{ mod }} 4$ and $k \equiv 0 \normalfont{\text{ mod }} 4$, then $G_{n,k} = G_{\frac{n}{2},k}$.
\end{lem}
\begin{proof}
Letting $\omega$ be a primitive $(2nk)$th root of unity, we need to show that
\begin{align*}
\Bigg\langle \hspace{-3pt} \begin{pmatrix} \omega^{2k} & 0 \\ 0 & \omega^{-2k} \end{pmatrix}, \begin{pmatrix}0 & \omega^{n} \\ \omega^{n} & 0 \end{pmatrix} \hspace{-3pt} \Bigg\rangle = \Bigg\langle \hspace{-3pt} \begin{pmatrix} \omega^{4k} & 0 \\ 0 & \omega^{-4k} \end{pmatrix}, \begin{pmatrix}0 & \omega^{n} \\ \omega^{n} & 0 \end{pmatrix} \hspace{-3pt} \Bigg\rangle.
\end{align*}
The fact that the right hand side is contained in the left hand side is obvious. 
For the reverse inclusion, it suffices to show that $\begin{psmallmatrix} \omega^{2k} & 0 \\ 0 & \omega^{-2k} \end{psmallmatrix}$ lies in the right hand side. Indeed, we have
\begin{align*}
\begin{pmatrix} \omega^{4k} & 0 \\ 0 & \omega^{-4k} \end{pmatrix}^{\frac{n+2}{4}}
\begin{pmatrix}[1.1] 0 & \omega^n \\ \omega^n & 0 \end{pmatrix}^k
=
\begin{pmatrix} \omega^{k(n+2)} & 0 \\ 0 & \omega^{-k(n+2)} \end{pmatrix}
\begin{pmatrix}[1.1] 0 & \omega^{nk} \\ \omega^{nk} & 0 \end{pmatrix} 
= 
\begin{pmatrix} \omega^{2k} & 0 \\ 0 & \omega^{-2k} \end{pmatrix}, 
\end{align*}
as required.
\end{proof}

By Lemma \ref{equalgroups}, if $k$ is even, we can assume that $n$ is odd. However, since Lemma \ref{nkleq2} also tells us that if $G$ contains no quasi-reflections then necessarily  $\gcd(n,k) \leqslant 2$, $n$ and $k$ can be assumed to be coprime. To summarise, we have shown the following:

\begin{thm} \label{classifyG}
Let $A = \Bbbk_{-1}[u,v]$. Then, up to conjugation by elements of $\Autgr(A)$, the small subgroups of $\Autgr(A)$ are 
\begin{align*}
\frac{1}{n}(1,a) = \bigg\langle \hspace{-3pt}
\begin{pmatrix} \omega_n & 0 \\ 0 & \omega_n^{a} \end{pmatrix}
\hspace{-3pt}\bigg\rangle ,
\end{align*}
where $1 \leqslant a < n$ and $\gcd(a,n) = 1$, and
\begin{align*}
G_{n,k} = \bigg\langle \hspace{-3pt}
\begin{pmatrix} \omega_n & 0 \\ 0 & \omega_n^{-1} \end{pmatrix},
\begin{pmatrix} 0 & \omega_{2k} \\ \omega_{2k} & 0 \end{pmatrix}
\hspace{-3pt}\bigg\rangle ,
\end{align*}
where $\gcd(n,k) = 1$ and $k \not\equiv 2 \normalfont{\text{ mod }} 4$.
\end{thm}

We record the following lemma, which is easy to prove:

\begin{lem} \label{Gpropslem}
Suppose that $G_{n,k}$ is as in Theorem \ref{classifyG}. Then this group has order $2nk$ and has the presentation
\begin{align*}
G_{n,k} \cong \langle \gap g,h \mid g^n, h^{2k}, hg = g^{n-1} h \gap \rangle.
\end{align*}
\end{lem}

It turns out that the invariant rings $A^{G_{n,k}}$ are frequently not commutative. While we postpone writing down formulas for the generators and relations until sections \ref{comminvsec} and \ref{noncomminvsec}, the following shows which choices of $n$ and $k$ lead to invariant rings that are not commutative.

\begin{prop} \label{commvsnoncomm}
Let $G = G_{n,k}$ be as above. Then $A^{G}$ is commutative if and only if $n$ or $k$ is even.
\end{prop}
\begin{proof}
($\Leftarrow$) First suppose that at least one (and hence exactly one) of $n$ and $k$ is even. If $n$ is even, then $g^{n/2} = \begin{psmallmatrix} -1 & 0 \\ 0 & -1 \end{psmallmatrix}$, while if $k$ is even then $h^k = \begin{psmallmatrix} -1 & 0 \\ 0 & -1 \end{psmallmatrix}$. In either case, call this element $\alpha$. Then
\begin{align*}
A^G \subseteq A^{\langle \alpha \rangle} = \Bbbk[u^2,v^2,uv] \cong \frac{\Bbbk[x,y,z]}{\langle xy+z^2 \rangle},
\end{align*}
where the last ring is commutative. Therefore $A^G$ is commutative. \\
\indent ($\Rightarrow$) As before, it will be convenient to let $\omega$ be a primitive $(2nk)$th root of unity, and to assume $\omega_n = \omega^{2k}$ and $\omega_{2k} = \omega^n$. Now suppose that both $n$ and $k$ are odd, and choose an odd integer $m$ such that $mk > n$. Set $i = \frac{mk+n}{2}$ and $j = \frac{mk-n}{2}$, both of which are positive integers. Moreover, since $i-j = n$ is odd, one of $i$ and $j$ is odd and the other is even. Also observe that $i+j = mk$. Now set $a = u^iv^j - u^jv^i$ and $b = u^{3i}v^{3j} - u^{3j} v^{3i}$. Then
\begin{align*}
g \cdot a = \omega^{2k(i-j)} u^i v^j - \omega^{2k(j-i)} u^j v^i = \omega^{2nk} u^i v^j - \omega^{-2nk} u^j v^i = u^iv^j - u^jv^i = a,
\end{align*}
so $a$ is invariant under the action of $g$. Moreover, noting that $u^i$ and $v^j$ commute since one of $i$ and $j$ is even, 
\begin{align*}
h \cdot a &= \omega^{n(i+j)} v^i u^j - \omega^{n(i+j)} v^j u^i = \omega^{mnk}(u^j v^i - u^i v^j) \\
&= -\omega^{mnk} (u^i v^j - u^j v^i) = \omega^{(m+1)nk} a = a,
\end{align*}
so that $a$ is also invariant under the action of $h$ and hence lies in $\Bbbk_{-1}[u,v]^G$. Similarly, one can show that $b \in A^G$. \\
\indent Finally, we show that $a$ and $b$ do not commute, showing that $A[u,v]^G$ is not commutative. Assuming that $i$ is even and $j$ is odd, we have
\begin{align*}
ab &= (u^iv^j - u^jv^i) (u^{3i}v^{3j} - u^{3j} v^{3i}) \\
&= u^iv^j u^{3i}v^{3j} - u^iv^ju^{3j}v^{3i}-u^jv^iu^{3i}v^{3j} + u^jv^iu^{3j}v^{3i} \\
&= u^{4i}v^{4j} + u^{i+3j}v^{3i+j} - u^{3i+j} v^{i+3j} + u^{4j}v^{4i}
\end{align*}
while
\begin{align*}
ba &= (u^{3i}v^{3j} - u^{3j} v^{3i})(u^iv^j - u^jv^i) \\
&= u^{3i}v^{3j}u^iv^j -u^{3i}v^{3j} u^jv^i - u^{3j}v^{3i} u^iv^j + u^{3j}v^{3i}u^j v^i \\
&= u^{4i}v^{4j} + u^{3i+j}v^{i+3j} - u^{i+3j} v^{3i+j} + u^{4j}v^{4i}.
\end{align*}
Therefore
\begin{align*}
ab-ba = 2 u^{i+3j}v^{3i+j} - 2 u^{3i+j} v^{i+3j} \neq 0.
\end{align*}
If instead $i$ is odd and $j$ is even, then a similar calculation shows that
\begin{align*}
ab-ba = 2 u^{3i+j} v^{i+3j} - 2 u^{i+3j}v^{3i+j} \neq 0,
\end{align*}
and so in either case $a$ and $b$ do not commute, as claimed.
\end{proof}

\subsection{Summary of the classification}
We now summarise our findings from this section.

\begin{thm} \label{smallclassification} 
Suppose that $A$ is a two-dimensional AS regular algebra which is not commutative and that $G$ is a small subgroup of $\Autgr(A)$. Then, up to conjugation of $G$ by an element of $\Autgr(A)$, the possible pairs $(A,G)$ are as follows:
\begin{figure}[h]
{\tabulinesep=1.5pt
\begin{tabu}{c|c|c|c|l}
\normalfont{Case} & $A$ & $G$ & \normalfont{Generators} & \normalfont{Conditions} \\ \hline
\normalfont{(i)} &  $\Bbbk_q[u,v]$ & $\frac{1}{n}(1,a)$ & \hspace{-3pt} $\begin{pmatrix} \omega_n & 0 \\ 0 & \omega_n^a \end{pmatrix} \hspace{-3pt} $ & $q \neq 1$, $1 \leqslant a < n$, $\gcd(a,n) = 1$. \\ \hline
\normalfont{(ii)} &  $\Bbbk_{-1}[u,v]$ & $G_{n,k}$ & $ \hspace{-3pt} \begin{pmatrix} \omega_n & 0 \\ 0 & \omega_n^{-1} \end{pmatrix}$, $\begin{pmatrix} 0 & \omega_{2k} \\ \omega_{2k} & 0 \end{pmatrix} \hspace{-3pt}$ & $k \not\equiv 2 \text{\normalfont{ mod }} 4$, $\gcd(n,k) = 1$. \\ \hline
\normalfont{(iii)} &  $\Bbbk_J[u,v]$ & $\frac{1}{n}(1,1)$ & $ \hspace{-3pt} \begin{pmatrix} \omega_n & 0 \\ 0 & \omega_n \end{pmatrix} \hspace{-3pt} $ & $n \geqslant 2$.  \\
\end{tabu}}
\end{figure}
\end{thm}

Using Lemma \ref{homologicaldetlem} and Theorem \ref{ausgorthm}, the following is easy to check:

\begin{cor} \label{gorensteincor}
Suppose that $(A,G)$ is a pair from Theorem \ref{smallclassification}. Then the action has trivial homological determinant, and hence $A^G$ is AS Gorenstein, if and only if: 
\begin{itemize}[leftmargin=20pt,topsep=0pt,itemsep=2pt]
\item Case \emph{(i):} $a = n-1$;
\item Case \emph{(ii):} $k=1$;
\item Case \emph{(iii):} $n = 2$.
\end{itemize}
\end{cor}

The cases in which the action has trivial homological determinant are precisely those considered in \cite{ckwz}, with the exception of $(n,k) = (2,1)$ in case (ii). We will see later (Remark \ref{smallinvringiso}) that, from the perspective of invariant theory, its omission from \cite{ckwz} is not too important. We also remark that, in case (ii) when $k = 1$, the 2-dimensional representation of our groups is different from those appearing in \cite{ckwz}, but they are conjugate under $\Autgr(A)$. This causes a change in some signs when writing down generators for the invariant rings, but has no other effects.


\section{Auslander's Theorem} \label{auslandersec}
\noindent The goal of this section is to show that the Auslander map is an isomorphism for any small subgroup of $\Autgr(A)$, where $A$ is AS regular of dimension 2. That is, we wish to show that if $(A,G)$ is a pair from Theorem \ref{smallclassification}, then the natural homomorphism
\begin{align*}
\phi: A \hash G \to \End_{A^G}(A), \quad \phi(ag)(b) = a (g \cdot b)
\end{align*}
is an isomorphism. We use the same strategy as was used in, for example, \cite{ckwz,bhz2,won} to establish similar isomorphisms; namely, we show that the conditions of the following result holds, which we restate from the introduction:

\begin{thm}[{\cite[Theorem 0.3]{bhz}}]
Let $G$ be a finite group acting on an AS regular, GK-Cohen-Macaulay algebra $A$ with $\GKdim A \geqslant 2$. Let $\overline{g} = \sum_{g \in G} g$, viewed as an element of $A \hash G$. Then the Auslander map is an isomorphism for the pair $(A,G)$ if and only if 
\begin{align*}
\GKdim \big((A \hash G)/\langle \overline{g} \rangle \big) \leqslant \GKdim A - 2.
\end{align*}
\end{thm}

\indent It is well-known that the algebras $\Bbbk_q[u,v]$ and $\Bbbk_J[u,v]$ satisfy the hypotheses on $A$ given in this theorem. Since in our setting $\GKdim A = 2$, we need to show that $\GKdim (A \hash G)/\langle \overline{g} \rangle =0$, i.e. that $(A \hash G)/\langle \overline{g} \rangle$ is finite dimensional. \\
\indent We split the proof into two parts. We first assume that $G$ is a diagonal subgroup of $\Autgr(A)$, so that we are considering either case (i) or case (iii) from Theorem \ref{smallclassification}. We will then consider case (ii), which requires a more involved argument.

\subsection{Diagonal groups}
Suppose that $A$ is either $\Bbbk_q[u,v]$ ($q \neq 1$) or $\Bbbk_J[u,v]$ and that 
\begin{align*}
G = \frac{1}{n} (1,a) = \bigg \langle \hspace{-3pt} \begin{pmatrix} \omega & 0 \\ 0 & \omega^a \end{pmatrix} \hspace{-3pt} \bigg \rangle,
\end{align*}
where $\omega^{n} = 1$, $1 \leqslant a < n$, and $\gcd(a,n) = 1$. (Necessarily $a=1$ if $A = \Bbbk_J[u,v]$.) Write $g$ for the given generator of $\frac{1}{n} (1,a)$, so that
\begin{align*}
\overline{g} = 1 + g + \dots + g^{n-1}.
\end{align*}
We note that $\omega$ satisfies the following identity:
\begin{align}
\sum_{j = 0}^n \omega^{ij} = \left\{
\begin{array}{ll}
n & \text{if } i \equiv 0 \text{ mod } n, \\
0 & \text{otherwise}.
\end{array} \right. \label{omegaidentity}
\end{align}

\begin{thm} \label{diagausthm}
Suppose that $A$ and $G$ are as above. Then $(A \hash G)/\langle \overline{g} \rangle$ is finite dimensional. In particular, the Auslander map is an isomorphism for the pair $(A,G)$.
\end{thm}
\begin{proof}
Let $0 \leqslant j \leqslant n-1$. The element $u^{n-1-j}  \gap\gap \overline{g} u^j$ lies in $\langle \overline{g} \rangle$, and
\begin{align*}
u^{n-1-j}  \gap\gap \overline{g} u^j = \sum_{i=0}^{n-1} \omega^{ij} u^{n-1} g^i.
\end{align*}
Summing over all $j$, we have
\begin{align*}
\langle \overline{g} \rangle \ni \sum_{j=0}^{n-1} u^{n-1-j}  \gap\gap \overline{g} u^j = \sum_{i=0}^{n-1} \bigg(\gap\sum_{j=0}^{n-1} \omega^{ij} \bigg) u^{n-1} g^i.
\end{align*}
Using (\ref{omegaidentity}), $\sum_{j=0}^{n-1} \omega^{ij}$ is equal to 0 unless $i=0$. Therefore $u^{n-1} \in \langle \overline{g} \rangle$. Similarly,
\begin{align*}
\langle \overline{g} \rangle \ni \sum_{j=0}^{n-1} v^{n-1-j}  \gap\gap \overline{g} v^j = \sum_{i=0}^{n-1} \bigg(\gap\sum_{j=0}^{n-1} \omega^{aij} \bigg) v^{n-1} g^i.
\end{align*}
Since $a$ is coprime to $n$, $ai \equiv 0 \text{ mod } n$ if and only if $i \equiv 0 \text{ mod } n$, and therefore as above we find that $v^{n-1} \in \langle \overline{g} \rangle$. It follows that $A_{\geqslant 2(n-1)}$ is contained in the ideal $\langle \overline{g} \rangle$, and hence $(A \hash G)_{\geqslant 2(n-1)} \subseteq \langle \overline{g} \rangle$. Therefore $(A \hash G)/\langle \overline{g} \rangle$ is finite dimensional, as claimed.
\end{proof}

\subsection{Nondiagonal groups}
We now consider case (ii), so $A = \Bbbk_{-1}[u,v]$ and $G = G_{n,k}$, where $n$ and $k$ are positive coprime integers with $k \not\equiv 2 \text{ mod } 4$. It will be convenient to let $\omega$ be a primitive $(2nk)$th root of unity, and to think of $G$ as being generated by the elements $g$ and $h$, where
\begin{align*}
g = \begin{pmatrix} \omega^{2k} & 0 \\ 0 & \omega^{-2k} \end{pmatrix}, \qquad
h = \begin{pmatrix} 0 & \omega^n \\ \omega^n & 0 \end{pmatrix}.
\end{align*}
By Lemma \ref{Gpropslem}, this is a group of order $2nk$ and has a presentation given by
\begin{align*}
G \cong \langle g,h \mid g^n, h^{2k}, hg = g^{n-1} h \rangle.
\end{align*}
A complete list of elements of $G$ is
\begin{align*}
\Big \{ \gap g^i h^{2j} \gap\gap \Big | \gap\gap 0 \leqslant i \leqslant n{-}1, \gap \gap 0 \leqslant j \leqslant k{-}1 \gap \Big \} \sqcup \Big \{ \gap g^i h^{2j+1} \gap\gap \Big | \gap\gap 0 \leqslant i \leqslant n{-}1, \gap \gap 0 \leqslant j \leqslant k{-}1 \gap \Big \},
\end{align*}
where the first set in this disjoint union consists of diagonal elements, and the second set consists of antidiagonal elements. We will frequently make use of this fact. \\
\indent In this subsection, we show that the Auslander map is an isomorphism for the pair $(A,G)$ using a sequence of lemmas. We first define some elements of $A \hash G$ that will play an important role in the proof. For any integer $\ell$, define
\begin{align*}
G_\ell = \sum_{i = 0}^{n-1} \sum_{j = 0}^{k-1} \omega^{2\ell (nj+ki)} g^i h^{2j}, \qquad
H_\ell = \sum_{i = 0}^{n-1} \sum_{j = 0}^{k-1} \omega^{\ell (n(2j+1) - 2ki)} g^i h^{2j+1}.
\end{align*}
In the following lemma, we will show (see part (5)) that these elements satisfy
\begin{align*}
G_0 u^\ell = u^\ell G_\ell \quad \text{and} \quad H_0 u^\ell = v^\ell H_\ell
\end{align*}
for any $\ell \geqslant 0$, which motivates their definition. We also record a number of other useful properties:

\begin{lem} \label{Gprops} Let $\ell$ be an integer (in parts (5) and (6), also assume that $\ell \geqslant 0$).
\begin{enumerate}[{\normalfont (1)},leftmargin=*,topsep=0pt,itemsep=2pt]
\item $G_\ell = G_{\ell + nkr}$ for any integer $r$.
\item $H_\ell = (-1)^{nr} H_{\ell + nkr}$ for any integer $r$.
\item $\sum_{\ell=0}^{nk-1} G_\ell = nk$.
\item $\overline{g} = G_0 + H_0$.
\item $G_0 u^\ell = u^\ell G_\ell$ and $H_0 u^\ell = v^\ell H_\ell$.
\item Since $n$ and $k$ are coprime, there exist integers $a,b$ with $an+bk=1$. Set $m = an-bk$, which is uniquely determined modulo $2nk$. Then $G_0 v^\ell = v^\ell G_{m\ell}$ and $H_0 v^\ell = u^\ell H_{m \ell}$.
\end{enumerate}
\end{lem}
\begin{proof} \leavevmode
\begin{enumerate}[{\normalfont (1)},wide=0pt,topsep=0pt,itemsep=2pt]
\item Since $\omega^{2nk} = 1$,
\begin{align*}
G_{\ell+nkr} &= \sum_{i = 0}^{n-1} \sum_{j = 0}^{k-1} \omega^{2(\ell+nkr)(nj+ki)} g^i h^{2j} = \sum_{i = 0}^{n-1} \sum_{j = 0}^{k-1} \omega^{2\ell(nj+ki)}\omega^{2nkr(nj+ki)} g^i h^{2j} \\
&= \sum_{i = 0}^{n-1} \sum_{j = 0}^{k-1} \omega^{2\ell(nj+ki)} g^i h^{2j} = G_\ell.
\end{align*}
\item In a similar vein to the above, we have
\begin{align*}
H_{\ell+nkr} &= \sum_{i = 0}^{n-1} \sum_{j = 0}^{k-1} \omega^{(\ell+nkr) (n(2j+1) - 2ki)} g^i h^{2j+1} \\
&= \sum_{i = 0}^{n-1} \sum_{j = 0}^{k-1} \omega^{\ell (n(2j+1) - 2ki)} \omega^{n^2kr(2j+1) - 2nk^2ri} g^i h^{2j+1} \\
&= \sum_{i = 0}^{n-1} \sum_{j = 0}^{k-1} \omega^{\ell (n(2j+1) - 2ki)} \omega^{(nr)kr} g^i h^{2j+1} \\
&= (-1)^{nr} H_\ell.
\end{align*}
\item We have 
\begin{align*}
\sum_{\ell=0}^{nk-1} G_\ell = \sum_{i = 0}^{n-1} \sum_{j = 0}^{k-1} \left( \sum_{\ell=0}^{nk-1} \omega^{2\ell (nj+ki)} \right) g^i h^{2j}.
\end{align*}
Writing $\varepsilon = \omega^2$, which is a primitive $nk$th root of unity, the term in parentheses is
\begin{align*}
\sum_{\ell=0}^{nk-1} \varepsilon^{\ell (nj+ki)} = \left \{
\begin{array}{ll}
    nk & \text{if } nj+ki \equiv 0 \text{ mod } nk, \\
    0 & \text{otherwise.}
\end{array}
\right.
\end{align*}
Since $n$ and $k$ are coprime, the only way for $nj+ki$ to be divisible by $nk$ is if $j$ is divisible by $k$ and $i$ is divisible $k$. Since $0 \leqslant i < n$ and $0 \leqslant j < k$, this happens if and only if $i=0=j$, so
\begin{align*}
\sum_{\ell=0}^{nk-1} \varepsilon^{\ell (nj+ki)} = \left \{
\begin{array}{ll}
    nk & \text{if } i=0=j, \\
    0 & \text{otherwise.}
\end{array}
\right.
\end{align*}
Therefore $\sum_{\ell=0}^{nk-1} G_\ell = nk$, as claimed.
\item This is clear.
\item These are both true by construction; indeed, since $2j$ is even,
\begin{align*}
G_0 u^\ell &= \sum_{i = 0}^{n-1} \sum_{j = 0}^{k-1} g^i h^{2j} u^\ell
= \sum_{i = 0}^{n-1} \sum_{j = 0}^{k-1} \omega^{2 \ell nj} g^i u^\ell h^{2j} 
= \sum_{i = 0}^{n-1} \sum_{j = 0}^{k-1} \omega^{2 \ell nj + 2\ell k i} u^\ell g^i h^{2j} 
= u^\ell G_\ell.
\end{align*}
The calculation showing that $H_0 u^\ell = v^\ell H_\ell$ is similar, after noting that, since $2j+1$ is always odd, $h^{2j+1} u^\ell = \omega^{\ell n(2j+1)} v^\ell h^{2j+1}$.
\item This time we provide the calculation for $H_0 v^\ell$, with the calculation for $G_0 v^\ell$ being similar. We have
\begin{align*}
H_0 v^\ell &= \sum_{i = 0}^{n-1} \sum_{j = 0}^{k-1} g^i h^{2j+1} v^\ell 
= \sum_{i = 0}^{n-1} \sum_{j = 0}^{k-1} \omega^{\ell n (2j+1)} g^i u^\ell h^{2j+1} 
= \sum_{i = 0}^{n-1} \sum_{j = 0}^{k-1} \omega^{\ell n (2j+1) + 2\ell k i} u^\ell g^i h^{2j+1}.
\end{align*}
We now consider the exponent of $\omega$. Working modulo $2nk$, we have
\begin{align*}
\ell n (2j+1) + 2\ell k i 
&\equiv \ell(n (2j+1) + 2 k i)(an+bk) 
\equiv \ell ((2j+1)an^2 + 2bk^2 i + bnk) \\
&\equiv \ell ((2j+1)(m+bk)n + 2(an-m)k i + bnk) \\
&\equiv \ell (mn(2j+1) - 2mki + 2bnk) \\
&\equiv m \ell (n(2j+1) - 2ki).
\end{align*}
Therefore 
\begin{align*}
H_0 v^\ell = \sum_{i = 0}^{n-1} \sum_{j = 0}^{k-1} \omega^{m \ell (n(2j+1) - 2ki)} u^\ell g^i h^{2j+1} = u^\ell H_{m\ell},
\end{align*}
as claimed. \qedhere
\end{enumerate}
\end{proof}

\indent Our goal now is to show that both $u^s$ and $v^s$ lie in $\langle \overline{g} \rangle$ for some $s \gg 0$, which will allow us to show that $(A \hash G)/ \langle \overline{g} \rangle$ is finite dimensional. As a first step towards this, we have the following lemma, where we will often invoke the properties established above without mention.

\begin{lem} \label{sumlem}
\leavevmode
\begin{enumerate}[{\normalfont (1)},leftmargin=*,topsep=0pt,itemsep=2pt]
\item $(u^{nk} + v^{nk})G_0 \in \langle \overline{g} \rangle$ or $(u^{nk} - v^{nk})G_0 \in \langle \overline{g} \rangle$.
\item Let $m$ be as in Lemma \ref{Gprops} \emph{(6)}. Fix $\ell \in \{ 1, \dots, nk-1 \}$, and let $r \in \{ 1, \dots, 2nk-1 \}$ be such that $m\ell \equiv r \text{ mod } 2nk$. Then $(u^{\ell+r} + v^{\ell+r})G_\ell \in \langle \overline{g} \rangle$ or $(u^{\ell+r} - v^{\ell+r})G_\ell \in \langle \overline{g} \rangle$.
\end{enumerate}
\end{lem}
\begin{proof} \leavevmode
\begin{enumerate}[{\normalfont (1)},wide=0pt,topsep=0pt,itemsep=2pt]
\item Both of the following elements lie in $\langle \overline{g} \rangle$:
\begin{gather*}
\overline{g} u^{nk} = G_0 u^{nk} + H_0 u^{nk} = u^{nk} G_{n,k} + v^{nk} H_{nk} = u^{nk} G_{0} + \delta v^{nk} H_0, \\
v^{nk} \overline{g} = v^{nk} G_{0} + v^{nk} H_0,
\end{gather*}
where $\delta \in \{\pm 1\}$. Adding or subtracting one from the other, depending on the sign of $\delta$, we see that the claim holds.
\item Again, the following elements both lie in $\langle \overline{g} \rangle$:
\begin{gather*}
u^r \overline{g} u^{\ell} = u^r G_0 u^{\ell} + u^r H_0 u^{\ell} = u^{\ell+r} G_{\ell} + u^{r} v^\ell H_{\ell}, \\
v^\ell \overline{g} v^{r} = v^\ell G_0 u^r + v^\ell H_0 u^{r} = v^{\ell+r} G_{mr} + \delta u^{r} v^\ell H_{mr},
\end{gather*}
for some $\delta \in \{\pm 1\}$ (the sign of $\delta$ depends on $r$ as in Lemma \ref{Gprops} (2), as well as the fact that $vu = -uv$). Since $m = an-bk$ as in the statement of Lemma \ref{Gprops} (6), working modulo $2nk$ we have 
\begin{align*}
m^2 \equiv (an-bk)^2 \equiv a^2 n^2 - 2nkab + b^2 k^2 \equiv a^2 n^2 + 2nkab + b^2 k^2 \equiv (an+bk)^2 \equiv 1.
\end{align*}
Therefore $mr \equiv m^2 \ell \equiv \ell$, so that $v^\ell \overline{g} v^{r} = v^{\ell+r} G_{\ell} + \delta u^{r} v^\ell H_{\ell}$. The claim now follows by adding or subtracting the above two elements, depending on the sign of $\delta$. \qedhere
\end{enumerate}
\end{proof}

\begin{rem} \label{lplusreven}
The above proof shows that $2nk$ divides $m^2-1$, and hence $m$ must be odd. Therefore we find that $\ell$ and $r$, as in part (2) of the statement, have the same parity, and hence the exponent $\ell+r$ is even. This will be used later.
\end{rem}

\begin{lem} \label{tlem} \leavevmode
\begin{enumerate}[{\normalfont (1)},leftmargin=*,topsep=0pt,itemsep=2pt]
\item $u^k v^k G_0 \in \langle \overline{g} \rangle$.
\item $u^{(n+1)k} v^{(n+1)k} \in \langle \overline{g} \rangle$.
\end{enumerate}
\end{lem}
\begin{proof} \leavevmode
\begin{enumerate}[{\normalfont (1)},wide=0pt,topsep=0pt,itemsep=2pt]
\item First suppose that $k$ is odd. Then
\begin{align*}
\overline{g} u^k v^k 
&= \sum_{i=0}^{n-1} \sum_{j=0}^{2k-1} g^i h^{2j} u^k v^k + \sum_{i=0}^{n-1} \sum_{j=0}^{2k-1} g^i h^{2j+1} u^k v^k \\
&= \sum_{i=0}^{n-1} \sum_{j=0}^{2k-1} \omega^{4nkj} g^i u^k v^k h^{2j} + \sum_{i=0}^{n-1} \sum_{j=0}^{2k-1} \omega^{2nk(2j+1)} g^i v^k u^k h^{2j+1} \\
&= \sum_{i=0}^{n-1} \sum_{j=0}^{2k-1} \omega^{4nkj} u^k v^k g^i h^{2j} + \sum_{i=0}^{n-1} \sum_{j=0}^{2k-1} \omega^{2nk(2j+1)} v^k u^k g^i h^{2j+1} \\
&= u^k v^k G_0 + v^k u^k H_0 \\
&= u^k v^k G_0 - u^k v^k H_0,
\end{align*}
where the final equality follows from the fact that $u^k$ and $v^k$ $(-1)$-commute since $k$ is odd. If instead $k$ is even (and hence $k \equiv 0 \text{ mod } 4$), then a similar calculation shows that
\begin{align*}
u^{k/2} v^{k/2} \overline{g} u^{k/2} v^{k/2} 
&= u^{k/2} v^{k/2} \hspace{-3pt} \left(\sum_{i=0}^{n-1} \sum_{j=0}^{2k-1} \omega^{nkj} u^{k/2} v^{k/2} g^i h^{2j} + \sum_{i=0}^{n-1} \sum_{j=0}^{2k-1} \omega^{2nkj + nk} v^{k/2} u^{k/2} g^i h^{2j+1} \right) \\
&= u^{k/2} v^{k/2}  \left( u^{k/2} v^{k/2} G_0 + \omega^{nk} v^{k/2} u^{k/2} H_0 \right) \\
&= u^{k} v^{k} G_0 - u^{k} v^{k} H_0,
\end{align*}
where the final equality holds since $\omega^{nk} = -1$, and $u^{k/2}$ and $v^{k/2}$ commute since $\frac{k}{2}$ is even. Therefore in either case we have $u^k v^k (G_0-H_0) \in \langle \overline{g} \rangle$. Since obviously $u^k v^k (G_0+H_0) = u^k v^k \overline{g} \in \langle \overline{g} \rangle$, the result follows.
\item Fixing $\ell \in \{ 0, \dots, nk-1 \}$, by part (1) and by Lemma \ref{Gprops} (5), we have
\begin{align*}
\langle \overline{g} \rangle \ni u^{nk-\ell} v^{nk} u^k v^k G_0 u^\ell = \pm u^{(n+1)k} v^{(n+1)k} G_\ell.
\end{align*}
By Lemma \ref{Gprops} (3), the sum of these elements (with appropriate sign choices) is equal to $u^{(n+1)k} v^{(n+1)k}$, and hence this element lies in $\langle \overline{g} \rangle$, as claimed. \qedhere
\end{enumerate}
\end{proof}

We are now in a position to prove the main result of this subsection:

\begin{thm} \label{nondiagausthm}
Suppose that $A = \Bbbk_{-1}[u,v]$ and $G = G_{n,k}$ are as above. Then $(A \hash G)/\langle \overline{g} \rangle$ is finite dimensional. In particular, the Auslander map is an isomorphism for the pair $(A,G)$.
\end{thm}
\begin{proof}
We first show that $u^s-v^s \in \langle \overline{g} \rangle$ for some positive integer $s$. By Lemma \ref{tlem} (2), there exists an even integer $t$ such that $u^t v^t \in \langle \overline{g} \rangle$ (and hence $u^t v^t G_\ell \in \langle \overline{g} \rangle$ for all $\ell \in \{ 0, \dots, nk-1 \}$). As in Lemma \ref{sumlem} (2), if we fix $\ell \in \{ 1, \dots, nk-1 \}$, then there exists a positive integer $r$ such that $(u^{\ell+r} + v^{\ell+r})G_\ell \in \langle \overline{g} \rangle$ or $(u^{\ell+r} - v^{\ell+r})G_\ell \in \langle \overline{g} \rangle$. Multiplying this element on the left by $u^{\ell+r} - v^{\ell+r}$ or $u^{\ell+r} + v^{\ell+r}$, respectively, and noting that $u^{\ell+r}$ and $v^{\ell+r}$ are both central in $A$ since $\ell + r$ is even (Remark \ref{lplusreven}), we have $(u^{2(\ell+r)} - v^{2(\ell+r)})G_\ell \in \langle \overline{g} \rangle$. Now multiplying this element on the left by $(u^{2(\ell+r)} + v^{2(\ell+r)})$ shows that $(u^{4(\ell+r)} - v^{4(\ell+r)}) \in \langle \overline{g} \rangle$. 
Repeating this process in the obvious way, for each $\ell$ we can find an even integer $s_\ell \geqslant t$ such that $(u^{s_\ell} - v^{s_\ell}) G_\ell \in \langle \overline{g} \rangle$. Set 
\begin{align*}
s = t + \max \{ s_\ell \mid 0 \leqslant \ell \leqslant nk-1 \}.
\end{align*}
Now, for any $\ell \in \{ 0, \dots, nk-1 \}$, since $s_\ell \geqslant t$ and $s-s_\ell \geqslant t$, we have $u^{s-s_\ell} v^{s_\ell} G_\ell, v^{s-s_\ell} u^{s_\ell} G_\ell \in \langle \overline{g} \rangle$. Therefore
\begin{align*}
\langle \overline{g} \rangle &\ni (u^{s-s_\ell} + v^{s-s_\ell})(u^{s_\ell}-v^{s_\ell})G_\ell + u^{s-s_\ell} v^{s_\ell} G_\ell - v^{s-s_\ell} u^{s_\ell} G_\ell \\
&= u^s G_\ell - u^{s-s_\ell} v^{s_\ell} G_\ell + v^{s-s_\ell} u^{s_\ell} G_\ell - v^s G_\ell + u^{s-s_\ell} v^{s_\ell} G_\ell - v^{s-s_\ell} u^{s_\ell} G_\ell \\
&= (u^s-v^s) G_\ell.
\end{align*}
Summing over $\ell$ and applying Lemma \ref{Gprops} (3) shows that $u^s-v^s \in \langle \overline{g} \rangle$. \\
\indent Since $s \geqslant t$ we have $u^s v^s \in \langle \overline{g} \rangle$, and so it follows that 
\begin{align*}
u^{2s} = u^s(u^s-v^s) + u^s v^s \in \langle \overline{g} \rangle,
\end{align*}
and similarly $v^{2s} \in \langle \overline{g} \rangle$. Hence $A_{\geqslant 4s}$ is contained in $\langle \overline{g} \rangle$, and so the same is true of $(A \hash G)_{\geqslant 4s}$. In particular, $(A \hash G)/\langle \overline{g} \rangle$ is finite dimensional, as claimed.
\end{proof}

\subsection{Corollaries}
We conclude this section by noting some corollaries that follow from results in the literature. Following \cite{ueyama2013}, if $R$ is a noetherian graded algebra, then $R$ is a \emph{graded isolated singularity} if $\gldim(\text{qgr } R) < \infty$, where $\text{qgr } R = \gr R / \text{tors-}R$. If $R$ is commutative, then this definition generalises the usual commutative definition of $R$ being an isolated singularity. By \cite[Theorem 2.13, Lemma 2.12]{ample}, it follows that the algebras $A^G$ are graded isolated singularities. \\
\indent Since the Auslander map is an isomorphism for each of the pairs $(A,G)$ appearing in Theorem \ref{smallclassification}, we are also able to deduce the existence of bijections between objects in various module categories related to $A$ and $G$. We note that $M \in \gr A$ is called \emph{initial} if it is generated in degree 0 and $M_{<0} = 0$.

\begin{thm}[{\cite[Lemma 1.6, Proposition 2.3, Proposition 2.5, Corollary 4.5]{ckwz3}}]
Suppose that $(A,G)$ is a pair from Theorem \ref{smallclassification}. Then there are natural bijections between isomorphism classes of the following objects:
\begin{itemize}[topsep=0pt,leftmargin=25pt]
\item Simple $\Bbbk G$-modules;
\item Indecomposable direct summands of $A$ as a left $A^G$-module;
\item Indecomposable, finitely generated, projective, initial left $A \hash G$-modules;
\item Indecomposable, finitely generated, projective, initial left $\End_{A^G}(A)$-modules; and
\item Indecomposable maximal Cohen-Macaulay left $A^G$-modules, up to a degree shift.
\end{itemize}
\end{thm}

\section{A Presentation for Invariants of the Jordan Plane} \label{jordanplaneinvsec}
\noindent The remainder of this paper is devoted to providing presentations for the invariant rings $A^G$, where the pair $(A,G)$ is as in Theorem \ref{smallclassification}. The most interesting case is case (iii) when $A^G$ is noncommutative, which is analysed in Section \ref{noncomminvsec}. \\
\indent In this section, we write $A = \Bbbk_J[u,v]$, and let $G$ be a small subgroup of $\Autgr(A)$. By Lemma 3.2, $G = \frac{1}{n}(1,1)$ for some $n \geqslant 2$. It is easy to see that invariant ring $\Bbbk_J[u,v]^G$ consists of those elements lying in degrees which are multiples of $n$, so $\Bbbk_J[u,v]^G$ is simply the $n$th Veronese of $\Bbbk_J[u,v]$. The main goal of this section is to provide a presentation for this invariant ring. We will in fact show that these rings can be written as factors of AS regular algebras. \\
\indent To achieve this, we take advantage of the fact that $G$ also acts as a small group on the commutative ring $\Bbbk[u,v]$, and that the resulting invariant rings $\Bbbk[u,v]^G$ are well understood. In particular, by the results in Section \ref{typeApres}, if we write
\begin{align*}
x_i \coloneqq u^{n-i} v^i, \quad \text{for } 0 \leqslant i \leqslant n,
\end{align*}
then these elements generate $\Bbbk[u,v]^G$. Moreover, a minimal set of relations between the $x_i$ is given by
\begin{align}
x_i x_j = x_{i-1} x_{j+1} \text{ for all } i,j \text{ satisfying } 1 \leqslant i \leqslant j < n. \label{xirels}
\end{align}
\indent We note that we will occasionally swap between viewing $u$ and $v$ as elements of $\Bbbk_J[u,v]$ and viewing them as elements of the commutative ring $\Bbbk[u,v]$, but we will always make it clear which ring we are working inside. \\
\indent For the remainder of this section, for $\alpha \in \Bbbk$ and $k \in \mathbb{N}$ we use the notation
\begin{gather*}
\binom{\alpha}{k} \coloneqq \frac{\alpha(\alpha-1)\dots (\alpha-k+1)}{k!}.
\end{gather*}
If $\alpha$ is an integer, then this definition agrees with the usual definition of the binomial coefficient when $\alpha \geqslant k$, while $\binom{\alpha}{k} = 0$ when $0 \leqslant \alpha < k$. One useful identity that this notation satisfies is
\begin{gather}
(-1)^k \binom{\alpha}{k} = \binom{k-\alpha-1}{k}, \label{usefulidentity}
\end{gather}
which we will use later on. \\
\indent We first need to write down a formula for the commutator of monomials in $\Bbbk_J[u,v]$ in terms of the standard basis $\{ u^i v^j \mid i,j \geqslant 0 \}$.

\begin{lem} \label{jordanrel}
In the Jordan plane $\Bbbk_J[u,v]$, the following relations hold:
\begin{enumerate}[{\normalfont (1)},leftmargin=*,topsep=0pt,itemsep=2pt]
\item For all positive integers $j$,
\begin{align*}
vu^j = u^j v + j u^{j+1}.
\end{align*}
\item For all positive integers $i$,
\begin{align*}
v^i u = \sum_{k=0}^i k! \binom{i}{k} u^{k+1} v^{i-k}.
\end{align*}
\item For all positive integers $i$ and $j$,
\begin{align*}
v^i u^j = \sum_{k=0}^i k! \binom{j+k-1}{k} \binom{i}{k} u^{j+k} v^{i-k}.
\end{align*}
\end{enumerate}
\end{lem}
\begin{proof}
Parts (1) and (2) can be found in \cite{iyudu}. Presumably part (3) is also known, but we have been unable to find a reference; for completeness, we provide the proof, which is routine but lengthy. \\
\indent We proceed by induction on $i$. The case $i=1$ is the statement of part (2), so now fix $i \geqslant 2$ and suppose that the result holds for all smaller $i$. Then
\begin{align*}
v^i u^j &= v \cdot v^{i-1} u^j \\
&= v \left( \sum_{k=0}^{i-1} k! \binom{j+k-1}{k} \binom{i-1}{k} u^{j+k} v^{i-1-k} \right) \\
&= \sum_{k=0}^{i-1} k! \binom{j+k-1}{k} \binom{i-1}{k} \big( u^{j+k}v + (j+k)u^{j+k+1} \big) v^{i-1-k} \\
&= \sum_{k=0}^{i-1} k! \binom{j+k-1}{k} \binom{i-1}{k} u^{j+k} v^{i-k} + \sum_{k=0}^{i-1} k! (j+k) \binom{j+k-1}{k} \binom{i-1}{k} u^{j+k+1} v^{i-(k+1)} \\
&= u^j v^i + \sum_{k=1}^{i-1} k! \binom{j+k-1}{k} \binom{i-1}{k} u^{j+k} v^{i-k} \\
& \hspace{35pt} + \sum_{k=1}^{i-1} (k-1)! (j+k-1) \binom{j+k-2}{k-1} \binom{i-1}{k-1} u^{j+k} v^{i-k} \\
& \hspace{35pt} + (i-1)! (i+j-1) \binom{i+j-2}{i-1} u^{i+j} \\
&= u^j v^i + \sum_{k=1}^{i-1} \left( k! \binom{j+k-1}{k} \binom{i-1}{k} + (k-1)! (j+k-1) \binom{j+k-2}{k-1} \binom{i-1}{k-1} \right) u^{j+k} v^{i-k} \\
& \hspace{35pt} + i! \binom{i+j-1}{i} u^{i+j} \\
&= u^j v^i + \sum_{k=1}^{i-1} \left( k! \binom{j+k-1}{k} \binom{i-1}{k} + k! \binom{j+k-1}{k} \binom{i-1}{k-1} \right) u^{j+k} v^{i-k} \\
& \hspace{35pt} + i! \binom{i+j-1}{i} u^{i+j} \\
&= u^j v^i + \sum_{k=1}^{i-1} k! \binom{j+k-1}{k} \left( \binom{i-1}{k} + \binom{i-1}{k-1} \right) u^{j+k} v^{i-k} + i! \binom{i+j-1}{i} u^{i+j} \\
&= u^j v^i + \sum_{k=1}^{i-1} k! \binom{j+k-1}{k} \binom{i}{k} u^{j+k} v^{i-k} + i! \binom{i+j-1}{i} u^{i+j} \\
&= \sum_{k=0}^{i} k! \binom{j+k-1}{k} \binom{i}{k} u^{j+k} v^{i-k},
\end{align*}
where the second equality follows from the inductive hypothesis, the third equality follows from part (1), and where we have suppressed a number of routine binomial coefficient calculations. Hence the result follows by induction.
\end{proof}

Now fix $n$ and let $G = \frac{1}{n}(1,1)$. As with the polynomial ring $\Bbbk[u,v]$, we write
\begin{align*}
x_i \coloneqq u^{n-i} v^i, \quad \text{for } 0 \leqslant i \leqslant n,
\end{align*}
now viewed as elements of $\Bbbk_J[u,v]$. It will actually be more convenient to work with certain scalar multiples of the $x_i$; to this end, we also define
\begin{align}
y_i \coloneqq \frac{(-1)^i}{i!} x_i = \frac{(-1)^i}{i!} u^{n-i} v^i, \quad \text{for } 0 \leqslant i \leqslant n. \label{jordangens}
\end{align}
We first verify that these elements generate the invariant ring $\Bbbk_J[u,v]^G$, which is straightforward:

\begin{lem} \label{jordaninvgen}
The elements $y_i$ for $0 \leqslant i \leqslant n$ (or equivalently, the $x_i$) generate $\Bbbk_J[u,v]^G$.
\end{lem}
\begin{proof}
We prove the statement for the $x_i$. As observed previously, $\Bbbk_J[u,v]^G$ is simply the $n$th Veronese of $\Bbbk_J[u,v]$, so it suffices to show that we can write any element of the form $u^i v^j$, with $i+j \equiv 0 \text{ mod } n$, as a product of the $x_i$. By Euclid's algorithm, $i = an + r$ and $j = bn + s$ for some integers $a,b,r,s$, where $0 \leqslant r,s < n$. Therefore $u^i v^j = (u^n)^a u^r v^s (v^n)^b$. Working modulo $n$, $0 \equiv i+j \equiv (an+r) + (bn+s) \equiv r+s$, and since $0 \leqslant r,s < n$, it follows that either $r=0=s$ or $r+s=n$. In the former case, $u^i v^j = x_0^a x_n^b$, while in the latter case, $u^i v^j = x_0^a x_s x_n^b$. It follows that the $x_i$ generate $\Bbbk_J[u,v]^G$.
\end{proof}

We now wish to give a formula for the commutator $y_j y_i - y_i y_j$. We will make use of the Chu-Vandermonde identity, which states that for any $\alpha,\beta \in \Bbbk$ and any non-negative integer $n$,
\begin{align}
\sum_{k=0}^n \binom{\alpha}{k} \binom{\beta}{n-k} = \binom{\alpha+\beta}{n} \label{vandermonde}.
\end{align}

\begin{prop} \label{kJcommrels}
In $\Bbbk_J[u,v]$,
\begin{align}
\sum_{k=0}^{j} \binom{n-i}{k} y_{j-k} y_{i} = \sum_{\ell=0}^{i} \binom{n-j}{\ell} y_{i-\ell} y_{j} \qquad \text{ for } j > i. \label{xjxirel}
\end{align}
\end{prop}

\begin{rem} \label{altidentity}
Observe that the relation (\ref{xjxirel}), can be rewritten as
\begin{align*}
y_j y_i - y_i y_j = \sum_{\ell=1}^{i} \binom{n-j}{\ell} y_{i-\ell} y_{j} - \sum_{k=1}^{j} \binom{n-i}{k} y_{j-k} y_{i},
\end{align*}
and so we do have a formula for the commutator $y_j y_i - y_i y_j$.
\end{rem}

\begin{proof}[Proof of Proposition \ref{kJcommrels}]
We claim that both sides of (\ref{xjxirel}) are equal to
\begin{align*}
\frac{(-1)^{i+j}}{i! j!} u^{2n-i-j} v^{i+j}.
\end{align*}
First consider the left hand side of (\ref{xjxirel}). Using Lemma \ref{jordanrel}, we have
\begin{align*}
\sum_{k=0}^{j} & \binom{n-i}{k}  y_{j-k} y_{i} 
= \sum_{k=0}^{j} \frac{(-1)^{i+j-k}}{i! (j-k)!} \binom{n-i}{k} u^{n-j+k} v^{j-k} u^{n-i} v^i \\
&= \sum_{k=0}^{j} \sum_{m=0}^{j-k} (-1)^{i+j-k} \frac{m!}{i! (j-k)!} \binom{n-i}{k} \binom{n-i+m-1}{m} \binom{j-k}{m} u^{2n-i-j+k+m} v^{i+j-k-m} \\
&= \sum_{k=0}^{j} \sum_{m=0}^{j-k} (-1)^{i+j-k} \frac{1}{i! (j-k-m)!} \binom{n-i}{k} \binom{n-i+m-1}{m} u^{2n-i-j+k+m} v^{i+j-k-m} \\
&= \sum_{r=0}^j (-1)^{i+j} \frac{1}{i! (j-r)!} \sum_{\genfrac{}{}{0pt}{}{k+m=r}{0 \leqslant k \leqslant j}} (-1)^{-k} \binom{n-i}{k} \binom{n-i+m-1}{m} u^{2n-i-j+r} v^{i+j-r} \\
&= \sum_{r=0}^j (-1)^{i+j+r} \frac{1}{i! (j-r)!} \sum_{k=0}^r (-1)^{r-k} \binom{n-i}{k} \binom{n-i+r-k-1}{r-k} u^{2n-i-j+r} v^{i+j-r} \\
&\stackrel{(\ref{usefulidentity})}{=} \sum_{r=0}^j  (-1)^{i+j+r} \frac{1}{i! (j-r)!} \sum_{k=0}^r \binom{n-i}{k} \binom{i-n}{r-k} u^{2n-i-j+r} v^{i+j-r} \\
&\stackrel{(\ref{vandermonde})}{=}\sum_{r=0}^j (-1)^{i+j+r} \frac{1}{i! (j-r)!} \binom{0}{r} u^{2n-i-j+r} v^{i+j-r} \\
&= \frac{(-1)^{i+j}}{i! j!} u^{2n-i-j} v^{i+j}.
\end{align*}
An entirely similar calculation shows that 
\begin{align*}
\sum_{\ell=0}^{i} \binom{n-j}{\ell} y_{i-\ell} y_{j} = \frac{(-1)^{i+j}}{i! j!} u^{2n-i-j} v^{i+j},
\end{align*}
and so we have the claimed equality.
\end{proof}

In addition to the above relations, one should expect to have relations which are similar to those given in (\ref{xirels}). This is indeed the case, as we now show.

\begin{prop} \label{relsinkJ}
In $\Bbbk_J[u,v]$,
\begin{align*}
i y_i y_j = (j+1) y_{i-1} y_{j+1} - \big(n-1-(j-i)\big) y_{i-1} y_{j} \text{ for all } i,j \text{ satisfying } 1 \leqslant i \leqslant j < n.
\end{align*}
\end{prop}
\begin{proof}
It is easy to see that it is equivalent to show
\begin{align*}
x_i x_j = x_{i-1} x_{j+1} + \big(n-1-(j-i)\big) x_{i-1} x_{j} \text{ for all } i,j \text{ satisfying } 1 \leqslant i \leqslant j < n.
\end{align*}
We verify this equality using Lemma \ref{jordanrel} (3). Again, we suppress some of the routine calculations involving binomial coefficients. We have
\begin{align*}
x_{i-1} x_{j+1} &= u^{n-i+1} v^{i-1} u^{n-j-1} v^{j+1} = \sum_{k=0}^{i-1} k! \binom{n-j+k-2}{k} \binom{i-1}{k} u^{2n-i-j+k} v^{i+j-k},
\end{align*}
and similarly,
\begin{align*}
x_{i-1} x_{j} &= \sum_{k=0}^{i-1} k! \binom{n-j+k-1}{k} \binom{i-1}{k} u^{2n-i-j+k+1} v^{i+j-k-1}, \\
&= \sum_{k=1}^{i} (k-1)! \binom{n-j+k-2}{k-1} \binom{i-1}{k-1} u^{2n-i-j+k} v^{i+j-k}.
\end{align*}
Therefore
\begin{align*}
x_{i-1} & x_{j+1} + \big(n-1 - (j-1)\big) x_{i-1} x_j \\ 
&= u^{2n-i-j} v^{i+j} + \sum_{k=1}^{i-1} c_k u^{2n-i-j+k} v^{i+j-k} + \big(n-1-(j-i)\big) (i-1)! \binom{n-j+i-2}{i-1} u^{2n-j} v^j,
\end{align*}
where 
\begin{align*}
c_{k} = k! \binom{n-j+k-2}{k} \binom{i-1}{k} + \big(n-1-(j-i)\big) (k-1)! \binom{n-j+k-2}{k-1} \binom{i-1}{k-1}.
\end{align*}
It is easily shown that
\begin{align*}
c_{k} = k! \binom{n-j+k-1}{k} \binom{i}{k}
\end{align*}
for all $1 \leqslant k \leqslant i-1$, and that
\begin{align*}
\big(n-1-(j-i)\big) (i-1)! \binom{n-j+i-2}{i-1} = i! \binom{n-j+i-1}{i}.
\end{align*}
Hence 
\begin{align*}
x_{i-1} & x_{j+1} + \big(n-1 - (j-1)\big) x_{i-1} x_j \\ 
&= u^{2n-i-j} v^{i+j} + \sum_{k=1}^{i-1} k! \binom{n-j+k-1}{k} \binom{i}{k} u^{2n-i-j+k} v^{i+j-k} + i! \binom{n-j+i-1}{i} u^{2n-j} v^j \\
&= \sum_{k=0}^{i} k! \binom{n-j+k-1}{k} \binom{i}{k} u^{2n-i-j+k} v^{i+j-k}.
\end{align*}
This last expression is easily seen to be equal to $x_{i} x_{j}$, whence the result.
\end{proof}

This allows us to give a presentation for the invariant ring $\Bbbk_J[u,v]^G$. The proof uses the theory of (noncommutative) Groebner bases, so we briefly recall some terminology and results. We order the words in the free algebra $\Bbbk \langle X_i \rangle$ using degree lexicographic order with $X_0 < X_1 < \dots < X_n$. Given an ideal $K$ of $\Bbbk \langle X_i \rangle$, we write $\text{Lead}(K)$ for the ideal generated by all the leading words of elements of $K$. A \emph{Groebner basis} of $K$ is a subset of $K$ such that the leading terms of this subset generate $\text{Lead}(K)$. In this case, we have $\hilb \Bbbk \langle X_i \rangle/K = \hilb \Bbbk \langle X_i \rangle/\text{Lead}(K)$.

\begin{thm} \label{jordaninvpres}
Let $n \geqslant 2$ and $G = \frac{1}{n}(1,1)$. Then $\Bbbk_J[u,v]^G$ has the presentation
\begin{align*}
\Bbbk_J[u,v]^G \cong \frac{\Bbbk \langle X_0, X_1, \dots, X_n \rangle}{I}
\end{align*}
where the ideal of relations $I$ is generated by
\begin{gather*}
\sum_{k=0}^{j} \binom{n-i}{k} X_{j-k} X_{i} = \sum_{\ell=0}^{i} \binom{n-j}{\ell} X_{i-\ell} X_{j} \qquad \text{ for } j > i \\
i X_i X_j = (j+1) X_{i-1} X_{j+1} - \big(n-1-(j-i)\big) X_{i-1} X_{j} \text{ for all } i,j \text{ satisfying } 1 \leqslant i \leqslant j < n.
\end{gather*}
\end{thm}
\begin{proof}
Writing $y_i = \frac{(-1)^i}{i!} u^{n-i} v^i$ for the generators of $\Bbbk_J[u,v]^G$, Propositions \ref{kJcommrels} and \ref{relsinkJ} show that there is a surjection
\begin{align}
\frac{\Bbbk \langle X_0, X_1, \dots, X_n \rangle}{I} \twoheadrightarrow \Bbbk_J[u,v]^G, \quad X_i \mapsto y_i. \label{Jsurj}
\end{align}
To show that this is in fact an isomorphism, it suffices to show that $\Bbbk \langle X_i \rangle/I$ and $\Bbbk_J[u,v]^G$ have the same Hilbert series. Since the above map is a surjection, certainly 
\begin{align*}
\hilb \Bbbk \langle X_i \rangle /I \geqslant \hilb \Bbbk_J[u,v]^G.
\end{align*}
\indent By equation (\ref{xirels}), we know that $\Bbbk[u,v]^G$ can be presented as
\begin{align*}
\Bbbk[u,v]^G \cong \frac{\Bbbk \langle X_0, X_1, \dots, X_n \rangle}{J},
\end{align*}
where $J$ is the ideal of relations with generators
\begin{gather*}
X_j X_i = X_i X_j \text{ for } j > i, \\
X_i X_j = X_{i-1} X_{j+1} \text{ for all } i,j \text{ satisfying } 1 \leqslant i \leqslant j < n.
\end{gather*}
Here we present the commutative ring $\Bbbk[u,v]^G$ as the factor of a free algebra with, in particular, the obvious commutativity relations $X_j X_i = X_i X_j$, so as to more easily draw comparisons with our claimed presentation for $\Bbbk_J[u,v]^G$. With respect to degree lexicographic ordering, the elements on the left hand side of the relations above are the leading words. The generators of $J$ form a Groebner basis, so
\begin{align*}
\hilb \Bbbk \langle X_i \rangle/J = \hilb \Bbbk \langle X_i \rangle/\text{Lead}(J).
\end{align*}
\indent Now observe that the relations which define $I$ can be written as
\begin{gather*}
X_j X_i = -\sum_{k=1}^{j} \binom{n-i}{k} X_{j-k} X_{i} + \sum_{\ell=0}^{i} \binom{n-j}{\ell} X_{i-\ell} X_{j} \qquad \text{ for } j > i \\
X_i X_j = \frac{j+1}{i} X_{i-1} X_{j+1} - \frac{n-1-(j-i)}{i} X_{i-1} X_{j} \text{ for all } i,j \text{ satisfying } 1 \leqslant i \leqslant j < n.
\end{gather*}
The leading words are those which appear on the left hand side of each equality, and these are precisely the leading words for the relations in $J$. Therefore if we find a Groebner basis for $J$, it follows that $\text{Lead}(I) \supseteq \text{Lead}(J)$. Hence we have 
\begin{align*}
\hilb \Bbbk \langle X_i \rangle/I = \hilb \Bbbk \langle X_i \rangle/\text{Lead}(I) \leqslant \hilb \Bbbk \langle X_i \rangle/\text{Lead}(J) = \hilb \Bbbk \langle X_i \rangle/J.
\end{align*}
However, since $\Bbbk[u,v]^G$ and $\Bbbk_J[u,v]^G$ have the same Hilbert series by Molien's Theorem, and we have an inequality $\Bbbk \langle X_i \rangle/J \cong \Bbbk[u,v]^G$, we find 
\begin{align*}
\hilb \Bbbk \langle X_i \rangle/I \leqslant \hilb \Bbbk \langle X_i \rangle/J = \hilb \Bbbk[u,v]^G = \hilb \Bbbk_J[u,v]^G.
\end{align*}
Therefore $\hilb \Bbbk \langle X_i \rangle/I = \hilb \Bbbk_J[u,v]^G$ and so the map in (\ref{Jsurj}) is an isomorphism.
\end{proof}

\begin{cor} \label{jordancor}
Let $n \geqslant 2$ and $G = \frac{1}{n}(1,1)$. Then $\Bbbk_J[u,v]^G$ is a factor of an AS regular algebra of dimension $n+1$.
\end{cor}
\begin{proof}
Define $B \coloneqq \Bbbk[X_0, \dots, X_n]/K$, where $K$ is the ideal of relations generated by
\begin{align*}
\sum_{k=0}^{j} \binom{n-i}{k} X_{j-k} X_{i} = \sum_{\ell=0}^{i} \binom{n-j}{\ell} X_{i-\ell} X_{j} \qquad \text{ for } j > i.
\end{align*}
By \cite[Proposition 3.15]{lecoutre}, $B$ is isomorphic to an algebra denoted $R(n,a)$, where $a=-n$, which is an AS regular algebra of dimension $n+1$ by \cite[Proposition 3.8]{lecoutre}. The result then follows from Theorem \ref{jordaninvpres}.
\end{proof}

We end this section with an example:

\begin{example} \label{jordanexample}
In this example, we set $a = X_0, \hspace{2pt} b= X_1, \hspace{2pt} c = X_2, \hspace{2pt}  \dots$ when writing down presentations for the invariant rings $\Bbbk_J[u,v]^G$. When $n = 2$, the presentation for $\Bbbk_J[u,v]^G$ given in Theorem \ref{jordaninvpres} is
\begin{align*}
\Bbbk_J[u,v]^G \cong 
\frac{\Bbbk \langle a,b,c \rangle}{\left \langle
\begin{array}{cc}
ba + 2a^2 - ab, & ca + 2ba + a^2 - ac \\
cb + b^2 - bc, & b^2 - 2ac + ab
\end{array}
\right \rangle }.
\end{align*}
The algebra
\begin{align*}
B = 
\frac{\Bbbk \langle a,b,c \rangle}{\bigg \langle
\begin{array}{c}
\begin{array}{cc}
ba + 2a^2 - ab, & ca + 2ba + a^2 - ac
\end{array} \\
cb + b^2 - bc
\end{array}
\bigg \rangle },
\end{align*}
is AS regular, and $\Omega = b^2 - 2ac + ab$ is a regular normal element in $B$. Observe that an equivalent presentation for $\Bbbk_J[u,v]^G$ is the following:
\begin{align*}
\Bbbk_J[u,v]^G \cong 
\frac{\Bbbk \langle a,b,c \rangle}{\left \langle
\begin{array}{cc}
ba -ab + 2a^2, & ca - ac + 2ab - 3a^2 \\
cb - bc + b^2, & b^2 - 2ac + ab
\end{array}
\right \rangle },
\end{align*}
and that if we make the change of variables
\begin{align*}
a \mapsto a', \quad b \mapsto -b', \quad c \mapsto \frac{1}{2}c',
\end{align*}
then we obtain the presentation given in \cite[Theorem 5.2]{ckwz}. \\
\indent Now suppose that $n=3$. Then, by Theorem \ref{jordaninvpres},
\begin{align*}
\Bbbk_J[u,v]^G \cong 
\frac{\Bbbk \langle a,b,c,d \rangle}{
\sbox0{$
\begin{array}{ccc}
ba + 3a^2 - ab, & ca + 3ba + 3a^2 - ac, & da + 3ca + 3ba + a^2 - ad \\
cb + 2b^2 + ab - bc - ac, & db + 2cb + b^2 - bd, & dc + c^2 - cd \\
b^2 - 2ac + 2ab, & bc - 3ad + ac, & 2c^2 - 3bd + 2bc
\end{array}$}
\mathopen{\resizebox{1.2\width}{\ht0}{$\Bigg\langle$}}
\usebox{0}
\mathclose{\resizebox{1.2\width}{\ht0}{$\Bigg\rangle$}} }.
\end{align*}
If we omit the three relations in the last row, then the resulting algebra $B$ is AS regular. 
\end{example}

\section{Invariant Rings for Diagonal Subgroups of $\Autgr(\Bbbk_q[u,v])$} \label{kqinvsec}
\noindent In this section, we fix $A = \Bbbk_q[u,v]$ and 
\begin{align*}
G = \frac{1}{n} (1,a) = \bigg \langle \hspace{-3pt} \begin{pmatrix} \omega & 0 \\ 0 & \omega^a \end{pmatrix} \hspace{-3pt} \bigg \rangle,
\end{align*}
where $\omega^{n} = 1$, $1 \leqslant a < n$, and $\gcd(a,n) = 1$. We wish to write down a presentation for the invariant ring $A^G$. Since $G$ acts on $A$ in the same way as on $\Bbbk[u,v]$, one should expect the relations in $A^G$ to be $q$-deformed versions of those in $\Bbbk[u,v]^G$; we will show that this is in fact the case. The proof is quite straightforward if one uses the results in \cite{riemen}, with the only real difficulty stemming from keeping track of the powers of $q$ that appear in the relations. \\
\indent Recall from Section 2 that if we write
\begin{align*}
\frac{n}{n-a} = [\beta_1, \dots, \beta_{d-2}].
\end{align*}
and define two series of integers $i_1, \dots, i_d$ and $j_1, \dots, j_d$ as follows,
\begin{gather*}
\begin{array}{llllll}
i_1 = n, & i_2 = n-a & & \text{and} & & i_k = \beta_{k-2} i_{k-1} - i_{k-2} \text{ for } 3 \leqslant k \leqslant d, \\[2pt]
j_1 = 0, & j_2 = 1 & & \text{and} & & j_k = \beta_{k-2} j_{k-1} - j_{k-2} \text{ for } 3 \leqslant k \leqslant d,
\end{array}
\end{gather*}
\noindent then the invariant ring $\Bbbk[u,v]^{G}$ is minimally generated by the $d$ elements
\begin{align*}
x_k \coloneqq u^{i_k} v^{j_k}, \quad 1 \leqslant k \leqslant d.
\end{align*}
Moreover, a minimal set of relations between the $x_k$ is as follows:
\begin{gather}
\begin{array}{lcl}
x_{k-1} x_{k+1} = x_k^{\beta_{k-1}} & & \text{for } 2 \leqslant k \leqslant d-1, \\[6pt]
x_k x_\ell = x_{k+1}^{\beta_{k} - 1} x_{k+2}^{\beta_{k+1} - 2} \dots x_{\ell-2}^{\beta_{\ell-3} - 2}  x_{\ell-1}^{\beta_{\ell-2} - 1} & & \text{for }  2 \leqslant k+1 < \ell-1 \leqslant d-1.
\end{array} \label{diagrels}
\end{gather}
\indent Now view the $x_k$ as elements of $A^G$. (As in the previous subsection, we will often swap between viewing the $x_k$ as elements of $\Bbbk[u,v]^G$ and of $A^G$.) Using the same arguments as for $\Bbbk[u,v]^G$, one can show that the $x_k$ generate $A^G$. It remains to determine the relations between them. \\
\indent In contrast with the commutative case, when viewed as elements of $A^G$, in general the $x_k$ do not commute. Instead, we have the following:

\begin{lem}
Suppose that $1 \leqslant k < \ell \leqslant d$. Then
\begin{align*}
x_\ell x_k = q^{i_k j_\ell - i_\ell j_k} x_k x_\ell.
\end{align*}
\end{lem}
\begin{proof}
Since $vu = quv$, direct calculation using the definition of the $x_k$ gives
\begin{align*}
x_\ell x_k = u^{i_\ell} v^{j_\ell} u^{i_k} v^{j_k} = q^{i_k j_\ell} u^{i_k + i_\ell} v^{j_k + j_\ell}, \\
x_k x_\ell = u^{i_k} v^{j_k} u^{i_\ell} v^{j_\ell} = q^{i_\ell j_k} u^{i_k + i_\ell} v^{j_k + j_\ell}.
\end{align*}
Therefore,
\begin{align*}
x_\ell x_k = q^{i_k j_\ell - i_\ell j_k} q^{i_\ell j_k} u^{i_k + i_\ell} v^{j_k + j_\ell} = q^{i_k j_\ell - i_\ell j_k} x_k x_\ell, 
\end{align*}
as claimed.
\end{proof}

The remaining relations are $q$-deformed versions of those from (\ref{diagrels}). As mentioned previously, the only difficulty in establishing the relations is keeping track of the powers of $q$. To this end, we record an easy lemma, the proof of which is omitted:

\begin{lem} \label{kqidentities}
Let $a,b,c, a_1, \dots, a_m, b_1, \dots, b_m$ be non-negative integers. Then, in $A$, we have
\begin{gather*}
(u^a v^b)^c = q^{\frac{1}{2}abc(c-1)} u^{ac} v^{bc}, \\
(u^{a_1} v^{b_1}) \dots (u^{a_m} v^{b_m}) = q^{r} u^{\sum a_i} v^{\sum b_i}, \quad \text{where } r = \sum_{i=2}^m \sum_{j=1}^{i-1} a_i b_j.
\end{gather*}
\end{lem}

Using this, we can prove the following:

\begin{lem} \label{kqawkwardrels}
In $A$, the following relations hold:

\begin{gather*}
\begin{array}{lll}
x_k^{\beta_{k-1}} = q^{\frac{1}{2} i_k j_k \beta_{k-1}(\beta_{k-1}-1) - i_{k+1} j_{k-1}} x_{k-1} x_{k+1} & & \text{ for } 2 \leqslant k \leqslant d-1, \\[2pt]
q^{r_{k \ell}} x_k x_\ell = x_{k+1}^{\beta_{k} - 1} x_{k+2}^{\beta_{k+1} - 2} \dots x_{\ell-2}^{\beta_{\ell-3} - 2}  x_{\ell-1}^{\beta_{\ell-2} - 1} & & \text{ for } 2 \leqslant k+1 < \ell-1 \leqslant d-1,
\end{array}
\end{gather*}
where
\begin{align*}
r_{k \ell} = \sum_{m=k+1}^{\ell-1} &\frac{1}{2} i_m j_m (\beta_{m-1} - 2 + \delta_{m,k+1} + \delta_{m,\ell-1})(\beta_{m-1} - 3 + \delta_{m,k+1} + \delta_{m,\ell-1}) \\
&+ \sum_{m=k+2}^{\ell-1} \sum_{r=k+1}^{m-1} i_m j_r (\beta_{m-1} - 2 + \delta_{m,\ell-1}) (\beta_{r-1} - 2 + \delta_{r,k+1}) \hspace{7pt} - \hspace{7pt} i_\ell j_k,
\end{align*}
and where $\delta_{s,t}$ is the Kronecker delta.
\end{lem}
\begin{proof}
To prove the first relation, observe that, by Lemma \ref{kqidentities},
\begin{align*}
x_k^{\beta_{k-1}} = (u^{i_k} v^{j_k})^{\beta_{k-1}} = q^{\frac{1}{2} i_k j_k \beta_{k-1}(\beta_{k-1}-1)} u^{i_k \beta_{k-1}} u^{j_k \beta_{k-1}}.
\end{align*}
On the other hand, 
\begin{align*}
x_{k-1} x_{k+1} = u^{i_{k-1}} v^{j_{k-1}} u^{i_{k+1}} v^{j_{k+1}} = q^{i_{k+1} j_{k-1}} u^{i_{k-1} + i_{k+1}} v^{j_{k-1} + j_{k+1}}
\end{align*}
Now, viewing the $x_i$ as elements of $\Bbbk[u,v]^G$, we know that $x_k^{\beta_{k-1}} =  x_{k-1} x_{k+1}$, and hence 
\begin{align*}
i_k \beta_{k-1} = i_{k-1} + i_{k+1} \qquad \text{and} \qquad j_k \beta_{k-1} = j_{k-1} + j_{k+1}.
\end{align*}
(In fact, these identities follow directly from the definition of the $i_k$ and $j_k$ given at the beginning of this section.) Therefore
\begin{align*}
x_k^{\beta_{k-1}} &= q^{\frac{1}{2} i_k j_k \beta_{k-1}(\beta_{k-1}-1)} u^{i_k \beta_{k-1}} u^{j_k \beta_{k-1}} \\
&= q^{\frac{1}{2} i_k j_k \beta_{k-1}(\beta_{k-1}-1) - i_{k+1} j_{k-1}} q^{i_{k+1} j_{k-1}} u^{i_{k-1} + i_{k+1}} v^{j_{k-1} + j_{k+1}} \\
&= q^{\frac{1}{2} i_k j_k \beta_{k-1}(\beta_{k-1}-1) - i_{k+1} j_{k-1}} x_{k-1} x_{k+1},
\end{align*}
as claimed. \\
\indent For the second relation, fix $k$ and $\ell$, and for ease of notation write
\begin{align*}
\gamma_m = \beta_{m-1} - 2 + \delta_{m,k+1} + \delta_{m,\ell-1}
\end{align*}
for $k+1 \leqslant m \leqslant \ell -1$. Notice that $\gamma_m$ is precisely the exponent of $x_m$ appearing on the right hand side of the claimed relation. Then, using Lemma \ref{kqidentities},
\begin{align*}
x&_{k+1}^{\gamma_{k+1}} x_{k+2}^{\gamma_{k+2}} \dots x_{\ell-2}^{\gamma_{\ell-2}}  x_{\ell-1}^{\gamma_{\ell-1}} \\
&= (u^{i_{k+1}} v^{j_{k+1}})^{\gamma_{k+1}} (u^{i_{k+2}} v^{j_{k+2}})^{\gamma_{k+2}} \dots (u^{i_{\ell-2}} v^{j_{\ell-2}})^{\gamma_{\ell-2}} (u^{i_{\ell-1}} v^{j_{\ell-1}})^{\gamma_{\ell-1}} \\
&= q^{a} (u^{i_{k+1}\gamma_{k+1}} v^{j_{k+1}\gamma_{k+1}}) (u^{i_{k+2}\gamma_{k+2}} v^{j_{k+2}\gamma_{k+2}}) \dots (u^{i_{\ell-2}\gamma_{\ell-2}} v^{j_{\ell-2}\gamma_{\ell-2}}) (u^{i_{\ell-1}\gamma_{\ell-1}} v^{j_{\ell-1}\gamma_{\ell-1}}) \\
&= q^a q^b u^c v^d,
\end{align*}
where
\begin{gather*}
\begin{array}{lll}
\displaystyle{a = \sum_{m=k+1}^{\ell-1} \frac{1}{2} i_m j_m \gamma_m (\gamma_m - 1),} & & \displaystyle{b = \sum_{m=k+2}^{\ell-1} \sum_{r=k+1}^{m-1} i_m j_r \gamma_m \gamma_r,} \\[15pt]
\displaystyle{c = \sum_{s=k+1}^{\ell-1} i_s \gamma_s,} & & \displaystyle{d = \sum_{s=k+1}^{\ell-1} j_s \gamma_s}.
\end{array}
\end{gather*}
Observe that $a+b-i_\ell j_k = r_{k \ell}$. On the other hand, $x_k x_\ell = q^{i_\ell j_k} u^{i_k + i_\ell} u^{j_k + j_\ell}$. Again, if we view the $x_i$ as elements of $\Bbbk[u,v]^G$, then we know that $x_k x_\ell = x_{k+1}^{\beta_{k} - 1} x_{k+2}^{\beta_{k+1} - 2} \dots x_{\ell-2}^{\beta_{\ell-3} - 2}  x_{\ell-1}^{\beta_{\ell-2} - 1}$, and hence we have the identities
\begin{align*}
c = i_k + i_\ell \qquad \text{and} \qquad d = j_k + j_\ell.
\end{align*}
Therefore, in $A^G$,
\begin{align*}
x_{k+1}^{\gamma_{k+1}} x_{k+2}^{\gamma_{k+2}} \dots x_{\ell-2}^{\gamma_{\ell-2}}  x_{\ell-1}^{\gamma_{\ell-1}} &= q^{a+b} u^c v^d 
= q^{a+b-i_\ell j_k} q^{i_\ell j_k} u^{i_k + i_\ell} u^{j_k + j_\ell} 
= q^{a+b-i_\ell j_k} x_k x_\ell,
\end{align*}
as claimed.
\end{proof}

This allows us to give a presentation for the invariant ring $A^G$. The proof of this result is similar to that of Theorem \ref{jordaninvpres}, and is therefore omitted.

\begin{thm} \label{kqinvariantspres}
Let $G = \frac{1}{n}(1,a)$, where $1 \leqslant a < n$ and $gcd(a,n) = 1$. Let $q \in \Bbbk^\times$. Then $\Bbbk_q[u,v]^G$ has the presentation
\begin{align*}
\Bbbk_q[u,v]^G \cong \frac{\Bbbk \langle x_1, \dots, x_d \rangle}{I}
\end{align*}
where the ideal of relations $I$ is generated by
\begin{gather*}
\begin{array}{lll}
x_\ell x_k = q^{i_k j_\ell - i_\ell j_k} x_k x_\ell & & \text{for } 1 \leqslant k < \ell \leqslant d, \\[5pt]
x_k^{\beta_{k-1}} = q^{\frac{1}{2} i_k j_k  \beta_{k-1}(\beta_{k-1}-1) - i_{k+1} j_{k-1}} x_{k-1} x_{k+1} & & \text{for } 2 \leqslant k \leqslant d-1, \\[5pt]
q^{r_{k \ell}} x_k x_\ell = x_{k+1}^{\beta_{k} - 1} x_{k+2}^{\beta_{k+1} - 2} \dots x_{\ell-2}^{\beta_{\ell-3} - 2}  x_{\ell-1}^{\beta_{\ell-2} - 1} & & \text{for }  2 \leqslant k+1 < \ell-1 \leqslant d-1,
\end{array}
\end{gather*}
where $r_{k\ell}$ is as in Lemma \ref{kqawkwardrels}.
\end{thm}

Since the factor of $\Bbbk \langle x_1, \dots, x_d \rangle$ by the ideal generated by the first line of relations in Theorem \ref{kqinvariantspres} is a quantum polynomial ring, which are well known to be AS regular, the following is immediate:

\begin{cor} \label{kqcor}
Let $G = \frac{1}{n}(1,a)$, where $1 \leqslant a < n$ and $gcd(a,n) = 1$. Let $q \in \Bbbk^\times$. Then $A^G$ is a factor of an AS regular algebra of dimension $d$.
\end{cor}

\section{Commutative Invariant Rings of the Form $\Bbbk_{-1}[u,v]^{G_{n,k}}$} \label{comminvsec}
\noindent In this section, let $n$ and $k$ be positive coprime integers where $k \not\equiv 2 \text{ mod } 4$, and let $G = G_{n,k} = \langle g,h \rangle$, where
\begin{align}
g = \begin{pmatrix} \omega^{2k} & 0 \\ 0 & \omega^{-2k} \end{pmatrix}, \qquad
h = \begin{pmatrix} 0 & \omega^n \\ \omega^n & 0 \end{pmatrix}, \label{Gnkaltpres}
\end{align}
and where $\omega$ is a primitive ($2nk$)th root of unity. We have already seen that the invariant ring $\Bbbk_{-1}[u,v]^G$ is commutative if and only if $n$ or $k$ is even, so we will assume that this is the case in this subsection. When this is the case, we will show that the rings $\Bbbk_{-1}[u,v]^G$ are commutative quotient singularities, and hence are well-understood. More precisely, we shall see that every type $\mathbb{D}$ singularity can be obtained in this way, as well as some type $\mathbb{A}$ singularities.

\subsection{Cyclic quotient singularities} We first consider the type $\mathbb{A}$ cases, which occur when $n \leqslant 2$. If $n=1$ then $k$ has to be even, and since $k \not\equiv 2 \text{ mod } 4$, it follows that $k \equiv 0 \text{\normalfont{ mod} } 4$. If instead $n=2$, then since $n$ and $k$ are coprime, it follows that $k$ is odd. 

\begin{prop} \label{nequals1}
Suppose that $n=1$ and $k$ is even, so $k \equiv 0 \text{\normalfont{ mod} } 4$. Then 
\begin{align*}
\Bbbk_{-1}[u,v]^{G_{n,k}} \cong \Bbbk[u,v]^{\frac{1}{2k}(1,k+1)},
\end{align*}
so that $\Bbbk_{-1}[u,v]^{G_{n,k}}$ is a cyclic quotient singularity.
\end{prop}
\begin{proof}
In this case $\omega^{2k} = 1$, so $g$ is the identity and $G_{n,k} = \langle h \rangle$. Observe that $h^k = \begin{psmallmatrix} -1 & 0 \\ 0 & -1 \end{psmallmatrix}$, which is the generator of $\frac{1}{2}(1,1)$. Working inside $\Bbbk_{-1}[u,v]$, set
\begin{align*}
x = \omega^{k/2} (u^2-v^2), \quad y = 2uv, \quad z = u^2 + v^2,
\end{align*}
which generate the commutative invariant ring $\Bbbk_{-1}[u,v]^{h^k}$ and satisfy the single relation $x^2+y^2+z^2 = 0$. We also have
\begin{align*}
h^{\frac{k}{2}+1} \cdot x &= \omega^{k+2}\omega^{k/2}(v^2-u^2) = \omega^2 \omega^{k/2} (u^2-v^2) = \omega^2 x, \\
h^{\frac{k}{2}+1} \cdot y &= \omega^{k+2}(2vu) = \omega^{k+2} \omega^k (2uv) = \omega^2 y, \\
h^{\frac{k}{2}+1} \cdot z &= \omega^{k+2}(v^2+u^2) = \omega^{k+2} z,
\end{align*}
where $h^{k/2+1}$ is also a generator for $G$. \\
\indent Now consider the group $\frac{1}{2k}(1,k+1)$ with generator $\gamma = \begin{psmallmatrix} \omega & 0 \\ 0 & \omega^{k+1} \end{psmallmatrix}$, acting on the commutative ring $\Bbbk[u,v]$ where, as above, $\omega^{2k} = 1$. Notice that $\gamma^k = \begin{psmallmatrix} -1 & 0 \\ 0 & -1 \end{psmallmatrix}$ and that the invariant ring $\Bbbk[u,v]^{\gamma^k}$ is generated by
\begin{align*}
x = \omega^{k/2} (u^2+v^2), \quad y = u^2-v^2, \quad z = 2uv,
\end{align*}
where we emphasise that these elements live inside $\Bbbk[u,v]$. These satisfy $x^2+y^2+z^2 = 0$ and
\begin{align*}
\gamma \cdot x = \omega^2 x , \quad
\gamma \cdot y = \omega^2 y , \quad 
\gamma \cdot z = \omega^{k+2} z.
\end{align*}
Therefore we have a chain of isomorphisms
\begin{align*}
\Bbbk_{-1}[u,v]^{G_{1,k}} &= \left( \Bbbk_{-1}[u,v]^{\langle h^k \rangle} \right)^{\langle h^{k/2+1} \rangle} \cong \left( \frac{\Bbbk[x,y,z]}{\langle x^2+y^2+z^2 \rangle} \right)^{\langle h^{k/2 + 1}\rangle} \\
&\cong \left( \frac{\Bbbk[x,y,z]}{\langle x^2+y^2+z^2 \rangle} \right)^{\langle \gamma \rangle} \cong \left( \Bbbk[u,v]^{\langle \gamma^k \rangle} \right)^{\frac{1}{2k}(1,k+1)} \\
&= \Bbbk[u,v]^{\frac{1}{2k}(1,k+1)},
\end{align*}
as claimed.
\end{proof}

\begin{prop} \label{nequals2}
Suppose that $n=2$ and $k$ is odd. Then
\begin{align*}
\Bbbk_{-1}[u,v] ^{G_{n,k}} \cong \Bbbk[u,v]^{\frac{1}{4k}(1,2k+1)},
\end{align*}
so that $\Bbbk_{-1}[u,v] ^{G_{n,k}}$ is a cyclic quotient singularity.
\end{prop}
\begin{proof}
Here we have that $\omega^{4k} = 1$ and 
\begin{align*}
g = \begin{pmatrix} -1 & 0 \\ 0 & -1 \end{pmatrix}, \quad h = \begin{pmatrix} 0 & w^2 \\ w^2 & 0 \end{pmatrix}.
\end{align*}
Observe that the invariant ring $\Bbbk_{-1}[u,v]^{\langle g \rangle}$ is generated by
\begin{align*}
x = \omega^k(u^2-v^2), \quad y = 2uv, \quad z = u^2+v^2,
\end{align*}
and these satisfy the relation $x^2 + y^2 + z^2 = 0$. Now define
\begin{align*}
\ell = \left\{ \begin{array}{cc}
k+1 & \text{if } k \equiv 1 \text{ mod } 4, \\
3k+1 & \text{if } k \equiv 3 \text{ mod } 4,
\end{array} \right.
\end{align*}
so that $\ell/2$ is an odd integer and $2\ell \equiv 2k+2 \text{ mod } 4k$. Moreover, $\ell/2$ is coprime to $2k$ and so $\langle h \rangle = \langle h^{\ell/2} \rangle$. Then
\begin{align*}
h^{\ell/2} \cdot x &= \omega^{2\ell} \omega^k (v^2-u^2) = \omega^{2k+2} \omega^{2k} \omega^k (u^2-v^2) = \omega^2 x, \\
h^{\ell/2} \cdot y &= \omega^{2\ell} (2vu) = \omega^{2k+2} \omega^{2k} (2uv) = \omega^2 y, \\
h^{\ell/2} \cdot z &= \omega^{2\ell} (u^2+v^2) = \omega^{2k+2} z.
\end{align*}
\indent Now instead consider the group $\frac{1}{4k}(1,2k+1)$ with generator $\gamma = \begin{psmallmatrix} \omega & 0 \\ 0 & \omega^{2k+1} \end{psmallmatrix}$, where $\omega^{4k} = 1$, acting on the commutative ring $\Bbbk[u,v]$. Since $\gamma^{2k} = \begin{psmallmatrix} -1 & 0 \\ 0 & -1 \end{psmallmatrix}$, the invariant ring $\Bbbk[u,v]^{\gamma^{2k}}$ is generated by
\begin{align*}
x = w^{3k}(u^2+v^2), \quad y = u^2-v^2, \quad z = 2uv,
\end{align*}
and these satisfy $x^2+y^2+z^2 = 0$. Also, we have
\begin{align*}
\gamma \cdot x = \omega^2 x , \quad
\gamma \cdot y = \omega^2 y , \quad 
\gamma \cdot z = \omega^{2k+2} z,
\end{align*}
and so there is a chain of isomorphisms
\begin{align*}
\Bbbk_{-1}[u,v]^{G_{2,k}} &= \left( \Bbbk_{-1}[u,v]^{\langle g \rangle} \right)^{\langle h^{\ell/2} \rangle} \cong \left( \frac{\Bbbk[x,y,z]}{\langle x^2+y^2+z^2 \rangle} \right)^{\langle h^{\ell/2}\rangle} \\
&\cong \left( \frac{\Bbbk[x,y,z]}{\langle x^2+y^2+z^2 \rangle} \right)^{\langle \gamma \rangle} \cong \left( \Bbbk[u,v]^{\langle \gamma^{2k} \rangle} \right)^{\frac{1}{4k}(1,2k+1)} \\
&= \Bbbk[u,v]^{\frac{1}{4k}(1,2k+1)},
\end{align*}
as claimed.
\end{proof}

In both cases, one can obtain an explicit presentation of $\Bbbk_{-1}[u,v]^{G_{n,k}}$ using the results from Section \ref{typeApres}. In particular, every cyclic quotient singularity of the form $\Bbbk[u,v]^{\frac{1}{4k}(1,2k+1)}$ with $k \geqslant 1$ occurs as an invariant ring of the form $\Bbbk_{-1}[u,v]^{G_{n,k}}$.

\begin{rem} \label{smallinvringiso}
Suppose that $R = \Bbbk[x_1, \cdots, x_n]$ and $G$ and $H$ are small subgroups of $\Autgr(A) \cong \GL(n,\Bbbk)$. By \cite[Theorem 1.9]{behnke}, if $R^G \cong R^H$, then $G$ and $H$ are conjugate. However, we can use the results from above to show a version of this result need not hold if we replace $R$ by an AS regular algebra.  \\
\indent Indeed, let $A = \Bbbk_{-1}[u,v]$ and consider
\begin{align*}
G_{2,1} = \bigg\langle \hspace{-3pt}
\begin{pmatrix} -1 & 0 \\ 0 & -1 \end{pmatrix},
\begin{pmatrix} 0 & -1 \\ -1 & 0 \end{pmatrix}
\hspace{-3pt} \bigg\rangle  .
\end{align*}
Proposition \ref{nequals2} tells us that $\Bbbk_{-1}[u,v]^{G_{2,1}} \cong \Bbbk[u,v]^{\frac{1}{4}(1,3)}$, which is the coordinate ring of a type $\AA_3$ Kleinian singularity. Explicitly, the elements
\begin{align*}
x = (u-v)(u+v)(u-v)(u+v), \quad x = (u+v)(u-v)(u+v)(u-v), \quad z = (u+v)^2
\end{align*}
generate the invariant ring and give rise to an isomorphism
\begin{align*}
\Bbbk_{-1}[u,v]^{G_{2,1}} \cong \frac{\Bbbk[x,y,z]}{\langle xy - z^4 \rangle}.
\end{align*}
However, using Theorem \ref{kqinvariantspres} one can show that $\Bbbk_{-1}[u,v]^{\frac{1}{4}(1,3)}$ is also the coordinate ring of a type $\AA_3$ Kleinian singularity, and so we have an isomorphism 
\begin{align*}
\Bbbk_{-1}[u,v]^{G_{2,1}} \cong \Bbbk_{-1}[u,v]^{\frac{1}{4}(1,3)}.
\end{align*}
The groups $G_{2,1}$ and $\frac{1}{4}(1,3)$ are both small but are certainly not isomorphic (the former is isomorphic to the Klein four group, while the latter is cyclic), let alone conjugate.
\end{rem}

\subsection{Type $\DD$ singularities}
It remains to consider the cases where $n \geqslant 3$. In Section \ref{commutativecase}, we recalled the definition of the group $\mathbb{D}_{m,q}$, where $m$ and $q$ are positive coprime integers with $1 < q < m$. In this subection, it will be convenient to let $\omega$ be a primitive $4q(m-q)$th root of unity, and to write
\begin{align}
\DD_{m,q} = \Bigg\langle \hspace{-3pt} \begin{pmatrix} \omega^{2(m-q)} & 0 \\ 0 & \omega^{-2(m-q)} \end{pmatrix}, \begin{pmatrix} 0 & \omega^{q} \\ \omega^{q} & 0 \end{pmatrix} \hspace{-3pt} \Bigg\rangle, \label{Dmqaltdef}
\end{align}
We now seek to show that, if $n \geqslant 3$ and $\Bbbk_{-1}[u,v]^{G_{n,k}}$ is commutative, then there is a one-to-one correspondence between the invariant rings $\Bbbk_{-1}[u,v]^{G_{n,k}}$ and $\Bbbk[u,v]^{\DD_{m,q}}$. \\
\indent To be more precise, first define
\begin{gather*}
S \coloneqq \{ (n,k) \mid \gcd(n,k) = 1, \gap \gap  n \geqslant 3, \gap \gap n-k \equiv 1 \text{ mod } 2, \gap \gap k \not\equiv 2 \text{ mod } 4 \}, \\
T \coloneqq \{ (m,q) \mid 1 < q < m, \gap \gap \gcd(m,q) = 1 \}.
\end{gather*}
By Theorem \ref{classifyG} and Proposition \ref{commvsnoncomm}, the set $S$ consists of all pairs of integers $(n,k)$ with $n \geqslant 3$ and such that $\Bbbk_{-1}[u,v]^{G_{n,k}}$ is commutative, while the set $T$ consists of all pairs of integers $(m,q)$ which can be used to define the group $\DD_{m,q}$. The aim of this subsection is to construct a bijection $\theta : S \to T$ such that 
\begin{align}
\Bbbk_{-1}[u,v]^{G_{n,k}} \cong \Bbbk[u,v]^{\DD_{\theta(n,k)}}. \label{nkmqiso}
\end{align}
We begin by constructing the map $\theta$.

\begin{lem}
The maps
\begin{align*}
\theta : S \to T, \qquad \gap \gap \theta(n,k) &= \left \{
\begin{array}{ll}
    (n+\frac{k}{2}, n) & \text{\normalfont{if }} n \text{\normalfont{ is odd }} (\Leftrightarrow k \text{\normalfont{ is even)},} \\
    (\frac{n}{2}+k, \frac{n}{2}) & \text{\normalfont{if }} n \text{\normalfont{ is even }} (\Leftrightarrow k \text{\normalfont{ is odd)},}
\end{array} \right. \\
\eta : T \to S, \qquad \eta(m,q) &= \left \{
\begin{array}{ll}
    (q, 2(m-q)) & \text{\normalfont{if }} m-q \equiv 0 \text{\normalfont{ mod }} 2,  \\
    (2q, m-q) & \text{\normalfont{if }} m-q \equiv 1 \text{\normalfont{ mod }} 2, 
\end{array} \right.
\end{align*}
are well-defined, mutually inverse bijections.
\end{lem}
\begin{proof}
We first show that $\theta$ is well-defined. Let $(n,k) \in S$, and suppose that $n$ is odd and $k$ is even, so that $\theta(n,k) = (n+\frac{k}{2}, n) \eqqcolon (m,q)$. Since $q = n \geqslant 3$ and $n + \frac{k}{2} > n$, we have $1 < q < m$. Moreover, since $\gcd(n,k) = 1$, we have $\gcd(m,q) = \gcd(n+\frac{k}{2},n) = \gcd(\frac{k}{2},n) = 1$. Therefore $\theta(n,k) \in T$. If instead $n$ is even, so that $k$ is odd, we have $\theta(n,k) = (\frac{n}{2}+k, \frac{n}{2}) \eqqcolon (m,q)$. Here $n \geqslant 4$ so $\frac{n}{2} \geqslant 2$, meaning that $1 < q < m$. As before, $\gcd(m,q) = 1$ follows from the fact that $\gcd(n,k) = 1$. Again we deduce that $\theta(n,k) \in T$. \\
\indent To show that $\eta$ is well-defined, let $(m,q) \in T$, and first suppose that $m-q \equiv 0 \text{ mod } 2$, which means that $m$ and $q$ are both odd. We then have $\eta(m,q) = (q,2(m-q)) \eqqcolon (n,k)$. Now, since $1 < q < m$ and $q$ is odd, we have that $n = q \geqslant 3$. Moreover, since $n$ and $k$ have different parities, $n-k \equiv 1 \text{ mod } 2$. Also, $m-q \equiv 0 \text{ mod } 2$, so $k = 2(m-q) \equiv 0 \text{ mod } 4$. Finally, $\gcd(n,k) = \gcd(q,2(m-q)) = \gcd(q,2m) = \gcd(1,m) = 1$, where the penultimate equality follows since $q$ and $m$ are both odd. It follows that $\eta(m,q) \in S$. \\
\indent If instead $m-q \equiv 1 \text{ mod } 2$, then $\eta(m,q) = (2q,m-q) \eqqcolon (n,k)$. Since $q \geqslant 2$, we have $n = 2q \geqslant 3$, and clearly $2q$ and $m-q$ have different parities, so $n-k \equiv 1 \text{ mod } 2$. Since $m-q \equiv 1 \text{ mod } 2$, it follows that $k = m-q \not \equiv 2 \text{ mod } 4$. Finally, using the fact that $n-q$ is odd, $\gcd(n,k) = \gcd(2q,n-q) = \gcd(q,n-q) = \gcd(n,q) = 1$. Therefore $\eta(m,q) \in S$. \\
\indent Finally, it is straightforward to check that $\theta$ and $\eta$ are mutual inverses after noting that, if the pair $(n,k)$ satisfies the conditions in the first (respectively, second) line in the definition of $\theta$, then the pair $(m,q) = \theta(n,k)$ satisfies the conditions in the first (respectively, second) line in the definition of $\eta$.
\end{proof}

We now seek to establish the isomorphism (\ref{nkmqiso}). We split the proof of this into two cases depending on the parity of $n$ (and hence of $k$). The proof when $n$ is odd is straightforward due to the fact that $G_{n,k} = \DD_{\theta(n,k)}$ in this case.

\begin{prop} \label{noddiso}
Suppose that $(n,k) \in S$ and that $n$ is odd and $k$ is even (hence $k \equiv 0 \text{\normalfont{ mod }} 4$). Then there is an isomorphism
\begin{align*}
\Bbbk_{-1}[u,v]^{G_{n,k}} \cong \Bbbk[u,v]^{\DD_{\theta(n,k)}}.
\end{align*}
\end{prop}
\begin{proof}
Let $(m,q) \coloneqq \theta(n,k)$, so $m = n + \frac{k}{2}$ and $q = n$. Observe that $4q(m-q) = 2nk$, and hence the roots of unity appearing in (\ref{Gnkaltpres}) and (\ref{Dmqaltdef}) are the same; we write $\omega$ for this common root of unity. Since $m-q = \frac{k}{2}$ and $q = n$, we have
\begin{align*}
G_{n,k} = \Bigg\langle \hspace{-3pt} \begin{pmatrix} \omega^{2k} & 0 \\ 0 & \omega^{-2k} \end{pmatrix}, \begin{pmatrix} 0 & \omega^{n} \\ \omega^{n} & 0 \end{pmatrix} \hspace{-3pt}\Bigg\rangle , \qquad 
\DD_{\theta(n,k)} = \Bigg\langle \hspace{-3pt} \begin{pmatrix} \omega^{k} & 0 \\ 0 & \omega^{-k} \end{pmatrix}, \begin{pmatrix} 0 & \omega^{n} \\ \omega^{n} & 0 \end{pmatrix}\hspace{-3pt} \Bigg\rangle . 
\end{align*}
Clearly $G_{n,k} \subseteq \DD_{\theta(n,k)}$. Moreover, we have $|\DD_{m,q}| = 4q(m-q) = 2nk = |G_{n,k}|$, and hence $G_{n,k} = \DD_{\theta(n,k)}$. Notice that both groups contain $\begin{psmallmatrix} 0 & w^n \\ w^n & 0 \end{psmallmatrix}^k = \begin{psmallmatrix} -1 & 0 \\ 0 & -1 \end{psmallmatrix} \eqqcolon \alpha$, and that
\begin{align*}
\Bbbk_{-1}[u,v]^{\langle \alpha \rangle} = \Bbbk[u^2,-v^2,uv] \cong \frac{\Bbbk[x,y,z]}{\langle xy - z^2 \rangle}, \qquad
\Bbbk[u,v]^{\langle \alpha \rangle} = \Bbbk[u^2,v^2,uv] \cong \frac{\Bbbk[x,y,z]}{\langle xy - z^2 \rangle},
\end{align*}
where the induced actions of $G_{n,k}$ and $\DD_{\theta(n,k)}$ on $\Bbbk[x,y,z]/\langle xy - z^2 \rangle$ are the same. Therefore,
\begin{align*}
\Bbbk_{-1}[u,v]^{G_{n,k}} &= \big(\Bbbk_{-1}[u,v]^{\langle \alpha \rangle}\big)^{G_{n,k}} \cong \big(\Bbbk[u,v]^{\langle \alpha \rangle}\big)^{\DD_{\theta(n,k)}} = \Bbbk[u,v]^{\DD_{\theta(n,k)}},
\end{align*}
as claimed.
\end{proof}

We now turn our attention to the case when $n$ is even, where the proof is more involved.

\begin{prop} \label{neveniso}
Suppose that $(n,k) \in S$ and that $n$ is even and $k$ is odd. Then there is an isomorphism
\begin{align*}
\Bbbk_{-1}[u,v]^{G_{n,k}} \cong \Bbbk[u,v]^{\DD_{\theta(n,k)}}.
\end{align*}
\end{prop}
\begin{proof}
Let $(m,q) \coloneqq \theta(n,k)$, so that $m = \frac{n}{2} + k$ and $q = \frac{n}{2}$. Then $4q(m-q) = 2nk$ and so, as in the proof of Proposition \ref{noddiso}, the roots of unity appearing in (\ref{Gnkaltpres}) and (\ref{Dmqaltdef}) are the same, and we write $\omega$ for this common root of unity. Since $m-q = k$ and $q = \frac{n}{2}$, we have
\begin{align*}
G_{n,k} = \Bigg\langle \hspace{-3pt} \begin{pmatrix} \omega^{2k} & 0 \\ 0 & \omega^{-2k} \end{pmatrix}, \begin{pmatrix} 0 & \omega^{n} \\ \omega^{n} & 0 \end{pmatrix} \hspace{-3pt}\Bigg\rangle, \qquad 
\DD_{\theta(n,k)} = \Bigg\langle \hspace{-3pt} \begin{pmatrix} \omega^{2k} & 0 \\ 0 & \omega^{-2k} \end{pmatrix}, \begin{pmatrix} 0 & \omega^{n/2} \\ \omega^{n/2} & 0 \end{pmatrix} \hspace{-3pt}\Bigg\rangle . 
\end{align*}
Writing $a$ and $b$ for the generators of $\DD_{\theta(n,k)}$ in the order given above, we have
$\Bbbk[u,v]^{\DD_{\theta(n,k)}} = \big( \Bbbk[u,v]^{\langle a \rangle} \big)^{\langle b \rangle}$. Notice that, since $b$ has order $4k$ and $\gcd(4k, k+2) = \gcd(-8,k+2) = 1$ (using that $k$ is odd), we have $\langle b \rangle = \langle b^{k+2} \rangle$ and hence $\Bbbk[u,v]^{\DD_{\theta(n,k)}} = \big( \Bbbk[u,v]^{\langle a \rangle} \big)^{\langle b^{k+2} \rangle}$. Now, if we set $x=u^n$, $y=v^n$, and $z = uv$, then we obtain an isomorphism
\begin{align*}
\Bbbk[u,v]^{\langle a \rangle} \cong \frac{\Bbbk[x,y,z]}{\langle xy-z^n \rangle}.
\end{align*}
The action of $b^{k+2}$ on this ring is as follows:
\begin{align}
b^{k+2} \cdot x = \omega^{\frac{1}{2} n^2(k+2)} y, \quad b^{k+2} \cdot y = \omega^{\frac{1}{2} n^2(k+2)} x, \quad b^{k+2} \cdot z = \omega^{n(k+2)} z. \label{bkplustwoaction}
\end{align}
\indent We now turn our attention to the invariant ring $\Bbbk_{-1}[u,v]^{G_{n,k}}$. Clearly if we write $g$ and $h$ for the generators of $G_{n,k}$ in the order given above, then $\Bbbk_{-1}[u,v]^{G_{n,k}} = \big( \Bbbk_{-1}[u,v]^{\langle g \rangle} \big)^{\langle h \rangle}$. We split our analysis into two cases. First suppose that $n \equiv 0 \text{ mod } 4$. Setting $x=u^n$, $y=v^n$, and $z = uv$, we obtain an isomorphism 
\begin{align*}
\Bbbk_{-1}[u,v]^{\langle g \rangle} \cong \frac{\Bbbk[x,y,z]}{\langle xy-z^n \rangle}.
\end{align*}
Here, the fact that $n \equiv 0 \text{ mod } 4$ ensures that the relation is $xy-z^n$ rather than $xy+z^n$. The action of $h$ on this ring is given by
\begin{align}
h \cdot x = \omega^{n^2} y, \quad h \cdot y = \omega^{n^2} x, \quad h \cdot z = \omega^{2n} vu = \omega^{2n + nk} uv = \omega^{n(k+2)} z. \label{hactionnequiv0}
\end{align}
Since $n \equiv 0 \text{ mod } 4$, $\frac{n}{2}$ is an even integer and so
\begin{align*}
\frac{1}{2} n^2 (k+2) \equiv n^2 + \frac{n}{2} nk \equiv n^2 \text{ mod } 2nk.
\end{align*}
Therefore, comparing the actions given in (\ref{bkplustwoaction}) and (\ref{hactionnequiv0}), we obtain isomorphisms
\begin{align*}
\Bbbk_{-1}[u,v]^{G_{n,k}} &= \big( \Bbbk_{-1}[u,v]^{\langle g \rangle} \big)^{\langle h \rangle} \cong \left( \frac{\Bbbk[x,y,z]}{\langle xy-z^n \rangle} \right)^{\langle h \rangle} \\
&= \left( \frac{\Bbbk[x,y,z]}{\langle xy-z^n \rangle} \right)^{\langle b^{k+2} \rangle} \cong \big( \Bbbk[u,v]^{\langle a \rangle} \big)^{\langle b^{k+2} \rangle} = \Bbbk[u,v]^{\DD_{\theta(n,k)}}.
\end{align*}
\indent If instead $n \equiv 2 \text{ mod } 4$, then we can use a similar approach. In this case, if we set $x=u^n$, $y=-v^n$, and $z = uv$, then we obtain an isomorphism 
\begin{align*}
\Bbbk_{-1}[u,v]^{\langle g \rangle} \cong \frac{\Bbbk[x,y,z]}{\langle xy-z^n \rangle}.
\end{align*}
(This time the fact that $n \equiv 2 \text{ mod } 4$ and $y = -v^n$ ensures that the correct relation is $xy-z^n$.) The action of $h$ in this case is given by
\begin{align}
h \cdot x = \omega^{n^2} v^n = \omega^{n^2 + nk} y, \quad h \cdot y = -\omega^{n^2} u^n = \omega^{n^2 + nk} x, \quad h \cdot z = \omega^{n(k+2)} z. \label{hactionnequiv2}
\end{align}
Since $n \equiv 2 \text{ mod } 4$, $\frac{n}{2}$ is an odd integer and so
\begin{align*}
\frac{1}{2} n^2 (k+2) \equiv n^2 + \frac{n}{2} nk \equiv n^2 + nk \text{ mod } 2nk.
\end{align*}
Hence, comparing the actions given in (\ref{bkplustwoaction}) and (\ref{hactionnequiv2}), and arguing as in the previous case, we obtain an isomorphism 
\begin{gather*}
\pushQED{\qed} 
\Bbbk_{-1}[u,v]^{G_{n,k}} \cong \Bbbk[u,v]^{\DD_{\theta(n,k)}}. \qedhere
\end{gather*}
\end{proof}

\begin{rem}
The proof of Proposition \ref{noddiso} shows that $G_{n,k} \cong \DD_{\theta(n,k)}$ when $n$ is odd. On the other hand, when $n$ is even, $G_{n,k}$ and $\DD_{\theta(n,k)}$ have the following presentations:
\begin{gather*}
G_{n,k} = \langle g,h \mid g^n = 1 = h^{2k}, hg = g^{n-1} h \rangle, \\
\DD_{\theta(n,k)} = \langle a,b \mid a^n = 1, a^{n/2} = b^{2k}, ba = a^{n-1} b \rangle, 
\end{gather*}
and one can show that these groups are not isomorphic.
\end{rem}

\section{Noncommutative Invariant Rings of the Form $\Bbbk_{-1}[u,v]^{G_{n,k}}$} \label{noncomminvsec}
\noindent Throughout, $A = \Bbbk_{-1}[u,v]$ and $G = G_{n,k}$ as in Theorem \ref{classifyG}, and we write $g$ amd $h$ for the generators of $G$ as in (\ref{Gnkaltpres}). We now seek to obtain generators for the invariant rings $A^G$ when this ring is not commutative; by Proposition \ref{commvsnoncomm}, this happens precisely when $n$ and $k$ are both odd, and we will assume that this is the case throughout this section. \\
\indent We begin by identifying a $\Bbbk$-basis for the invariants.

\begin{prop}
A $\Bbbk$-basis for $A^G$ is given by
\begin{align*}
\Big\{ u^i v^j + (-1)^{(i+1)(j+1)+1} u^j v^i \hspace{2pt} \Big | \hspace{3pt} i,j, \geqslant 0, \hspace{3pt} i-j \equiv 0 \text{\emph{ mod }} n \text{ and } i+j \equiv 0 \text{\emph{ mod }} k \Big\}. 
\end{align*}
\end{prop}
\begin{proof}
By \cite{kkzGor}, the Reynolds operator
\begin{align*}
\rho : A \to A^G, \quad \rho(a) = \sum_{x \in G} x \cdot a,
\end{align*}
is surjective. Therefore, to find a $\Bbbk$-basis for $A^G$ it suffices to evaluate $\rho$ on the $\Bbbk$-basis for $A$ given by $\{ u^i v^j \mid i,j \geqslant 0 \}$. Fixing $i,j \geqslant 0$, we have
\begin{align*}
\sum_{\ell=0}^{n-1} g^\ell \cdot (u^i v^j) = \sum_{\ell=0}^{n-1} \omega^{2k \ell (i-j)} u^i v^j =
\left \{
\begin{array}{cl}
n u^i v^j & \text{if } i-j \equiv 0 \text{ mod } n, \\
0 & \text{otherwise}, 
\end{array} \right.
\end{align*}
using (\ref{omegaidentity}) and noting that $\omega^{2k}$ is a primitive $n$th root of unity. Similarly, using the fact that $\omega^{2n}$ is a primitive $k$th root of unity 
\begin{align*}
\sum_{m=0}^{2k-1} h^m \cdot (u^i v^j) &= \sum_{m=0}^{k-1} h^{2m} \cdot (u^i v^j) + \sum_{m=0}^{k-1} h^{2m+1} \cdot (u^i v^j) \\
&= \sum_{m=0}^{k-1} \omega^{2mn(i+j)} u^i v^j + \sum_{m=0}^{k-1} \omega^{(2m+1)n(i+j)} v^i u^j \\
&= \sum_{m=0}^{k-1} \omega^{2mn(i+j)} u^i v^j + (-1)^{ij} \omega^{n(i+j)} \sum_{m=0}^{k-1} \omega^{2mn(i+j)} u^j v^i \\
&= \left \{
\begin{array}{cl}
k \Big( u^i v^j + (-1)^{ij} \omega^{n(i+j)} u^j v^i \Big) & \text{if } i+j \equiv 0 \text{ mod } k, \\
0 & \text{otherwise}, 
\end{array} \right. \\
&= \left \{
\begin{array}{cl}
k \Big( u^i v^j + (-1)^{(i+1)(j+1)+1} u^j v^i \Big) & \text{if } i+j \equiv 0 \text{ mod } k, \\
0 & \text{otherwise}, 
\end{array} \right. 
\end{align*}
since if $i+j \equiv 0 \text{ mod } k$, then $(-1)^{ij} \omega^{n(i+j)} = (-1)^{ij} (-1)^{i+j} = (-1)^{(i+1)(j+1)+1}$. Therefore, since $G = \{ h^m g^\ell \mid 0 \leqslant m \leqslant 2k-1, \gap 0 \leqslant \ell \leqslant n-1 \}$, we have
\begin{align*}
\rho(u^i v^j) &= \sum_{m=0}^{2k-1} \sum_{\ell=0}^{n-1} h^m g^\ell \cdot (u^i v^j) \\
&= \left \{
\begin{array}{cl}
nk \Big( u^i v^j + (-1)^{(i+1)(j+1)+1} u^j v^i \Big) & \text{if } i-j \equiv 0 \text{ mod } n \text{ and } i+j \equiv 0 \text{ mod } k, \\
0 & \text{otherwise}. 
\end{array} \right. 
\end{align*}
It follows that $A^G$ has the claimed $\Bbbk$-basis.
\end{proof}

We now write the $\Bbbk$-basis from the above proposition in a modified form, which will make it easier to write down a set of generators. Without loss of generality, we may assume that $i \geqslant j$. By our assumptions on $i$ and $j$, we have $i-j = ns$ and $i+j = kt$ for some non-negative integers $s,t$. Clearly $i = \frac{1}{2}(kt+ns)$ and $j = \frac{1}{2}(kt-ns)$, and we also have
\begin{align}
u^i v^j + (-1)^{(i+1)(j+1)+1} u^j v^i &= (u^{i-j} + (-1)^{(i-j)j + (i+1)(j+1) + 1} v^{i-j} ) (uv)^j \nonumber \\
&= (u^{i-j} + (-1)^{i} v^{i-j} ) (uv)^j. \label{altpres}
\end{align}
If $i = j$, then (\ref{altpres}) is nonzero if and only if $i \equiv 0 \equiv j \text{ mod } 2k$, and so the elements we obtain are powers of $(uv)^{2k}$. Otherwise $i > j$, in which case (\ref{altpres}) is equal to
\begin{align*}
(u^{ns} + (-1)^{\frac{1}{2}(kt+ns)} v^{ns}) (uv)^{\frac{1}{2}(kt-ns) } = (u^{ns} + (-1)^{r+ns} v^{ns}) (uv)^{r},
\end{align*}
where $r = \frac{1}{2}(kt-ns)$. We summarise our findings below:

\begin{prop} \label{altbasis}
A $\Bbbk$-basis for $A^G$ is given by
\begin{align*}
\Big\{ (uv)^{2ki} \hspace{2pt} \Big | \hspace{3pt} i \geqslant 1 \Big\}
\cup
\Big\{ (u^{ns} + (-1)^{r+ns} v^{ns}) (uv)^{r} \hspace{2pt} \Big | \hspace{3pt} r \geqslant 0, \gap \gap s,t \geqslant 1, \gap \gap 2r+ns = kt \Big\}.
\end{align*}
\end{prop}

\indent We now seek to use Proposition \ref{altbasis} to find a set of generators for $A^G$. The form of the generators depends on whether $n > k$ or $n < k$, and so we need to divide the argument into these two cases. The only case which this does not consider is when $n= 1 = k$, but this is covered by \cite[Remark 2.6]{downup}.

\subsection{The case $n > k$} First suppose that $n > k$. Write
\begin{align*}
\frac{n}{\frac{1}{2}(n+k)} = [\gamma_1, \dots, \gamma_{d-1}]
\end{align*}
for the Hirzebruch-Jung continued fraction expansion of $\frac{n}{\frac{1}{2}(n+k)}$ (obviously one can write this fraction as $\frac{2n}{n+k}$, but the given presentation is in lowest terms). Since $n > k$, we note that $\gamma_1 = 2$. Also define
\begin{align*}
\beta_i = \left\{
\begin{array}{cc}
\gamma_3 + 1 & \text{if } i = 3, \\
\gamma_i & \text{if } i \neq 3.
\end{array}  \right.
\end{align*}
Define three series of integers $r_1, \dots, r_{d-1}$, $s_1, \dots, s_{d-1}$, and $t_1, \dots, t_{d-1}$  as follows:
\begin{gather*}
\begin{array}{lll | ll}
s_1 = 1 & s_2 = 1 & & & s_i = \beta_{i} s_{i-1} - s_{i-2} \text{ for } 3 \leqslant i \leqslant d-1, \\[2pt]
t_1 = 2\beta_2+1 & t_2 = 2\beta_2-1 & & & t_i = \beta_{i} t_{i-1} - t_{i-2} \text{ for } 3 \leqslant i \leqslant d-1, \\
\end{array} \\
r_i = \tfrac{1}{2} (kt_i - ns_i) \text{ for } 1 \leqslant i \leqslant d-1,
\end{gather*}
\noindent where the entries to the right of the vertical line only exist when $d > 2$, which happens if and only if $k \geqslant 3$, if and only if the action of $G$ on $A$ has nontrivial homological determinant. Also observe that the $r_i$ obey the same recurrence relation as the $s_i$ and $t_i$. We now collect some properties of these series:

\begin{lem}
If $k \geqslant 3$ then
\begin{align*}
[\beta_3, \beta_4, \dots, \beta_{d-1}] = \frac{k\beta_2 - \frac{1}{2}(n-k)}{k\beta_2 - \frac{1}{2}(n+k)}.
\end{align*}
\end{lem}
\begin{proof}
Let $\alpha = [\gamma_3, \gamma_4, \dots, \gamma_{d-1}] = [\beta_3-1, \beta_4, \dots, \beta_{d-1}]$, so that the quantity we wish to determine is $\alpha+1$. Then, since $\beta_1 = 2$, 
\begin{align*}
\frac{2n}{n+k} = 2- \frac{1}{\beta_2 - \frac{1}{\alpha}} = \frac{\alpha (2\beta_2-1) - 2}{\alpha \beta_2 - 1},
\end{align*}
which rearranges to give
\begin{align*}
\alpha = \frac{k}{k \beta_2 - \frac{1}{2}(n+k)}.
\end{align*}
Therefore,
\begin{align*}
[\beta_3, \beta_4, \dots, \beta_{d-1}] = \alpha+1 = \frac{k + k \beta_2 - \frac{1}{2}(n+k)}{k \beta_2 - \frac{1}{2}(n+k)} = \frac{k\beta_2 - \frac{1}{2}(n-k)}{k\beta_2 - \frac{1}{2}(n+k)},
\end{align*}
as claimed.
\end{proof}

\begin{lem} \label{propsofsiti} The $r_i$, $s_i$, and $t_i$ have the following properties:
\begin{enumerate}[{\normalfont (1)},leftmargin=*,topsep=0pt,itemsep=2pt]
\item $r_1 = k\beta_2 - \frac{1}{2}(n-k)$, $r_2 = k\beta_2 - \frac{1}{2}(n+k)$, $r_{d-2} = 1$ and $r_{d-1} = 0$. Moreover, $r_1 > r_2 > \dots > r_{d-2} > r_{d-1}$.
\item $r_{i} s_{i+1} - r_{i+1} s_{i} = k$ and $r_{i} t_{i+1} - r_{i+1} t_{i} = n$ for $1 \leqslant i \leqslant d-2$.
\item $s_{d-1} = k$ and $t_{d-1} = n$.
\end{enumerate}
\end{lem}
\begin{proof}
\leavevmode
\begin{enumerate}[{\normalfont (1)},wide=0pt,topsep=0pt,itemsep=2pt]
\item Using the definition of the $r_i$,
\begin{gather*}
r_1 = \frac{1}{2}\big( k(2\beta_2 + 1) -n \big) = k\beta_2 - \frac{1}{2}(n-k) \\
r_2 = \frac{1}{2}\big( k(2\beta_2 - 1) -n \big) = k\beta_2 - \frac{1}{2}(n+k).
\end{gather*}
Then, since $\frac{r_1}{r_2} = [\beta_3, \beta_4, \dots, \beta_{d-1}]$ and the $r_i$ satisfy $r_i = \beta_{i-1} r_{i-1} - r_{i-2}$ for $3 \leqslant i \leqslant d-1$, it follows from general theory that $r_1 > r_2 > \dots > r_{d-2} > r_{d-1}$, with $r_{d-2} = 1$ and $r_{d-1} = 0$.
\item These are both easy to prove by induction.
\item To determine the value of $s_{d-1}$, we use parts (1) and (2):
\begin{align*}
s_{d-1} = 1 \cdot s_{d-1} - 0 \cdot s_{d-2} = r_{d-2} s_{d-1} - r_{d-1} s_{d-2} = k.
\end{align*}
Similarly, we can use the identity $r_{d-2} t_{d-1} - r_{d-1} t_{d-2} = n$ to show that $t_{d-1} = n$. \qedhere
\end{enumerate}
\end{proof}

We also require the following two technical lemmas:

\begin{lem} \label{technicallemma2}
Suppose that the triple $(r,s,t)$ satisfies $2r+ns = kt$ and that $0 \leqslant r < 2k$ and $s,t \geqslant 1$. If $\frac{r}{s} > \frac{r_i}{s_i}$ for some $i$ with $1 \leqslant i \leqslant d-1$, then $i > 1$ and $r \geqslant r_{i-1}$.
\end{lem}
\begin{proof}
Let $(r,s,t)$ be as in the statement. We first show that $i > 1$. Observing that $\beta_2 = \lceil \frac{n+k}{2k} \rceil$, we need to show that 
\begin{align*}
\frac{r}{s} \leqslant \frac{r_1}{s_1} = k \left \lceil \frac{n+k}{2k} \right \rceil - \frac{1}{2}(n-k).
\end{align*}
If $s \geqslant 2$, then
\begin{align*}
\frac{r}{s} \leqslant \frac{r}{2} < \frac{2k}{2} = k = k \left( \frac{n+k}{2k} \right) - \frac{1}{2}(n-k) \leqslant k \left \lceil \frac{n+k}{2k} \right \rceil - \frac{1}{2}(n-k),
\end{align*}
as required. So now suppose that $s=1$. Write $n$ uniquely in the form $n = ka - b$ where $a$ is odd and $b \in \{0, 2, \dots, 2(k-1)\}$. Then
\begin{align*}
k \left \lceil \frac{n+k}{2k} \right \rceil - \frac{1}{2}(n-k) = k \left \lceil \frac{a+1}{2} - \frac{b}{2k} \right \rceil - \frac{k(a-1) - b}{2} = \frac{k(a+1)}{2} - \frac{k(a-1) - b}{2} = k + \frac{b}{2}.
\end{align*}
Rearranging the equality $2r + n = kt$ gives $r = k \left(\frac{t-a}{2}\right) + \frac{b}{2}$. Since $n$ and $k$ are both odd, $t$ is also odd, so $\frac{t-a}{2}$ is an integer. But since $r < 2k$, the previous equality forces $\frac{t-a}{2}$ to be at most $1$. Therefore
\begin{align}
r = k \left(\frac{t-a}{2}\right) + \frac{b}{2} \leqslant k + \frac{b}{2} = k \left \lceil \frac{n+k}{2k} \right \rceil + \frac{1}{2}(k-n). \label{rinequality}
\end{align}
This shows that $\frac{r}{s} \leqslant \frac{r_1}{s_1}$. \\
\indent Now assume that $\frac{r}{s} > \frac{r_i}{s_i}$ for some $i$ (where necessarily $i > 1$ by the above). Since, by Lemma \ref{propsofsiti}, the $r_i$ and $s_i$ satisfy $r_{i-1} s_{i} - r_{i} s_{i-1} = k$ for $2 \leqslant i \leqslant d-1$, it follows that $\frac{r_{i}}{s_{i}} < \frac{r_{i-1}}{s_{i-1}}$ for $2 \leqslant i \leqslant d-1$, so we may assume that $i$ has been chosen minimal subject to $\frac{r}{s} > \frac{r_i}{s_i}$. In particular, we have
\begin{align}
\frac{r_i}{s_i} < \frac{r}{s} \leqslant \frac{r_{i-1}}{s_{i-1}}. \label{rsinequality}
\end{align}
Seeking a contradiction, assume that $r < r_{i-1}$, so that $r_{i-1} - r \geqslant 1$. 
Consider the equations
\begin{align*}
2r + ns = kt \qquad \text{and} \qquad 2r_i + n s_i = kt_i;
\end{align*}
multiplying the former by $r_i$ and the latter by $r$ and then subtracting gives
\begin{align}
n(r_i s - r s_i) = k(r_i t - r t_i). \label{rtidentity}
\end{align}
Since $n$ and $k$ are coprime, this forces $r s_i - r_i s$ to be a multiple of $k$. Since (\ref{rsinequality}) implies that $rs_i - r_i s > 0$, we must therefore have $rs_i - r_i s \geqslant k$. Combining this with Lemma \ref{propsofsiti} (2), we find that
\begin{align*}
rs_i - r_i s \geqslant r_{i-1} s_{i} - r_{i} s_{i-1}.
\end{align*}
This rearranges to gives
\begin{align*}
r_i (s_{i-1} - s) \geqslant s_i (r_{i-1} - r) > 0,
\end{align*}
where the second inequality follows since $r_{i-1} - r > 0$ by assumption. In particular, $s_{i-1} > s$. Then, using Lemma \ref{propsofsiti} (2),
\begin{align*}
\frac{r}{s} &= \frac{r s_i - r_i s}{ss_i} + \frac{r_i}{s_i} \geqslant \frac{k}{ss_i} + \frac{r_i}{s_i} > \frac{k}{s_i s_{i-1}} + \frac{r_i}{s_i} = \frac{k + r_i s_{i-1}}{s_i s_{i-1}} = \frac{r_{i-1} s_i}{s_i s_{i-1}} = \frac{r_{i-1}}{s_{i-1}}.
\end{align*}
Hence $\frac{r}{s} > \frac{r_{i-1}}{s_{i-1}}$, contradicting the minimality of $i$, so we must have $r \geqslant r_{i-1}$.
\end{proof}

\begin{rem} \label{technicalrem}
Suppose that $s=1$ and $(r,s,t)$ satisfy the hypotheses of Lemma \ref{technicallemma2}, and let $a$ and $b$ be as in the above proof. It is straightforward to check that $a = 2 \beta_2 -1$. Before equation (\ref{rinequality}), we saw that $\frac{t-a}{2}$ was an integer which was at most 1. If $\frac{t-a}{2} \leqslant -1$, then
\begin{align*}
r = k \left( \frac{t-a}{2} \right) + \frac{b}{2} \leqslant -k + (k-1) = -1,
\end{align*}
contradicting the fact that $r > 0$. It follows that $\frac{t-a}{2}$ is either $0$ or $1$, i.e. $t = a = 2 \beta_2 -1 = t_2$ or $t = a+2 = 2 \beta_2 + 1 = t_1$. This in turn forces $r = r_2$ or $r = r_1$, respectively. We will require this fact in the proof of Theorem \ref{ngtkgens}.
\end{rem}

\begin{lem} \label{tripleslemma}
Suppose that the triple $(r,s,t)$ satisfies $2r + ns = kt$ with $0 \leqslant r < 2k$, $s,t \geqslant 1$. Then there exist non-negative integers $c_1, \dots, c_{d-1}$ such that
\begin{align}
r = \sum_{i=1}^{d-1} c_i r_i, \qquad s = \sum_{i=1}^{d-1} c_i s_i, \qquad t = \sum_{i=1}^{d-1} c_i t_i. \label{cisums}
\end{align}
\end{lem}
\begin{proof}
We prove this by induction on $r$. If $r=0$, then the triple $(r,s,t)$ satisfies $ns=kt$ and so, since $n$ and $k$ are coprime, $s = mk$ and $t = mn$ for some non-negative integer $m$. By Lemma \ref{propsofsiti}, we know that $r_{d-1} = 0$, $s_{d-1} = k$, and $t_{d-1} = n$, and so the claim follows by setting $c_{d-1} = m$ and $c_i = 0$ for $i \neq d-1$. \\
\indent Now suppose that $r \geqslant 1$. Since $r_1 > r_2 > \dots > r_{d-1}$, choose $i$ minimal subject to $r - r_i \geqslant 0$. By construction, $r_i$ satisfies $r_i \leqslant r < r_{i-1}$. By Lemma \ref{technicallemma2}, it follows that $\frac{r}{s} \leqslant \frac{r_i}{s_i}$, and hence $r s_i \leqslant r_i s$. Since $r_i \leqslant r$, this also forces $s_i \leqslant s$. Then, using (\ref{rtidentity}),
\begin{align*}
0 \leqslant n(r_i s - r s_i) = k(r_i t - r t_i) \leqslant kr(t - t_i),
\end{align*}
and so $t \geqslant t_i$. \\
\indent Now set $(r',s',t') = (r-r_i, s-s_i, t-t_i)$, which is a triple satisfying $r', s', t' \geqslant 0$, and $2r' + ns' = kt'$. If $s' = t' = 0$ then necessarily $r' = 0$, and so setting $c_i = 1$ and $c_j = 0$ for $j \neq i$ gives the result. If $s', t' > 0$, then applying the induction hypothesis to the triple $(r',s',t')$ gives a decomposition of the form (\ref{cisums}), and hence also one for the triple $(r,s,t)$. It remains to exclude the possibility that exactly one of $s'$ and $t'$ is 0. Indeed, if $t' = 0$ then $2r' + ns' = 0$, and so $r' = 0 = s'$. If instead $s' = 0$, then $2r' = kt'$, and since $k$ is odd we have $r' = mk$ and $t' = 2m$ for some non-negative integer $m$. We claim that necessarily $r' < k$, which will force $m=0$, and hence $r'=0=t'$. To see this, first notice that if $r < k$ then this is immediate, so suppose that $r \geqslant k$. As noted in the proof of Lemma \ref{technicallemma2}, $\beta_2 = \lceil \frac{n+k}{2k} \rceil$, and so $\beta_2 < \frac{n+k}{2k} + 1$. It follows that 
\begin{align*}
r_1 = k \beta_2 - \frac{1}{2}(n-k) < k \bigg( \frac{n+k}{2k} + 1 \bigg) - \frac{1}{2}(n-k) = 2k,
\end{align*}
and since $r_2 = r_1 - k$, we also have $r_2 < k$. Therefore, if $r \geqslant k$, then exactly one of $r_1$ and $r_2$ satisfies $0 \leqslant r - r_i < k$, and this is our choice for $r_i$. Therefore $r' = r-r_i < k$, which finishes the proof.
\end{proof}

We are finally able to write down a list of generators for $A^G$.

\begin{thm} \label{ngtkgens}
Suppose $ n > k$. Then the elements
\begin{align*}
x_i = \big(u^{ns_i} + (-1)^{r_i + ns_i} v^{ns_i}\big) (uv)^{r_i}, \quad 1 \leqslant i \leqslant d-1, \qquad \text{and} \qquad x_d = (uv)^{2k}
\end{align*}
generate $A^G$.
\end{thm}
\begin{proof}
Since $2r_i + n s_i = k t_i$, the $x_i$ have the form given in Proposition \ref{altbasis} and hence are elements of $A^G$. So now suppose that $x = \big(u^{ns} + (-1)^{r + ns} v^{ns}\big) (uv)^{r}$ is an element of $A^G$, where $2r + ns = kt$; we wish to show that $x$ can be generated by $x_1, \dots, x_d$. Since $x_d = (uv)^{2k}$, we may assume that $r < 2k$. We prove by induction on $s$ that this is possible. \\
\indent When $s=1$, Remark \ref{technicalrem} shows that the only possibilities are $(r,s,t) = (r_1,s_1,t_1)$ or $(r,s,t) = (r_2,s_2,t_2)$, and so $x = x_1$ or $x=x_2$. So now suppose that $s \geqslant 2$ and that the result holds for smaller $s$. By Lemma \ref{tripleslemma}, there exist non-negative integers $c_i$ such that 
\begin{align*}
r = \sum_{i=1}^{d-1} c_i r_i, \qquad s = \sum_{i=1}^{d-1} c_i s_i, \qquad t = \sum_{i=1}^{d-1} c_i t_i.
\end{align*}
If $\sum_{i=1}^{d-1} c_i = 1$, then $r=r_i$, $s=s_i$, and $t=t_i$ for some $i$, and hence $x = x_i$. Else $\sum_{i=1}^{d-1} c_i \geqslant 2$, and so we can choose integers $a_i, b_i \geqslant 0$ with $a_i+b_i = c_i$ for all $i$, and $\sum_{i=1}^{d-1} a_i, \hspace{2pt} \sum_{i=1}^{d-1} b_i \geqslant 1$. Now define
\begin{align*}
r_a = \sum_{i=1}^{d-1} a_i r_i, \qquad s_a = \sum_{i=1}^{d-1} a_i s_i, \qquad t_a = \sum_{i=1}^{d-1} a_i t_i,
\end{align*}
and define $r_b, s_b, t_b$ analogously. Without loss of generality, we may assume that $s_a \geqslant s_b$. Observe that $r = r_a + r_b$, $s = s_a + s_b$, and $t = t_a + t_b$, and that $2r_a + ns_a = kt_a$ and $2r_b + ns_b = kt_b$. Therefore the elements
\begin{align*}
x_a \coloneqq \big(u^{ns_a} + (-1)^{r_a + ns_a} v^{ns_a}\big) (uv)^{r_a} \quad \text{and} \quad \big(u^{ns_b} + (-1)^{r_b + ns_b} v^{ns_b}\big) (uv)^{r_b}
\end{align*}
lie in $A^G$, and since $s_a, s_b < s$, they can be generated by the $x_i$ by the induction hypothesis. We then have
\begin{align}
x_a x_b &= \big(u^{ns_a} + (-1)^{r_a + ns_a} v^{ns_a}\big) (uv)^{r_a} \big(u^{ns_b} + (-1)^{r_b + ns_b} v^{ns_b}\big) (uv)^{r_b} \nonumber \\
&= \pm \big(u^{ns_a} + (-1)^{r_a + ns_a} v^{ns_a}\big) \big(u^{ns_b} + (-1)^{r_b + ns_b}  v^{ns_b}\big) (uv)^{r} \nonumber \\
&= \pm (u^{ns} + (-1)^{r+ns} v^{ns})(uv)^r  \pm ( (-1)^{r_b+ns_b} u^{ns_a}v^{ns_b} + (-1)^{r_a+ns_a} v^{ns_a}u^{ns_b} ) (uv)^{r} \nonumber \\
&= \pm \gap x \hspace{3pt} \pm \hspace{2pt} \big( u^{n(s_a-s_b)} u^{ns_b} v^{ns_b} + (-1)^{r+ns+n^2s_b^2} v^{n(s_a-s_b)} u^{ns_b} v^{ns_b}  \big) (uv)^{r} \nonumber \\
&= \pm \gap x \hspace{3pt} \pm \hspace{2pt} \big( u^{n(s_a-s_b)} + (-1)^{r+ns+n s_b} v^{n(s_a-s_b)}  \big) (uv)^{r+ns_b} \nonumber \\
&= \pm \gap x \hspace{3pt} \pm \hspace{2pt} \big( u^{n(s_a-s_b)} + (-1)^{r+n s_b + n(s_a-s_b)} v^{n(s_a-s_b)}  \big) (uv)^{r+ns_b}, \label{finalelement}
\end{align}
where we use $\pm$ to indicate that one can keep track of the powers of $-1$ occurring at each step, but they do not affect the argument and so we neglect to do so. Now, the right hand element at (\ref{finalelement}) corresponds to the triple $(r',s',t') = (r+ns_b, s_a-s_b, t)$, and it is easy to check that $2r' + ns' = kt'$. In particular $s' < s$, and so the inductive hypothesis tells us that this element is generated by the $x_i$. Since the same was true of $x_a$ and $x_b$, it follows that $x$ can be generated by the $x_i$, completing the proof.
\end{proof}


\begin{example}
First suppose that $n \geqslant 3$ is odd and $k = 1$, so that $G_{n,k}$ is the dihedral group of order $2n$. Then the Hirzebruch-Jung continued fraction expansion of $\frac{n}{\frac{1}{2}(n+k)}$ is
\begin{align*}
\frac{n}{\frac{1}{2}(n+1)} = 2 - \frac{1}{\frac{1}{2}(n+1)} = [2,\tfrac{1}{2}(n+1)].
\end{align*}
We therefore obtain the following $\beta$-, $r$-, $s$-, and $t$-series:
\begin{center}
\begin{tabular}{c|c c c c c c}
$i$ & 1 & 2 \\ \hline 
$\beta_i$ & 2 & $\tfrac{1}{2}(n+1)$  \\
$r_i$ & 1 & 0 \\
$s_i$ & 1 & 1 \\
$t_i$ & $n+2$ & $n$ 
\end{tabular}
\end{center}
By Theorem \ref{ngtkgens}, the invariant ring $A^G$ is generated by the three elements
\begin{align*}
x_1 = (u^n+v^n)uv, \quad x_2 = (u^n-v^n), \quad x_3 = (uv)^2.
\end{align*}
\indent This case has previously been considered in \cite{ckwz}. Indeed, the above generators agree with those given in \cite[Table 3 (d)]{ckwz}, up to a change of sign for one generator, which is due to the fact that the representation of $G$ we consider is conjugate to the representation used in \cite{ckwz}. \\
\indent For a more involved example, suppose instead that $n = 17$ and $k = 11$. The Hirzebruch-Jung continued fraction expansion of $\frac{n}{\frac{1}{2}(n+k)}$ is then
\begin{align*}
\frac{17}{5} = [2,2,2,2,3,2].
\end{align*}
We therefore obtain the following $\beta$-, $r$-, $s$-, and $t$-series:
\begin{center}
\begin{tabular}{c|c c c c c c}
$i$ & 1 & 2 & 3 & 4 & 5 & 6 \\ \hline 
$\beta_i$ & 2 & 2 & 3 & 2 & 3 & 2 \\
$r_i$ & 19 & 8 & 5 & 2 & 1 & 0 \\
$s_i$ & 1 & 1 & 2 & 3 & 7 & 11 \\
$t_i$ & 5 & 3 & 4 & 5 & 11 & 17 
\end{tabular}
\end{center}
Observe that these series have all the properties predicted by Lemma \ref{propsofsiti}. Then, following the recipe given in Theorem \ref{ngtkgens}, a set of generators for $A^G$ is given by the following seven elements:
\begin{gather*}
x_1 = (u^{17} + v^{17}) (uv)^{19}, \quad x_2 = (u^{17} - v^{17}) (uv)^8, \quad x_3 = (u^{34} - v^{34}) (uv)^5, \\
x_4 = (u^{51} - v^{51}) (uv)^2, \quad x_5 = (u^{119} + v^{119}) (uv), \quad x_6 = (u^{187} - v^{187}) uv, \quad x_7 = (uv)^{22}.
\end{gather*}
\end{example}

It is natural to ask whether one can use the preceding results to give presentations for the invariant rings $\Bbbk_{-1}[u,v]^{G_{n,k}}$. While it is possible to achieve this for specific examples by hand or by using computer assistance, we have been unable to find a way to write down a presentation for arbitrary $n$ and $k$. 
We give an example to highlight some of the difficulties, even for small $n$ and $k$.

\begin{example}
Suppose that $n = 7$ and $k=3$. The Hirzebruch-Jung continued fraction expansion of $\frac{n}{\frac{1}{2}(n+k)}$ is
\begin{align*}
\frac{7}{5} = [2,2,3].
\end{align*}
By Theorem \ref{ngtkgens}, $A^G$ is generated by
\begin{align*}
a \coloneqq (u^7 - v^7)(uv)^4, \quad b \coloneqq (u^7 + v^7)(uv), \quad c \coloneqq u^{21}-v^{21}, \quad d \coloneqq (uv)^6.
\end{align*}
One can verify that $a,b,c,d$ satisfy the relations
\begin{gather*}
ba + ab + 4d^2,  \quad ca + ac - 2b^4 - 4d^3,  \quad cb + bc - 2b^2d, \\
da-ad, \quad  db-bd,  \quad dc-cd, \\
a^2 + b^2d, \quad ab^2 + cd + bd^2,  \quad ac + abd - b^4.
\end{gather*}
Therefore, if we define
\begin{align*}
B \coloneqq 
\frac{\Bbbk \langle a,b,c,d \rangle}{\left \langle
\begin{array}{ccc}
ba + ab + 4d^2, & ca + ac - 2b^4 - 4d^3, & cb + bc - 2b^2d \\
da-ad, & db-bd, & dc-cd
\end{array}
\right \rangle },
\end{align*}
and let $I$ be the following two-sided ideal of $B$,
\begin{align*}
I \coloneqq \langle a^2 + b^2d, \hspace{5pt} ab^2 + cd + bd^2, \hspace{5pt} ac + abd - b^4 \rangle,
\end{align*}
then there is a surjection of graded algebras
\begin{align*}
\phi : B/I \twoheadrightarrow A^G.
\end{align*}
\indent We first claim that $B$ is AS regular. To show this, we first determine the Hilbert series of $B$. Ordering the words in $\Bbbk \langle a,b,c,d \rangle$ using degree lexicographic order with $a < b < c < d$, it is straightforward to check that the overlap relations between the leading terms of the relations defining $B$ resolve. Therefore $B$ has a $\Bbbk$-basis given by
\begin{align*}
\{ a^i b^j c^\ell d^m \mid i,j,\ell,m \geqslant 0 \},
\end{align*}
and so 
\begin{align*}
\hilb B = \frac{1}{(1-t^{15})(1-t^9)(1-t^{21})(1-t^{12})}.
\end{align*}
In particular, $d$ is a right nonzerodivisor, and since it is central, it is a regular central element. Finally the ring
\begin{align*}
B/\langle d \rangle = 
\frac{\Bbbk \langle a,b,c \rangle}{\langle
ba + ab, \quad ca + ac - 2b^4, \quad cb + bc
\rangle }
\end{align*}
is easily seen to be AS regular of dimension 3, so $B$ is AS regular of dimension 4 by \cite[Proposition 2.1]{chirv}. \\
\indent We now claim that the map $\phi$ is an isomorphism; since we already know that this map is a surjection, it suffices to check that $B/I$ and $A^G$ have the same Hilbert series. Using Molien's Theorem, a computer calculation can be used to verify that
\begin{align*}
\hilb A^G = \frac{1-t^{30} - t^{33} - t^{36} + t^{48} + t^{51}}{(1-t^{15})(1-t^{9})(1-t^{21})(1-t^{12})}.
\end{align*}
On the other hand, to determine the Hilbert series of $B/I$, one can check that the following sequence of graded right $B$-modules is exact,
\begin{align*}
0 \to B[-48] \oplus B[-51] \xrightarrow{\psi} B[-30] \oplus B[-33] \oplus B[-36] \xrightarrow{\eta} B \to B/I \to 0,
\end{align*}
where the maps $\eta$ and $\psi$ are given by left multiplication by the following matrices:
\begin{align*}
\eta = \begin{pmatrix}
a^2 + b^2d & ab^2 + cd + bd^2 & ac + abd - b^4
\end{pmatrix}, \qquad 
\psi = \begin{pmatrix}
-b^2 & bd+c \\ 
a & -b^2 \\
d & a
\end{pmatrix}.
\end{align*}
This implies that
\begin{align*}
\hilb B/I = (1-t^{30} - t^{33} - t^{36} + t^{48} + t^{51}) \hilb B = \frac{1-t^{30} - t^{33} - t^{36} + t^{48} + t^{51}}{(1-t^{15})(1-t^{9})(1-t^{21})(1-t^{12})},
\end{align*}
and hence that we have an isomorphism
\begin{align*}
B/I \cong A^G.
\end{align*}
\end{example}

In the above example, in particular we saw that $A^G$ was a factor of an AS-regular algebra. We expect that this is always the case.

\subsection{The case $n<k$}
Now suppose that $n < k$. As with the case when $n > k$, write \begin{align*}
\frac{n}{\frac{1}{2}(n+k)} = [\gamma_1, \dots, \gamma_{d-1}]
\end{align*}
Since $n < k$, we note that $\gamma_1 = 1$. Also define
\begin{align*}
\beta_i = \left\{
\begin{array}{cc}
\gamma_2 + 2 & \text{if } i = 2, \\
\gamma_i & \text{if } i \neq 2.
\end{array}  \right.
\end{align*}
Define three series of integers $r_1, \dots, r_{d}$, $s_1, \dots, s_{d}$, and $t_1, \dots, t_{d}$  as follows:
\begin{gather*}
\begin{array}{lll ll}
s_1 = 1 & s_2 = 1 & & & s_i = \beta_{i-1} s_{i-1} - s_{i-2} \text{ for } 3 \leqslant i \leqslant d, \\[2pt]
t_1 = 3 & t_2 = 1 & & & t_i = \beta_{i-1} t_{i-1} - t_{i-2} \text{ for } 3 \leqslant i \leqslant d, \\
\end{array} \\
r_i = \tfrac{1}{2} (kt_i - ns_i) \text{ for } 1 \leqslant i \leqslant d.
\end{gather*}

\noindent In particular, note that the $r_i$ obey the same recurrence relation as the $s_i$ and $t_i$. \\
\indent At this point, the proof proceeds in a similar fashion to the case when $n > k$. We now state analogues of the results for that case, only providing proofs when the details are sufficiently different.

\begin{lem}
We have
\begin{align*}
\frac{\frac{1}{2}(3k-n)}{\frac{1}{2}(k-n)} = [\beta_2, \beta_3, \dots, \beta_{d-1}].
\end{align*}
\end{lem}

\begin{lem} The $r_i$, $s_i$, and $t_i$ have the following properties:
\begin{enumerate}[{\normalfont (1)},leftmargin=*,topsep=0pt,itemsep=2pt]
\item $r_1 = \frac{1}{2}(3k-n)$, $r_2 = \frac{1}{2}(k-n)$, $r_{d-2} = 1$ and $r_{d} = 0$. Moreover, $r_1 > r_2 > \dots > r_{d-1} > r_{d}$.
\item $r_{i} s_{i+1} - r_{i+1} s_{i} = k$ and $r_{i} t_{i+1} - r_{i+1} t_{i} = k$ for $1 \leqslant i \leqslant d-1$.
\item $s_{d} = k$ and $t_{d} = n$.
\end{enumerate}
\end{lem}

The statement of Lemma \ref{tripleslemma} also holds when $n < k$ (and this can be proved using an analogue of Lemma \ref{technicallemma2}) which allows us to write down generators for $A^G$ in this case:

\begin{thm} \label{nltkgens}
Suppose $ n < k$. Then the elements
\begin{align*}
x_i = \big(u^{ns_i} + (-1)^{r_i + ns_i} v^{ns_i}\big) (uv)^{r_i}, \quad 1 \leqslant i \leqslant d
\end{align*}
generate $A^G$.
\end{thm}
\begin{proof}
The only difference between the proof of this theorem and that of Theorem \ref{ngtkgens} is that we must first show that the claimed generators generate $(uv)^{2k}$, before proceeding as we did before. Since $s_1 = 1 = s_2$, and $r_1 = \frac{1}{2}(3k-n$) and $r_2 = \frac{1}{2}(k-n)$, we have
\begin{align*}
x_1 = (u^n + (-1)^{\frac{1}{2}(3k+n)} v^n)(uv)^{\frac{1}{2}(3k-n)} \quad \text{and} \quad x_2 = (u^n + (-1)^{\frac{1}{2}(k+n)} v^n)(uv)^{\frac{1}{2}(k-n)}.
\end{align*}
Direct calculation gives
\begin{gather*}
x_1 x_2 = (-1)^{\frac{1}{2}(3k-n)}\big( u^{2n} + (-1)^{\frac{1}{2}(k+n)} 2u^n v^n - v^{2n}  \big)(uv)^{2k-n} \\
x_2 x_1 = (-1)^{\frac{1}{2}(k-n)}\big( u^{2n} + (-1)^{\frac{1}{2}(3k+n)} 2u^n v^n - v^{2n}  \big)(uv)^{2k-n}.
\end{gather*}
Since $\frac{1}{2}(3k-n)$ and $\frac{1}{2}(k-n)$ have different parities, we have
\begin{align*}
x_1 x_2 + x_2 x_1 = 4 u^n v^n (uv)^{2k-n} = (-1)^{\frac{n-1}{2}} 4 (uv)^{2k},
\end{align*}
and so the $x_i$ generate $(uv)^{2k}$. The remainder of the proof is similar to that of Theorem \ref{ngtkgens}.
\end{proof}

\noindent \textbf{Acknowledgements.} The author is a postdoctoral fellow at the University of Waterloo and is supported by the NSERC. The author is grateful to Jason Bell and Michael Wemyss for helpful discussions, and to Susan J. Sierra for suggesting that the algebras in Section \ref{jordanplaneinvsec} might be factors of the algebras studied in \cite{lecoutre}.

\bibliographystyle{amsplain}
\bibliography{bibliography}

\providecommand{\bysame}{\leavevmode\hbox to3em{\hrulefill}\thinspace}
\providecommand{\MR}{\relax\ifhmode\unskip\space\fi MR }
\providecommand{\MRhref}[2]{%
  \href{http://www.ams.org/mathscinet-getitem?mr=#1}{#2}
}
\providecommand{\href}[2]{#2}
\begin{thebibliography}{10}

\bibitem{bhz2}
Y.-H. Bao, J.-W. He, and J.~J. Zhang, \emph{Noncommutative {A}uslander
  theorem}, Transactions of the American Mathematical Society \textbf{370}
  (2018), no.~12, 8613--8638.

\bibitem{bhz}
\bysame, \emph{Pertinency of {H}opf actions and quotient categories of
  {C}ohen-{M}acaulay algebras}, Journal of Noncommutative Geometry \textbf{13}
  (2019), no.~2, 667--710.

\bibitem{behnke}
K.~Behnke and O.~Riemenschneider, \emph{Quotient surface singularities and
  their deformations}, Singularity theory (1995), 1--54.

\bibitem{brie}
E.~Brieskorn, \emph{Rationale singularit{\"a}ten komplexer fl{\"a}chen},
  Inventiones mathematicae \textbf{4} (1968), no.~5, 336--358.

\bibitem{ckwz2}
K.~Chan, E.~Kirkman, C.~Walton, and J.~J. Zhang, \emph{Quantum binary
  polyhedral groups and their actions on quantum planes}, Journal f{\"u}r die
  reine und angewandte Mathematik (Crelles Journal) \textbf{2016} (2016),
  no.~719, 211--252.

\bibitem{ckwz}
\bysame, \emph{Mc{K}ay correspondence for semisimple {H}opf actions on regular
  graded algebras, {I}}, Journal of Algebra \textbf{508} (2018), 512--538.

\bibitem{ckwz3}
\bysame, \emph{Mc{K}ay correspondence for semisimple {H}opf actions on regular
  graded algebras, {II}}, Journal of Noncommutative Geometry \textbf{13}
  (2019), no.~1, 87--114.

\bibitem{chirv}
A.~Chirvasitu, R.~Kanda, and S.~P. Smith, \emph{New {A}rtin-{S}chelter regular
  and {C}alabi-{Y}au algebras via normal extensions}, Transactions of the
  American Mathematical Society \textbf{372} (2019), no.~6, 3947--3983.

\bibitem{won}
J.~Gaddis, E.~Kirkman, W.~Moore, and R.~Won, \emph{Auslander’s {T}heorem for
  permutation actions on noncommutative algebras}, Proceedings of the American
  Mathematical Society \textbf{147} (2019), no.~5, 1881--1896.

\bibitem{iyudu}
N.~K. Iyudu, \emph{Representation spaces of the {J}ordan plane}, Communications
  in Algebra \textbf{42} (2014), no.~8, 3507--3540.

\bibitem{jing}
N.~Jing and J.~J. Zhang, \emph{On the trace of graded automorphisms}, Journal
  of Algebra \textbf{189} (1997), no.~2, 353--376.

\bibitem{jorgensen}
P.~J{\o}rgensen and J.~J. Zhang, \emph{Gourmet's guide to {G}orensteinness},
  Advances in Mathematics \textbf{151} (2000), no.~2, 313--345.

\bibitem{kkzGor}
E.~Kirkman, J.~Kuzmanovich, and J.~J. Zhang, \emph{Gorenstein subrings of
  invariants under {H}opf algebra actions}, Journal of Algebra \textbf{322}
  (2009), no.~10, 3640--3669.

\bibitem{stc}
\bysame, \emph{Shephard--{T}odd--{C}hevalley theorem for skew polynomial
  rings}, Algebras and representation theory \textbf{13} (2010), no.~2,
  127--158.

\bibitem{downup}
\bysame, \emph{Invariant theory of finite group actions on down-up algebras},
  Transformation Groups \textbf{20} (2015), no.~1, 113--165.

\bibitem{lecoutre}
C.~Lecoutre and S.~J. Sierra, \emph{A new family of {P}oisson algebras and
  their deformations}, Nagoya Mathematical Journal \textbf{233} (2019), 32--86.

\bibitem{ample}
I.~Mori and K.~Ueyama, \emph{Ample group action on {AS}-regular algebras and
  noncommutative graded isolated singularities}, Transactions of the American
  Mathematical Society \textbf{368} (2016), no.~10, 7359--7383.

\bibitem{riemen}
O.~Riemenschneider, \emph{Die {I}nvarianten der endlichen {U}ntergruppen von
  {GL}(2,{C})}, Mathematische Zeitschrift \textbf{153} (1977), no.~1, 37--50.

\bibitem{ueyama2013}
K.~Ueyama, \emph{Graded maximal {C}ohen--{M}acaulay modules over noncommutative
  graded {G}orenstein isolated singularities}, Journal of Algebra \textbf{383}
  (2013), 85--103.

\end{thebibliography}
\end{document}